\documentclass[11pt]{amsart}

\usepackage[utf8]{inputenc}

\usepackage[margin=1in]{geometry}
\usepackage{amsmath,amssymb}

\usepackage{hyperref}

\usepackage{bbm}
\usepackage[T1]{fontenc}
\usepackage{enumitem}
\usepackage{mathabx}
\usepackage[dvipsnames]{xcolor}
\usepackage{mathtools}
\mathtoolsset{showonlyrefs}

\theoremstyle{plain}
\newtheorem{thm}{Theorem}[section]
\newtheorem{cor}[thm]{Corollary}
\newtheorem{lemma}[thm]{Lemma}
\newtheorem{prop}[thm]{Proposition}

\theoremstyle{definition}
\newtheorem{defi}[thm]{Definition}

\newtheorem{rem}[thm]{Remark}

\numberwithin{equation}{section}

\usepackage{amsmath}
\allowdisplaybreaks[1]

\newcommand{\udc}{\check{u}_{\delta}}

\newcommand{\uderh}{\hat{u}_{\delta,\epsilon,R}}
\newcommand{\uder}{u_{\delta,\epsilon,R}}
\newcommand{\Fde}{F_{\delta,\epsilon}}

\DeclareMathOperator{\vp}{\varphi}
\DeclareMathOperator{\PP}{\mathbb{P}}
\DeclareMathOperator{\EE}{\mathbb{E}}
\DeclareMathOperator{\NN}{\mathbb{N}}

\DeclareMathOperator{\MM}{\mathcal{M}}

\DeclareMathOperator{\sgn}{sgn}
\DeclareMathOperator{\TT}{\mathbb{T}}
\DeclareMathOperator{\RR}{\mathbb{R}}
\DeclareMathOperator{\ZZ}{\mathbb{Z}}

\DeclareMathOperator{\X}{\mathcal{X}}

\newcommand{\WW}{\mathcal{W}}
\newcommand{\cO}{\color{black}}
\definecolor{OliveGreen}{gray}{0}

\title{Solutions to the stochastic thin-film equation for the range of mobility exponents $n\in (2,3)$}
\author{Max Sauerbrey\textsuperscript{1}}
\thanks{\textsuperscript{1}Max Planck Institute for Mathematics in the Sciences,	Inselstr.\!\!\! 22,	04103 Leipzig,	Germany  (\href{mailto:maxsauerbrey97@gmail.com}{maxsauerbrey97@gmail.com}). This research was conducted while the author was employed at Delft University of Technology, 2628 CD Delft, Netherlands}
\renewcommand\rightmark{Solutions to the STFE for the range of mobility exponents $n\in (2,3)$}

\begin{document}
	\maketitle
	\begin{abstract}
		Recently, many existence results for the stochastic thin-film equation were established in the case of a quadratic mobility exponent $n=2$, in which the noise term $\partial_x(u^\frac{n}{2}\mathcal{W})$ becomes linear. In the case of a non-quadratic mobility exponent, results are only available in the situation that $n\ge \frac{8}{3}$ leaving the interval of mobility exponents $n\in (2,\frac{8}{3})$ untreated. In this article we resolve the current gap in the literature by presenting a proof, which works under the assumption $n\in (2,3)${\cO, i.e., the regime of weak slippage}. The key idea is to use that the $\log$-entropy dissipation coincides with the energy production due to the noise. To realize this idea, we approximate the stochastic thin-film equation by stochastic thin-film equations with inhomogeneous mobility functions{\color{OliveGreen}, which behave like a higher power near $0$}. As a consequence the approximate solutions are non-negative, which is vital to use the $\log$-entropy estimate.
	\end{abstract}

	\tableofcontents
	{\cO 
	\section{Introduction and statement of the results}
	The} question of existence of martingale solutions to the stochastic thin-film equation
	\begin{equation}\label{Eq100}
		\partial_t u \,=\,  -\partial_x (u^n \partial_x^3u )\,+\, \partial_x(u^\frac{n}{2}\WW)
	\end{equation}
	has attracted much interest during the last years. A solution $u$ to this stochastic partial differential equation describes the evolution of the height of a thin liquid film driven by surface tension and thermal noise.  The appearing parameter $n\in \RR_{>0}$ is called mobility exponent and the corresponding nonlinearity $u^n$ the mobility function of \eqref{Eq100}. {\cO It reflects the boundary condition imposed on the fluid velocity near the substrate, where, e.g., $n=3$ corresponds to a no-slip condition and $n=2$ to a Navier slip condition. The range $n\in (2,3)$ arises from a variable slip length, which decays as the film height approaches $0$ and is called the regime of weak slippage.} We remark that a physically meaningful solution to \eqref{Eq100} is non-negative, {\cO so} that there is no ambiguity in the interpretation of these nonlinearities. The variable $\WW$ denotes a {\cO Gaussian }random field, which is white in time. 
	
	A derivation of \eqref{Eq100} was given in \cite{DMES2005} using the fluctuation dissipation relation. Similarly, in \cite{GruenMeckeRauscher2006} a lubrication approximation was used to derive 
	\begin{equation}\label{Eq101}
		\partial_t u \,=\,  -\partial_x (u^n \partial_x(\partial_x^2 u - g'(u) ) )\,+\, \partial_x(u^\frac{n}{2}\WW)
	\end{equation}
	with spatio-temporal white noise $\WW$. We remark that the SPDE \eqref{Eq101} differs from
	\eqref{Eq100}  by the additional term 
	\[\partial_x (u^n\partial_x (g'(u))),\] 
	which accounts for interaction forces between the molecules of the fluid and the substrate. In these works numerical simulations were used to show that the presence of the noise term leads to faster spreading behavior and accelerated film rupture, compared to the corresponding deterministic equations. In particular, as pointed out in \cite{GruenMeckeRauscher2006}, the latter could be a possible explanation for discrepancies between simulations of the deterministic equation and the
	physical experiments using polystyrene films of height $\approx 4nm$ observed in \cite{becker2003complex}. We point out that precisely this regime of very small film heights is the one, where thermal fluctuations start to become relevant. For  information on the underlying gradient flow structure and  thermodynamically consistent simulations  of \eqref{Eq100} we refer to \cite{GGKO21}. Moreover, a derivation of an appropriate counterterm and corresponding estimates towards a {\color{OliveGreen}treatment} of \eqref{Eq100} in the framework of diagram-free regularity structures  can be found in  \cite{gvalani2023stochastic}. We remark that while in the previously mentioned works on  stochastic thin films the random field $\mathcal{W}$ is white in space and time, the following {\cO preprints and} articles, concerned with the problem of existence {\cO of martingale solutions to}  \eqref{Eq100} and \eqref{Eq101}, all impose the condition that $\WW$ is spatially correlated.

	In the pioneering work \cite{fischer_gruen_2018} existence of martingale solutions to the It\^o interpretation of \eqref{Eq101}	
	was shown.  In particular, in \cite{fischer_gruen_2018} the corresponding interface potential $g$ is assumed to be sufficiently singular at $0$ to enforce $u$ to stay away from $0$. Then, in \cite{GessGann2020} existence of martingale solutions to the Stratonovich interpretation of \eqref{Eq100} in the quadratic case $n=2$ was shown. The main advance in this work is the observation that the Stratonovich SPDE
	\[
		\partial_t u \,=\,  \partial_x(u \WW)
	\]
	can be solved independently, which allows for a Trotter-Kato type decomposition of the deterministic and stochastic dynamics of \eqref{Eq100}. This approach allows in particular for initial values which are $0$ on sets of positive measure and thus the presence of a contact line, i.e. the triple junction of air, liquid and solid. Subsequently, in \cite{dareiotis2021nonnegative} the question of existence to \eqref{Eq100}  for non-quadratic mobility exponents was addressed and could be answered positively for $n\in [\frac{8}{3},4)$. Although the Stratonovich interpretation is used there too, the result in \cite{dareiotis2021nonnegative} requires the initial  value to be positive almost everywhere. 
	The reason for this is that the presence of the nonlinear noise term in \eqref{Eq100} for $n=2$ requires some control on the smallness of $u$ for all times to be estimated.
	
	Afterwards, the question of existence to the physically relevant, two-dimensional analogues of \eqref{Eq100} and \eqref{Eq101}  were addressed in \cite{metzger2022existence} and \cite{Sauerbrey_2021}, both in the case of a quadratic mobility exponent $n=2$. These works can be seen as extensions of \cite{fischer_gruen_2018} and \cite{GessGann2020}, respectively, to the two-dimensional setting. 
	Moreover, also for $n=2$, it was shown in \cite{KleinGruen22} that solutions to the Stratonovich interpretation of \eqref{Eq100} can be obtained as a limit of solutions to \eqref{Eq101} with vanishing  interface potential $g$. Compared to \cite{GessGann2020}, the authors also prove additional {\cO $\alpha$-entropy }estimates, which imply that the solutions admit a $0$-contact angle for all times at the triple junction of air, liquid and gas. The question of existence to \eqref{Eq100} for a non-quadratic mobility exponent and a non-fully supported initial value was answered positively in \cite{dareiotis2023solutions}, at least {\cO for} $n\in [\frac{8}{3},3)$. The result comes at the expense of passing to a weaker notion of solutions compared to the previous results in one dimension. Most recently, existence of solutions to a lower order perturbation of \eqref{Eq100} with $n=2$ was proved in \cite{kapustyan2023film} using a time splitting scheme as in \cite{GessGann2020}.
	
	This review of the current existence literature on \eqref{Eq100} and \eqref{Eq101} shows that, while positive results could be obtained for $n=2$ and particular $n\ge \frac{8}{3}$, the situation for $n\in (2,\frac{8}{3})$ is left open. In the current article we alleviate this gap in the  literature and provide an existence proof of martingale solutions to \eqref{Eq100}, which applies for $n\in (2,3)$.
	
	{\cO \subsection{Relevant a-priori estimates}\label{SS_apriori}
	The discussed existence results concerning martingale solutions of \eqref{Eq100} as well as the results obtained in the current article require uniform estimates on suitable regularized versions of  \eqref{Eq100}. In this subsection, we recall the relevant a-priori estimates, which hold for strictly positive and smooth solutions to the deterministic thin-film equation
	\begin{equation}\label{Eq103a}
		\partial_t u \,=\,  -\partial_x (u^n \partial_x^3u ).
	\end{equation}
	By integrating \eqref{Eq103a} against $-\partial_x^2 u$, one obtains the \emph{energy estimate}
	\begin{equation}\label{Eq105a}
		\frac{1}{2}\partial_t \|\partial_x u \|_{L_x^2}^2\,=\, -\|u^\frac{n}{2}\partial_x^3 u\|_{L_x^2}^2,
	\end{equation}
	while integration against $\int^u r^{-n} \, dr$ yields the \emph{entropy estimate}
	\begin{equation}\label{Eq11}
		\partial_t \int \biggl(\int^u \int^{r} (r')^{-n} \, dr' \, dr\biggr) \, dx \,=\, -\int (\partial_x^2 u )^2 \, dx.
	\end{equation}
	Replacing the exponent $-n$ by $\alpha-1$ results for $\alpha \in [1/2-n, 2-n]$ in the \emph{$\alpha$-entropy estimate}
	\begin{equation}\label{Eq_AE_Est}
		\partial_t \int \biggl(\int^u \int^{r} (r')^{\alpha-1} \, dr' \, dr\biggr) \, dx \,=\, -\frac{1}{\gamma^2} \int u^{\alpha+n-2\gamma+1} (\partial_x^2 u^\gamma )^2 \, dx \,-\, c(\alpha+n,\gamma)\int u^{\alpha+n-3} (\partial_x u)^4 \, dx,
	\end{equation}
	where $\gamma>0$ can be chosen such that $c(\alpha+n,\gamma)\ge 0$ and even $c(\alpha+n,\gamma)> 0$ whenever $\alpha \in (1/2-n, 2-n)$,  see \cite[Proposition 2.1]{Beretta_Bertsch_DalPasso_95}.
	
	All  these estimates will play a role in our approximation procedure for  \eqref{Eq100}. As observed in \cite{dareiotis2021nonnegative,dareiotis2023solutions} the noise term \eqref{Eq100}, even when interpreted in Stratonovich form, leads to difficulties in closing stochastic versions of these a-priori estimates. Specifically, deriving a version for  the energy estimate in the nonlinear noise case $n\ne 2$ is challenging. We achieve this in the current manuscript using the $\alpha$-entropy estimate with $\alpha=-1$, to which we refer to as \emph{$\log$-entropy estimate}. 
	}
	\subsection{Result}
	{\cO We proceed to formulate the main result of this article. Namely, }
	 we consider the Stratonovich interpretation of \eqref{Eq100} on the one-dimensional torus $\TT$ with unit length.
	To state our our assumption on the spatial smoothness of the noise {\cO we follow} the notation of \cite{dareiotis2021nonnegative} and write 
	\begin{equation}
		e_k(x)\,=\, \sqrt{2}\begin{cases}
			\cos(2\pi kx), & k \ge 1 ,\\
			\frac{1}{\sqrt{2}}, & k=0,\\
			\sin(2\pi kx), & k \le -1,
		\end{cases}\qquad k\in \ZZ,
	\end{equation}
	for the eigenfunctions of the periodic Laplace operator and define $\sigma_k = \lambda_k e_k$ for a sequence $(\lambda_k)_{k\in \ZZ}$. We let $(\beta^{(k)})_{k\in \ZZ}$ be a family of independent Brownian motions and define the Wiener process
	\[
	B(t,x)\,=\, \sum_{k\in \ZZ} \sigma_k(x)\beta^{(k)}(t).
	\]
	We insert $\mathcal{W}=\frac{dB}{dt}$ in \eqref{Eq100}, set $F_0(r)= r^\frac{n}{2}$ and, as in \cite[Eq. (2.4)]{dareiotis2021nonnegative}, obtain the following It\^o formulation of the Stratonovich interpretation of \eqref{Eq100}:
	\begin{equation}\label{Eq102}
		du\,=\, -\partial_x (F_0^2(u) \partial_x^3 u)\,dt
		\,+\, \frac{1}{2}\sum_{k\in \ZZ} \partial_x (\sigma_k F_0'(u)\partial_x (\sigma_k F_0(u)))\, dt\,+\, \sum_{k\in \ZZ} \, \partial_x (\sigma_k F_0(u)\,)d\beta^{(k)}.
	\end{equation}
	We use the notion of weak martingale solutions to \eqref{Eq102} from \cite[Definition 2.1]{dareiotis2021nonnegative}. We point out that, in contrast to \cite[Definition 2.1]{dareiotis2021nonnegative}, we restrict our definition to the relevant case of non-negative solutions.
	\begin{defi}\label{defi_sol}
		A weak martingale solution to \eqref{Eq102} with initial value $u_0\in L^2(\Omega,\mathfrak{F}_0{\cO ;} H^1(\TT))$ consists  of a probability space $(\tilde{\Omega}, \tilde{\mathfrak{A}}, \tilde{\PP})$ with a filtration $\tilde{\mathfrak{F}}$ satisfying the usual conditions, a family $(\tilde{\beta}^{(k)})_{k\in \ZZ}$ of independent $\tilde{\mathfrak{F}}$-Brownian motions and a weakly continuous, $\tilde{\mathfrak{F}}$-adapted, non-negative, $H^1(\TT)$-valued process $\tilde{u}$ such that
		\begin{enumerate}[label=(\roman*)]
			\item $\tilde{u}(0)$ has the same distribution as $u_0$,
			\item $\tilde{\EE}[\sup_{t\in [0,T]} \|\tilde{u}(t)\|_{H^1(\TT)}^2]<\infty$,
			\item for $\tilde{\PP}\otimes dt$-almost all $(\tilde{\omega},t)\in \tilde{\Omega}\times [0,T]$ the weak derivative of third order $\partial_x^3\tilde{u}$ exists on $\{\tilde{u}> 0\}$ and satisfies $\tilde{\EE}[\|\mathbbm{1}_{\{\tilde{u}>0\}} F_0(\tilde{u}) \partial_x^3\tilde{u}\|_{L^2([0,T]\times \TT)}^2]<\infty$,
			\item for all $\vp\in C^\infty(\TT)$, $\tilde{\PP}$-almost surely, we have 
			\begin{align}\begin{split}\label{Eq136}
					&
				(\tilde{u}(t),\vp)_{L^2(\TT)}\,=\, 
				(\tilde{u}(0),\vp)_{L^2(\TT)}\,+\,\int_0^t \int_{\{\tilde{u}(s)> 0\}} 
				F_0^2(\tilde{u})\partial_x^3 \tilde{u} \,\partial_x\! \vp
				\, dx\, ds 
				\\&\qquad -\,\frac{1}{2} \sum_{k\in \ZZ}\int_0^t (\sigma_kF_0'(\tilde{u}) \partial_x (\sigma_k F_0(\tilde{u})), \partial_x\! \vp )_{L^2(\TT)}\, ds\,-\, \sum_{k\in \ZZ} \int_0^t (\sigma_k F_0(\tilde{u}), \partial_x\! \vp)_{L^2(\TT)} \, d{\cO\tilde{\beta}^{(k)}},
				\end{split}
			\end{align}
		  	{\cO for all $t\in [0,T]$.}
		\end{enumerate}
	\end{defi}
	
	The main result of this article reads as follows. We remark that for $n\in [\frac{8}{3},3)$, a comparable existence result is already available, see \cite[Theorem 2.2]{dareiotis2021nonnegative}.
	
	\begin{thm}\label{Thm_Existence} Let $T\in (0,\infty)$, $n\in (2,3)$, $p>n+2$ and $u_0\in L^p(\Omega, \mathfrak{F}_0{\cO;} H^1(\TT))$ non-negative with
		\begin{equation}\label{Eq124}
				{\color{OliveGreen}\EE\biggl[
			\biggl|\int_{\TT} u_0 \, dx\biggr|^{\frac{p(n+2)}{8-2n}}\biggr]}\,+\, \EE\biggl[\biggl(
			\int_{\TT}
			(u_0-1)-\log(u_0)
			\, dx
			\biggr)^\frac{p}{2}
			\biggr]\,<\, \infty
		\end{equation}
		and moreover 
		\begin{equation}\label{Eq120}
			\sum_{k\in \ZZ} \lambda_k^2k^4\,<\, \infty.
		\end{equation}
	Then \eqref{Eq102} admits a weak martingale solution
		\[
	\bigl\{
	(\tilde{\Omega}, \tilde{\mathfrak{A}},\tilde{\mathfrak{F}},\check{\PP} ), \,(\tilde{\beta}^{(k)})_{k\in \ZZ},\, \tilde{u}
	\bigr\}
	\]
	in the sense of Definition \ref{defi_sol} with initial value $u_0$ satisfying $\tilde{u}>0${\cO,} $\tilde{\PP}\otimes dt\otimes dx$-almost everywhere. Moreover, this solution satisfies the estimate
		\begin{align}\begin{split} \label{Eq137}
			&
			\tilde{\EE}	\biggl[
			\sup_{t\in [0,T]} \|\partial_x \tilde{u}(t)\|_{L^2(\TT)}^p \,+\,
			\sup_{t\in[0,T]} \biggl(\int_{\TT} (\tilde{u}(t)-1) -\log(\tilde{u}(t))\, dx\biggr)^\frac{p}{2}\biggr]
			\\&\qquad +\, \tilde{\EE}	\Bigl[\|\mathbbm{1}_{\{\tilde{u}>0\}} F_0(\tilde{u}) \partial_x^3\tilde{u}\|_{L^2([0,T]\times \TT)}^p
			\Bigr]
			\\& \quad \lesssim_{n, \sigma, p,T}
			\, {\EE}\biggl[
			\|\partial_x u_0\|_{L^2(\TT)}^p
			\,
			+\, {\color{OliveGreen}\biggl|\int_{\TT} u_0 \, dx\biggr|^\frac{p(n+2)}{8-2n}}
			\, +\,\biggl(
			\int_{\TT} (u_0-1)-\log(u_0)\, dx\biggr)^\frac{p}{2}\biggr]
			\,+\, 1
		\end{split}
	\end{align}
	and 
	\begin{equation}\label{Eq144}
		\tilde{u}\in L^{q}(\tilde{\Omega},C^{\frac{\gamma}{4},\gamma}([0,T]\times \TT) )
		\quad \text{for all}\quad \gamma\in \bigl(0,\tfrac{1}{2}\bigr) \quad \text{and}\quad q\in \bigl[1,\tfrac{2p}{n+2}\bigr).
	\end{equation}
	\end{thm}{\cO
	The estimate \eqref{Eq137}  can be seen as a combined energy-$\log$-entropy estimate for the solution, cf. Subsection \ref{SS_apriori}.
	Additionally}, in  case that the initial value has finite entropy, we also recover an entropy estimate in the spirit of \cite[Eq. (2.6)]{dareiotis2021nonnegative}.
	
	\begin{prop}\label{Prop_Entr_Est}
		Under the assumptions of Theorem \ref{Thm_Existence}, the constructed solution $\tilde{u}$ satisfies
		\begin{align}\begin{split}\label{Eq162}
				&
				\tilde{\EE}\Bigl[
				\sup_{t\in [0,T]} \bigl\|(\tilde{u}(t))^{2-n}\bigr\|_{L^1(\TT)}^q\,+\, 
				\|\partial_x^2 \tilde{u}\|_{L^2([0,T]\times \TT)}^{2q}
				\Bigr]\\&\quad \lesssim_{n,q,\sigma, T}\, {\EE}{\color{OliveGreen}\biggl[
				\bigl\|u_0^{2-n}\bigr\|_{L^1(\TT)}^q
				\,+\,  \biggl|\int_{\TT} u_0 \, dx\biggr|^{2q}
				\biggr]}\,+\, 1,
			\end{split}
		\end{align}
		 for each $q\ge 1$.
	\end{prop}
	
	\subsection{Extension of a previous existence result}
	
	{\cO Notably, Theorem \ref{Thm_Existence} as well as the aforementioned result \cite[Theorem 2.2]{dareiotis2021nonnegative} require that the initial value of the equation is positive almost everywhere on $\TT$.  At the expense of allowing for low regularity solutions in the sense of the following Definition \ref{defi_very_weak_sol}, existence of solutions to \eqref{Eq102} for non-fully supported initial data was shown in \cite[Theorem 1.6]{dareiotis2023solutions} under the assumption that $n\in [8/3,3)$. That the lower bound $n\ge 8/3$ coincides with the lower bound of mobility exponents from  \cite{dareiotis2021nonnegative} is no coincidence, because the result \cite[Theorem 2.2]{dareiotis2021nonnegative} is used in \cite{dareiotis2023solutions} to ensure existence of solutions when shifting the initial value away from $0$.
	Accordingly, it is mentioned in \cite[Consequence 2.1]{dareiotis2023solutions} that an extension of \cite[Theorem 2.2]{dareiotis2021nonnegative} to the range $n\in (2,8/3)$ would also yield an extension of \cite[Theorem 1.6]{dareiotis2023solutions} to said range of exponents.  As our main result Theorem \ref{Thm_Existence} together with Proposition \ref{Prop_Entr_Est} indeed extend \cite[Theorem 2.2]{dareiotis2021nonnegative} to this range, we conclude that \cite[Theorem 1.6]{dareiotis2023solutions} also holds for $n\in (2,8/3)$.} 
	
	More precisely, {\cO for the proof of \cite[Theorem 1.6]{dareiotis2023solutions}} it is needed that for {\cO every} 
	$u_0\in L^\infty(\Omega, \mathfrak{F}_0{\cO;} H^1(\TT))$ with $u_0>c>0$ there exists a weak martingale solution
	\[
	\bigl\{
	(\tilde{\Omega}, \tilde{\mathfrak{A}},\tilde{\mathfrak{F}},\check{\PP} ), \,(\tilde{\beta}^{(k)})_{k\in \ZZ},\, \tilde{u}
	\bigr\}
	\]
	to \eqref{Eq102} in the sense of Definition \ref{defi_sol} with initial value $u_0$, which is positive $\tilde{\PP}\otimes dt\otimes dx$-almost everywhere {\cO and satisfies}
	\begin{equation}\label{Eq6}
	\tilde{\EE}\biggl[
	\sup_{t\in[0,T]}\|\tilde{u}(t)\|_{H^1(\TT)}^{2n}\,+\, \|\mathbbm{1}_{\{\tilde{u}>0\}}F_0(\tilde{u})\partial_x^3\tilde{u}\|_{L^2([0,T]\times \TT)}^4\,+\, \|\tilde{u}\|_{L^2(0,T; H^2(\TT))}^4
	\biggr]\,<\, \infty,
	\end{equation}
	cf. \cite[Consequence 2.1]{dareiotis2023solutions}. {\cO The positivity of these solutions and the regularity requirement \eqref{Eq6} play a key role in  \cite{dareiotis2023solutions} to justify the derivation of suitable estimates.} By virtue of  Theorem \ref{Thm_Existence} and Proposition \ref{Prop_Entr_Est}, {\cO the existence of such solutions} is guaranteed also for $n\in (2,\frac{8}{3})$. {\cO Thus,} following \cite{dareiotis2023solutions}, one arrives at an extension of \cite[Theorem 1.6]{dareiotis2023solutions} to the whole interval $n\in (2,3)$. 
	
	To state {\cO the described extension}, we repeat the notion of solutions to \eqref{Eq102} from \cite[Definition 1.4]{dareiotis2023solutions}. Therefore, we denote the space of Radon measures on $\TT$ by $\MM(\TT)$ and equip it with the  sigma-field generated by its pre-dual space $C(\TT)$. Moreover, we denote the dual pairing between $\MM(\TT)$ and $C(\TT)$ by $\langle\cdot,\cdot\rangle$ and the set of $\MM(\TT)$-valued random variables by $ L^0(\Omega{\cO;} \MM(\TT))$.
	\begin{defi}\label{defi_very_weak_sol}
		A very weak martingale solution to \eqref{Eq102} with initial value $u_0\in L^0(\Omega, \mathfrak{F}_0{\cO;} \MM(\TT)) $ consists  of a probability space $(\tilde{\Omega}, \tilde{\mathfrak{A}},\check{\PP} )$ with a filtration $\tilde{\mathfrak{F}}$ satisfying the usual conditions, a family $(\tilde{\beta}^{(k)})_{k\in \ZZ}$ of independent $\tilde{\mathfrak{F}}$-Brownian motions and a vaguely continuous, $\tilde{\mathfrak{F}}$-adapted, non-negative, $\MM(\TT)$-valued process $\tilde{u}$ such that
		\begin{enumerate}[label=(\roman*)]
		\item $\tilde{u}(0)$ has the same distribution as $u_0$,
		\item for $\tilde{\PP}\otimes dt$-almost all $(\tilde{\omega},t)\in \tilde{\Omega}\times [0,T]$ the weak derivative  $\partial_x\tilde{u}$ exists in $L^1(\TT)$ and satisfies
		\[
		\tilde{u}^{n-2}(\partial_x\tilde{u})^3,\, 
		\tilde{u}^{n-1}(\partial_x\tilde{u})^2,\,
		\tilde{u}^{n}\partial_x\tilde{u},\,
		\tilde{u}^{n-2}\partial_x\tilde{u},\,
		\tilde{u}^{n}\,\in \, L^1([0,T]\times \TT)
		\]
		$\tilde{\PP}$-almost surely,
		\item for all $\vp\in C^\infty(\TT)${\cO, $\tilde{\PP}$-almost surely, we have} 
		\begin{align}\begin{split}
				&
				\langle \tilde{u}(t),\vp\rangle \,=\, 
				\langle \tilde{u}(0),\vp\rangle \,+\,\tfrac{n(n-1)}{2}\int_0^t \int_{\TT} 
				\tilde{u}^{n-2}(\partial_x\tilde{u})^3\partial_x \!\vp
				\,dx\, ds \,+\, \tfrac{3n}{2}\int_0^t \int_{\TT} 
				\tilde{u}^{n-1}(\partial_x\tilde{u})^2\partial_x^2\!\vp
				\,dx\, ds
				\\&\qquad +\,\int_0^t \int_{\TT} \tilde{u}^{n}\partial_x\tilde{u}\,\partial_x^3\! \vp  \,dx\, ds
				\,-\, 
				\tfrac{n}{4} \sum_{k\in \ZZ}\int_0^t\int_{\TT} \sigma_k \tilde{u}^{\frac{n}{2}-1} \partial_x (\sigma_k \tilde{u}^\frac{n}{2}) \partial_x\! \vp \, dx \, ds	\\&\qquad 
				-\, \sum_{k\in \ZZ} \int_0^t \int_{\TT} \sigma_k \tilde{u}^\frac{n}{2} \partial_x\! \vp \, dx \, d{\cO\tilde{\beta}^{(k)}},
			\end{split}
		\end{align}
	{\cO for all $t\in [0,T]$.}
	\end{enumerate}
	\end{defi}
	{\cO Finally, we state our extension of \cite[Theorem 1.6]{dareiotis2023solutions} from $n\in [\frac{8}{3},3)$ to the whole interval $(2,3)$. We remark that, in contrast to Theorem \ref{Thm_Existence}, the initial value is not required to have full support and that only $\alpha$-entropy estimates are derived for the solutions. As a consequence of the latter, less spatial regularity of the noise is assumed, cf. \eqref{Eq120}.}
	\begin{cor}Let $T\in (0,\infty)$, $n\in (2,3)$, $u_0\in L^0(\Omega, \mathfrak{F}_0{\cO;} \MM(\TT))$ non-negative and 
			\begin{equation}
			\sum_{k\in \ZZ} \lambda_k^2k^2\,<\, \infty.
		\end{equation}
		 Then \eqref{Eq102} admits a very weak martingale solution 
			\[
		\bigl\{
		(\tilde{\Omega}, \tilde{\mathfrak{A}},\tilde{\mathfrak{F}},\check{\PP} ), \,(\tilde{\beta}^{(k)})_{k\in \ZZ},\, \tilde{u}
		\bigr\}
		\]
		in the sense of Definition \ref{defi_very_weak_sol} satisfying the properties \cite[Theorem 1.6 (i)-(iv)]{dareiotis2023solutions} with initial value $u_0$.
	\end{cor}
	\subsection{Strategy of the proof}\label{Sec_strategy}
	Our main result Theorem \ref{Thm_Existence} can be seen as an extension of \cite[Theorem 2.2]{dareiotis2021nonnegative}, which is limited to $n\in [\frac{8}{3},4)$,  to the case $n\in (2,\frac{8}{3})$. Indeed, the authors of \cite{dareiotis2021nonnegative} write 'We expect that the limitations $n\ge \frac{8}{3}$ and $n < 4$ are due to technical reasons and that these restrictions can be potentially removed
	in future work by making use of so-called $\alpha$-entropies'. We employ this suggested strategy as follows: At the core of {\cO\cite{dareiotis2021nonnegative}}   are suitable approximations {\cO of\eqref{Eq102}}, which are  compatible with the energy estimate {\cO discussed in Subsection \ref{SS_apriori}.} 
	{\cO Since, as demonstrated in \cite{BERNIS1990} for  the deterministic thin-film equation,} the energy estimate \eqref{Eq105a} provides sufficient compactness to show that a limit solves \eqref{Eq103a} in a weak sense{\cO, weak martingale solutions to \eqref{Eq102} are obtained using a stochastic compactness argument}. 
	 
	For the {\cO stochastic thin-film equation} \eqref{Eq102} a formal application of It\^o's formula shows that
	\[
	{\cO\frac{1}{2}}d\| \partial_x u\|_{L_x^2}^2 \,=\, -\|u^\frac{n}{2}\partial_x^3u\|_{L_x^2}^2\,dt\,+\,dM\,+\, dR,
	\]
	where $M$ is a local martingale and $R$ is the remainder after canceling the energy production due to the noise with the energy dissipation due to the Stratonovich correction. As calculated in \cite[Eq. (4.10)]{dareiotis2021nonnegative} {\cO and elaborated in Appendix \ref{App_B}}, this remainder term takes the form
	\begin{equation}\label{Eq9}
	dR\,=\, \biggl(\dots\,+\, {\cO\tfrac{1}{2}}\int u^{n-4}(\partial_xu)^4\, dx\biggr)\, dt,
	\end{equation}
	where the left out terms are less difficult to estimate than the one made explicit. In this article, we use the observation that {\cO this term is part of the $\alpha$-entropy dissipation \eqref{Eq_AE_Est} for $\alpha=-1$, i.e. the $\log$-entropy dissipation. In particular, we confirm the conjecture  that $\alpha$-entropy estimates can be used to generalize  \cite[Theorem 2.2]{dareiotis2021nonnegative} beyond the range [8/3,4).} 

We remark that {\cO in the derivation of the $\alpha$-entropy estimates for the deterministic thin-film equation in \cite{Beretta_Bertsch_DalPasso_95} a special nonlinear regularization proposed in \cite[Section 6]{BERNIS1990} is used.} 
	This ensures that the approximations are non-negative and allows to close $\alpha$-entropy estimates on the approximate level.  We will follow this idea and first construct non-negative solutions to 
	\begin{align}\begin{split}&
			\label{Eq107}
		d u_\delta \,=\,  -\partial_x (F_\delta^2(u_\delta) \partial_x^3u_\delta )\, dt
			\\&\qquad+\, \frac{1}{2}\sum_{k\in \ZZ} \partial_x (\sigma_k F_\delta'(u_\delta)\partial_x (\sigma_k F_\delta(u_\delta)))\, dt\,+\, \sum_{k\in \ZZ} \, \partial_x (\sigma_k F_\delta(u_\delta))d\beta^{(k)},
		\end{split}
	\end{align}
	where
	\begin{equation}
		F_\delta(r)\,=\, \tfrac{r^\frac{n+\nu}{2}}{r^\frac{\nu}{2}+\delta^\frac{\nu-n}{2}r^\frac{n}{2} },
	\end{equation}
	for appropriate $\nu\in (3,4)$. In particular $F_\delta(r)\sim r^\frac{\nu}{2}$ for $r$ close to $0$, so that up to minor modifications the strategy from \cite{dareiotis2021nonnegative} can be employed. {\cO We remark that since $F_\delta(r)\sim r^\frac{n}{2}$ for $r$ large, we refer to $F_\delta^2(u)$ as inhomogeneous mobility function throughout this manuscript.} Subsequently, we use the aforementioned idea and show a {\cO stochastic} version of the $\log$-entropy estimate and then of the energy estimate, which are uniform in $\delta$. This again enables us to apply the stochastic compactness method to extract a weak martingale solution to \eqref{Eq102} as $\delta\searrow 0$. 
	
	{\cO In the course of this article, we show that this strategy is successful for $n\in (2,3)$ bridging in particular the gap of mobility exponents $(2,8/3)$ for which existence of solutions to \eqref{Eq102} was unknown. For the quadratic case $n=2$ in contrast, a stochastic version of the energy estimate can be closed on its own, even for non-fully supported initial values, due to the linearity of the noise term, see \cite{GessGann2020}.
	Unfortunately, even though the $\log$-entropy belongs to the class of admissible $\alpha$-entropy functionals for $n\in (3/2,2)$ as well, see Subsection \ref{SS_apriori}, our approach fails for subquadratic mobility exponents.	The reason for this is that the $\alpha$-entropy production by the noise contains the $(\alpha+n-1)$-th power of the solution, see \cite[Eqs. (2.15)--(2.19)]{dareiotis2023solutions}. For $\alpha=-1$, corresponding to the $\log$-entropy, and $n<2$ this power is negative and therefore 
	requires a stronger control on the smallness of the solution than can be obtained from the $\log$-entropy itself.}

	\subsection{Notation}\label{Sec_further_notation}
	For the remainder of this article we fix the finite time horizon $T\in (0,\infty)$ and a sequence $(\lambda_k)_{k\in \ZZ}$ subject to the condition \eqref{Eq120}. Moreover, we write $\TT$ for the flat torus, i.e. for the interval $[0,1]$ with its endpoints identified. 	For a function $f$, we write $f_+$ and $f_-$ for its positive and negative part, respectively. To not overload the subscript, we  may also use the notations $f^+$ and $f^-$.

	Let $\X$ be a Banach space and $\nu$ a non-negative measure on a  measurable space $S$. Then we write $L^p(S{\cO;} \mathcal{X})$, $p\in [1,\infty]$, for the Bochner space on $S$ {\color{OliveGreen}consisting of strongly measurable $\mathcal{X}$-valued functions such that
	\[
	\|f\|_{L^p(S;\mathcal{X})}^p\,=\, \int_S   \|f\|_{\mathcal{X}}^p \, d\nu\,<\, \infty
	\]
	for $ p\in [1,\infty)$ and the usual modification for $p=\infty$.}	If $[0,T]$ is an interval, we use the notation $L^p(0,T; \mathcal{X})$ for $L^p([0,T]; \X)$. In the case $\mathcal{X}=\RR$, we simply write $L^p(S)$. 
	
	If $\mathcal{T}$ is a compact topological space, we write $C(\mathcal{T}{\cO;}\mathcal{X})$ for the space of continuous,  $\X$-valued functions equipped with the norm
	\[
	\|f\|_{C(\mathcal{T};\X)}\,=\, \sup_{y\in \mathcal{T}} \|f(y)\|_{\X}.
	\]
	In the case $\mathcal{X}=\RR$, we simply write $C(\mathcal{T})$. For $\beta\in (0,1)$, we write $C^\beta([0,T];\mathcal{X})$ for the H\"older space of $\mathcal{X}$-valued functions, which carries the norm
	\[
	\|f\|_{C^\beta([0,T];\mathcal{X})}\,=\, \|f \|_{C([0,T];\mathcal{X})}\,+\, [f]_{C^\beta([0,T]; \mathcal{X})},
	\]
	where
	\[
	[f]_{C^\beta([0,T]; \mathcal{X})}\,=\, \sup_{\substack {s,t\in [0,T]\\  s\ne t}}
	\frac{\|f(t)-f(s)\|_{\mathcal{X}}}{|t-s|^\beta}.
	\]
	
	For $l\in \NN$, we denote the space of $l$-times continuously differentiable functions on $\TT$ by $C^l(\TT)$ and equip it with the norm
	\[ \|f\|_{C^l(\TT)}\,=\, \sum_{j=0}^l \|\partial_x^j f\|_{C(\TT)}.
	\]
	For the smooth functions on $\TT$ we write $C^\infty(\TT)$.
	We also write $H^l(\TT)$ for the Sobolev space of order $l$ equipped with the norm
	\[
	\|f\|_{H^l(\TT)}^2\,=\, \sum_{j=0}^l \|\partial_x^j f\|_{L^2(\TT)}^2.
	\] 
	
	For two exponents $\beta,\gamma\in (0,1)$, we write $C^{\beta,\gamma}([0,T]\times \TT)$ for the mixed exponent H\"older space carrying the norm
	\[
	\|f\|_{C^{\beta,\gamma}([0,T]\times \TT)}\,=\, \|f\|_{C([0,T]\times \TT)}\,+\, [f]_{C^{\beta,\gamma}([0,T]\times \TT)},
	\]
	where
	\[
	[f]_{C^{\beta,\gamma}([0,T]\times \TT)}\,=\,
	\sup_{x\in \TT} \sup_{\substack {s,t\in [0,T]\\  s\ne t}}
	\frac{|f(t,x)-f(s,x)|}{|t-s|^\beta}
	\,+\,  \sup_{t\in[0,T]} \sup_{\substack{x,y\in \TT \\ x\ne y}}
	\frac{|f(t,x)-f(t,y)|}{|x-y|^\gamma}.
	\]

	If $(\Omega, \mathfrak{A}, \PP)$ is a probability space, we write $\EE[\,\cdot\,]$ for the expectation.  For two quantities $A$ and $B$, we write $A\lesssim B$, if there exists a universal constant $C$ such that $A\le CB$. {\color{OliveGreen}As a consequence, the presence of a constant may be indicated only once along a series of inequalities. Indeed, the relation $A_1\lesssim B \le A_2$ expresses that $A_1 \le CB$ for a universal constant $C$ and $B\le A_2$ and therefore also $A_1\le CA_2$. If such a constant depends on parameters $p_1,\dots$, we write $A\lesssim_{p_1,\dots}B$ instead and we write $A\eqsim_{p_1,\dots} B$, whenever $A\lesssim_{p_1,\dots} B$ and $B\lesssim_{p_1,\dots}A$. \color{OliveGreen} 
		
		 Moreover, we also use throughout the manuscript the approximate functions and functionals introduced in Section \ref{Sec_approx_functions}.}
	
		\section{Organization of the manuscript}
	To implement the idea laid out in {\cO Subsection \ref{Sec_strategy}}, we carefully introduce approximate versions of the functions involved in the desired a-priori estimates {\cO in the following} Section \ref{Sec_approx_functions}. 
	{\cO Section \ref{Sec_main} contains our main line of argument and culminates in the proofs of Theorem \ref{Thm_Existence} and Proposition \ref{Prop_Entr_Est}. Some parts concerning the proof of existence of solutions to the auxiliary equation \eqref{Eq107}, which are analogous to the proof of \cite[Theorem 2.2]{dareiotis2021nonnegative} are postponed to Section \ref{Sec_remaining_proofs}. The vital but technical estimates on the approximating functions and functionals are collected in Appendix \ref{App_A}. For convenience of the reader we also elaborate the It\^o expansion of the energy of a solution to the stochastic thin-film equation from \cite[Eq. (4.10)]{dareiotis2021nonnegative} in Appendix \ref{App_B}.}

	\section{Approximate mobilities and functionals}\label{Sec_approx_functions}
	For the rest of this article we fix  $n\in (2,3)$ and a corresponding $\nu\in (3,4)$ subject to the conditions
		\begin{align}\label{Eq135}& 
		\nu\,<\, 6-n,\\&	\label{Eq106}
			\nu^2 \,+\, \nu (2-4n)\,+\, n(n+2)\,\le \, 0.
		\end{align}
	We convince ourselves that these conditions are compatible as follows. Since \eqref{Eq135} is satisfied as soon as we choose $\nu$ close to $3$, it suffices to check that
	\[
	15\,-\, 10n \,+\, n^2\,=\, 3^2\,+\, 3(2-4n)\,+\, n(n+2)\,<\, 0
	\]
	for all $n\in (2,3)$. This, however, is an easy exercise. The assumption that $\nu<4$ is necessary to ensure that the strategy of \cite{dareiotis2021nonnegative} {also applies to \cO\eqref{Eq107}}, while the reasons for  \eqref{Eq135} and \eqref{Eq106} are subtle, see, e.g., {\cO\eqref{Eq10}} and \eqref{Eq56}, where these conditions are used. We moreover use throughout the article the regularization parameters 
	\begin{equation}\label{Eq152}
		\delta,\epsilon\,\in\, (0,1)
	\end{equation}
	and $R\in (0,\infty)$.
	Furthermore, to make this article more readable, we introduce the following notation{\cO:}
	\begin{align}&
		l\,=\, \tfrac{\nu-n}{2},\\&
		F_0(r)\,=\, r^\frac{n}{2}, \quad r\ge 0,\\&
		\label{Eq_Fd}F_\delta(r)\,=\, \tfrac{r^\frac{n+\nu}{2}}{r^\frac{\nu}{2}+ \delta^l r^\frac{n}{2}},\quad r\ge 0,\\&
		K_\epsilon(r)\,=\, (r^2+\epsilon^2)^\frac{1}{2}, \\&
		F_{\delta,\epsilon}(r)\,=\, F_\delta(K_\epsilon(r)),\\&
		\label{Eq_Jd}
		J_\delta(r)\,=\, \int_0^r \int_{r'}^\infty (F_{\delta}''(r''))^2\, dr''\, dr', \quad r\ge 0, \\&
		\label{Eq_Jde}
		J_{\delta,\epsilon}(r)\,=\, \int_0^r \int_{r'}^\infty (F_{\delta,\epsilon}''(r''))^2\, dr''\, dr', \\&
		\label{Eq_Ld}
		L_{\delta}(r)\,=\, \int_1^r \int_{1}^{r'} \frac{J_{\delta}(r'')}{F_{\delta}^2(r'')}\, dr''\, dr',\quad  r>0,\\&
		\label{Eq_Lde}
		L_{\delta,\epsilon}(r)\,=\, \int_1^r \int_{1}^{r'} \frac{J_{\delta,\epsilon}(r'')}{F_{\delta,\epsilon}^2(r'')}\, dr''\, dr',\\&\label{Eq_Gd}
		G_\delta(r)\,=\, \int_r^\infty \int_{r'}^\infty \frac{1}{F_{\delta}^2(r'')}\, dr''\, dr', \quad r>0,\\&
		\label{Eq_Gde}
		G_{\delta,\epsilon}(r)\,=\, \int_r^\infty \int_{r'}^\infty \frac{1}{F_{\delta,\epsilon}^2(r'')}\, dr''\, dr',\\&\label{Eq_Hde}
		H_{\delta,\epsilon}(r)\,=\,\int_r^\infty \frac{1}{F_{\delta,\epsilon}(r')}\, dr'.
	\end{align}
	The parameter $l$ is chosen in this way to ensure that the denominator of $F_\delta$ scales appropriately. While $F_0$ is the square-root of the original mobility $u^n$, $F_\delta$ is its approximation in the spirit of  {\cO \cite[Section 6]{BERNIS1990}}. The smooth approximations of the modulus function $K_\epsilon$ {\color{OliveGreen}satisfy
		\begin{equation}\label{Eq145}
		\tfrac{1}{\sqrt{2}}(r+\epsilon)\,\le \, 
		K_\epsilon(r) \,\le \, r+\epsilon
		,\quad r> 0,
	\end{equation}}
	 and the corresponding mobility function $F_{\delta,\epsilon}$ appear in the regularization procedure from \cite{dareiotis2021nonnegative}. 
	
	To interpret $J_\delta$ and $L_\delta$, we calculate that, when inserting $F_0$ instead of its modifications, one obtains
	\[
	J_0(r)\,=\, \int_0^r \int_{r'}^\infty (F_0''(r''))^2\,dr''\, dr'\,\eqsim_n\, 
	\int_0^r \int_{r'}^\infty (r'')^{n-4} \,dr''\, dr'\, \eqsim_n\, r^{n-2},\quad r>0
	\]
	and consequently
	\begin{align}&
	L_0(r)\,=\, \int_1^r \int_{1}^{r'} \frac{ J_0(r'')}{F_0^2(r'')}\,dr''\, dr'\,\eqsim_n\, 
	\int_1^r \int_1^{r'} \frac{(r'')^{n-2}}{(r'')^n} \,dr''\, dr'\,=\, (r-1)\,-\, \log(r),\quad r>0.
	\end{align}
	Hence, $L_\delta$ and $L_{\delta,\epsilon}$ are approximations of the $\log$-entropy functional.
	
	Also a remark concerning the existence of the integrals in the definitions of $J_\delta$ and $J_{\delta,\epsilon}$ is in order. We calculate explicitly that
\begin{align}&\label{Eq34}
	F_\delta'(r) \,=\,
	r^{\frac{n+\nu}{2}-1}\tfrac{nr^\frac{\nu}{2}\,+\, \delta^l  \nu r^\frac{n}{2}}{2(r^\frac{\nu}{2}+\delta^l  r^\frac{n}{2})^2}, \quad r>0,
	\\& \label{Eq33}
	F_\delta''(r)\,=\, r^{\frac{n+\nu}{2}-2}\tfrac{\delta^{2l}(\nu-2)\nu r^n \,-\, \delta^l  (\nu^2+\nu(2-4n) +n(n+2)) r^\frac{n+\nu}{2}\,+\, n(n-2)r^\nu}{4(r^\frac{\nu}{2}+\delta^l  r^\frac{n}{2})^3}, \quad r>0,
\end{align}
from which we conclude
\begin{align}&\label{Eq110}
	|F_\delta'(r)|\,\lesssim_{n,\nu} \, r^{\frac{n}{2}-1},  \quad r>0,
	\\& \label{Eq109}\noeqref{Eq109}
	|F_\delta''(r)|\,\lesssim_{n,\nu} \,  
	r^{\frac{n}{2}-2},\quad r>0.
\end{align}
In particular \eqref{Eq109} implies that 
\[(F_\delta''(r))^2 \,\lesssim_{n,\nu} \,
r^{n-4}, \quad r>0
\]
is integrable at infinity with 
\[
\int_{r}^\infty (F_{\delta}''(r'))^2\, dr'\, \lesssim_{n,\nu}\,
r^{n-3}, \quad r>0.
\]
This again is integrable at $0$ such that \eqref{Eq_Jd} is well-defined.

To argue in the same way for the definition of $J_{\delta,\epsilon}$, we observe first that by the chain rule
\begin{equation}\label{Eq39}\noeqref{Eq39}
	F_{\delta,\epsilon}''(r)\,=\, F_{\delta}''(K_\epsilon(r)) (K_\epsilon'(r))^2\,+\, F_{\delta}'( K_\epsilon(r) )K_\epsilon''(r)
\end{equation}
and furthermore
\begin{align}& \label{Eq44}\noeqref{Eq44}
		K_\epsilon'(r)\,=\, \frac{r}{(r^2+\epsilon^2)^\frac{1}{2}}\,\in \, (-1,1),
		\\&
		\label{Eq45}
		K_\epsilon''(r)\,=\, 
		\frac{\epsilon^2}{(r^2+\epsilon^2)^\frac{3}{2}}\,\le \, 
		\frac{1}{K_\epsilon(r)}.
\end{align}
Combining {\cO\eqref{Eq110}--\eqref{Eq45}} we conclude that
\begin{align}\label{Eq134}
	|F_{\delta,\epsilon}''(r)| \,\lesssim_{n,\nu}\, K_\epsilon^{\frac{n}{2}-2}(r)\,+\,
	\frac{K_\epsilon^{\frac{n}{2}-1}(r)}{K_\epsilon(r)}\, \eqsim \, K_\epsilon^{\frac{n}{2}-2}(r).
\end{align}
Taking the square results in 
\begin{align} \label{Eq123}
	(F_{\delta,\epsilon}''(r))^2\, \lesssim_{n,\nu} K_\epsilon^{n-4}(r).
\end{align}
The right hand side is integrable at infinity and {\cO since $\Fde$ is smooth also}
\[
\int_{r}^\infty (F_{\delta,\epsilon}''(r'))^2\, dr'
\]
defines a smooth function and in particular \eqref{Eq_Jde} is well-defined.

Since, by the previous considerations, the integrands in \eqref{Eq_Ld} and \eqref{Eq_Lde} are smooth functions on the domains of integration, also the definitions of $L_\delta$ and $L_{\delta,\epsilon}$ make sense.
Lastly, the functions $G_{\delta,\epsilon}$, $G_\delta$ and $H_{\delta,\epsilon}$ are defined analogously to \cite[Eq. (4.1)]{dareiotis2021nonnegative} and $G_\delta$, $G_{\delta,\epsilon}$ are approximations of the entropy function {\cO appearing in \eqref{Eq11}}.

	{\color{OliveGreen}\section{Construction of solutions}\label{Sec_main}
		{\cO
	This section is divided into  four subsections along which we prove our main results Theorem \ref{Thm_Existence} and Proposition \ref{Prop_Entr_Est}. In Subsection \ref{SS_Galerkin} we argue that the Galerkin setup from \cite[Section 3]{dareiotis2021nonnegative} also works to construct solutions to the equation \eqref{Eq108} below. In particular, we identify the main difference to \cite{dareiotis2021nonnegative}, namely that the approximate mobility functions $F_{\delta,\epsilon}^2(u)$ are inhomogeneous in $(\epsilon, u)$, as a mere technicality. Similarly, we use in Subsection \ref{SS_r_eps_limit} the approach from \cite[Sections 4 and 5]{dareiotis2021nonnegative} to prove the existence of weak martingale solutions to the auxiliary equation \eqref{Eq107}, by letting first $R\to \infty$ and then $\epsilon\searrow 0$ in \eqref{Eq108}. While some adaptions to the inhomogeneous setting are in order, as can be seen from the changed exponents in the energy estimate \eqref{Eq170} compared to \cite[Eq. (2.6)]{dareiotis2021nonnegative}, we postpone the details to Section \ref{Sec_remaining_proofs}. 
	
	In Subsection \ref{SS_delta_Est} we prove the $\delta$-uniform $\log$-entropy and energy estimate on the solutions to \eqref{Eq107}, which are at the heart of this manuscript. In particular, we solve difficulties related to the degeneracy of \eqref{Eq107} by applying It\^o's formula to the solutions to \eqref{Eq108} and passing then to the limits $R\to \infty$ and $\epsilon\searrow 0$. Interestingly, while the estimates on the solutions to \eqref{Eq108} still contain $\delta$-dependent constants, these terms vanish as $\epsilon\searrow 0$. This is because the solutions to \eqref{Eq108} are not non-negative, which leads to difficulties when deriving $\alpha$-entropy estimates. With a uniform energy estimate at hand, we follow in Subsection \ref{SS_delta_lim} once more \cite[Section 5]{dareiotis2021nonnegative} to identify the limit as a weak martingale solution to \eqref{Eq102} and give proof to Theorem \ref{Thm_Existence} and Proposition \ref{Prop_Entr_Est}. 
	}}

	{\color{OliveGreen}\subsection{Galerkin approximation}\label{SS_Galerkin}
	The starting point of our approximation procedure is a Galerkin setup} to construct martingale solutions to the non-degenerate SPDE
			\begin{align}\begin{split}&\label{Eq108}
				d \uder \,=\,  -\partial_x (F_{\delta,\epsilon}^2(\uder) \partial_x^3{\cO \uder} )\, dt
				\\&\qquad+\, \frac{1}{2} g_R^2(\|\uder\|_{C(\TT)})\sum_{k\in \ZZ} \partial_x (\sigma_k F_{\delta,\epsilon}'(\uder)\partial_x (\sigma_k F_{\delta,\epsilon}(\uder)))\, dt	\\&\qquad+\, g_R(\|\uder\|_{C(\TT)}) \sum_{k\in \ZZ} \, \partial_x (\sigma_k F_{\delta,\epsilon}(\uder))d\beta^{(k)}.
			\end{split}
		\end{align}
	Here, $g\colon [0,\infty)\to [0,1]$ is a smooth function with $g(r)=1$ for $r\in [0,1]$ and $g(r)=0$ for $r\in [2,\infty)$ and $g_R(r)= g(\frac{r}{R})$. We use the following notion of martingale solutions to \eqref{Eq108} from \cite[Definition 3.2]{dareiotis2021nonnegative}.
	\begin{defi}\label{defi_sol_non_degenerate}
		A martingale solution to \eqref{Eq108} with initial value $u_0\in L^2(\Omega,\mathfrak{F}_0{\cO;} H^1(\TT))$ consists of a probability space $(\hat{\Omega}, \hat{\mathfrak{A}}, \hat{\PP})$ with a filtration $\hat{\mathfrak{F}}$ satisfying the usual conditions, a family $(\hat{\beta}^{(k)})_{k\in \ZZ}$ of independent $\hat{\mathfrak{F}}$-Brownian motions and a continuous, $\hat{\mathfrak{F}}$-adapted, $H^1(\TT)$-valued process $\uderh$ such that
		\begin{enumerate}[label=(\roman*)]
			\item $\uderh(0)$ has the same distribution as $u_0$,
			\item $\hat{\EE}[\sup_{t\in [0,T]} \|\uderh(t)\|_{H^1(\TT)}^2]<\infty$,
			\item for $\hat{\PP}\otimes dt$-almost all $(\hat{\omega},t)\in \hat{\Omega}\times [0,T]$ the weak derivative of third order $\partial_x^3\uderh$ exists and satisfies $\hat{\EE}[\| F_{\delta,\epsilon}(\uderh) \partial_x^3\uderh \|_{L^2([0,T]\times \TT)}^2]<\infty$,
			\item for all $\vp\in C^\infty(\TT)$, $\hat{\PP}$-almost surely, we have 
			\begin{align}&
				(\uderh(t),\vp)_{L^2(\TT)}\,=\, 
				(\uderh(0),\vp)_{L^2(\TT)}\,+\,\int_0^t \int_{\TT} 
				F_{\delta,\epsilon}^2(\uderh)\partial_x^3 \uderh \,\partial_x\! \vp
				\, dx\, ds 
				\\&\qquad -\, \frac{1}{2}\sum_{k\in \ZZ} \int_0^t g_R(\|\uderh\|_{C(\TT)}) (\sigma_kF_{\delta,\epsilon}'(\uderh) \partial_x (\sigma_k F_{\delta,\epsilon}(\uderh)), \partial_x \!\vp )_{L^2(\TT)}\, ds\\&\qquad-\, \sum_{k\in \ZZ} \int_0^t g_R(\|\uderh\|_{C(\TT)}) (\sigma_k F_{\delta,\epsilon}(\uderh), \partial_x\! \vp)_{L^2(\TT)} \, d{\cO\hat{\beta}^{(k)}},
			\end{align}
			{\cO for all $t\in [0,T]$.}
		\end{enumerate}
	\end{defi}
	The proof of existence of solutions to \eqref{Eq108} is completely analogous to the one of \cite[Proposition 3.4]{dareiotis2021nonnegative} with mobility exponent $n${\cO, which is the growth exponent of the mobility function $\Fde^2$.} 
	\begin{lemma}[Existence for \eqref{Eq108}]\label{Lemma_Ex_uder}
		Let $p\ge n+2$ and $u_0\in L^p(\Omega,\mathfrak{F}_0{\cO;} H^1(\TT))$, then there exists a martingale solution 
			\[
		\bigl\{
		(\hat{\Omega}, \hat{\mathfrak{A}},\hat{\mathfrak{F}},\hat{\PP} ), \,(\hat{\beta}^{(k)})_{k\in \ZZ},\, \uderh
		\bigr\}
		\]
		to \eqref{Eq108} in the sense of Definition \ref{defi_sol_non_degenerate} with initial value $u_0$. Moreover, this solution satisfies the estimate
		\begin{align}\begin{split}
				\label{Eq114}&
			\hat{\EE}\biggl[
			\sup_{t\in [0,T]} \|\partial_x \uderh (t)\|_{L^2(\TT)}^p \,+\, \| F_{\delta,\epsilon}(\uderh)\partial_x^3 \uderh \|_{L^2([0,T]\times \TT)}^p
			\biggr]
			\\&\quad
			\lesssim_{\delta,\epsilon,n,\nu,p,R,\sigma,T} \,
			\EE\Bigl[
			\|\partial_x u_0 \|_{L^2(\TT)}^p
			\Bigr]
			\,+\, 1.
			\end{split}
		\end{align}
	\end{lemma}
	\begin{proof}
		The proof of  \cite[Proposition 3.4]{dareiotis2021nonnegative} is based on a Galerkin scheme, which, by \cite[Lemma 3.1]{dareiotis2021nonnegative}, is compatible with the energy inequality for \eqref{Eq108}. With the energy estimate \cite[Eq. (3.6)]{dareiotis2021nonnegative} at hand, it suffices to use that the function $F_{\delta,\epsilon}$ is bounded away from zero and $F_{\delta,\epsilon}(r)\sim |r|^\frac{n}{2}$ for  $|r|$ large to take the limit in the Galerkin scheme. The estimate \eqref{Eq114} follows by using lower semi-continuity of the norm with respect to weak-* convergence and Fatou's lemma to take $N\to \infty$ in \cite[Eq. (3.6)]{dareiotis2021nonnegative}.
	\end{proof}
{\color{OliveGreen}\subsection{The limits \texorpdfstring{$R\to \infty$}{R to infinity} and \texorpdfstring{$\epsilon\searrow 0$}{epsilon to zero}}\label{SS_r_eps_limit}
Next, we} provide a version of the entropy estimate {\cite[Lemma 4.3]{dareiotis2021nonnegative}}, which is uniform in {\cO$R$, $\epsilon$ and $\delta$.}
	\begin{lemma}[{\cO$(R,\epsilon,\delta)$-Uniform Entropy Estimate}]\label{Lemma_Entropy_Est}
			Let $p\ge 1$ and $u_0\in L^{n+2}(\Omega,\mathfrak{F}_0{\cO;} H^1(\TT))$. Then any  martingale solution
		\[
		\bigl\{
		(\hat{\Omega}, \hat{\mathfrak{A}}, \hat{\mathfrak{F}}, \hat{\PP}), \,(\hat{\beta}^{(k)})_{k\in \ZZ}, \uderh 
		\bigr\}
		\]
		to \eqref{Eq108} in the sense of Definition \ref{defi_sol_non_degenerate} with initial value $u_0$ satisfies
		\begin{align}\begin{split}\label{Eq_Entr_E}
				&
				\hat{\EE}\biggl[
				\sup_{t\in [0,T]} \|G_{\delta,\epsilon}(\uderh(t))\|_{L^1(\TT)}^p
				\,+\, 
				\|\partial_x^2 \uderh\|_{L^2([0,T]\times \TT)}^{2p}
				\biggr]\\&\quad \lesssim_{n,\nu,p,\sigma, T}\, {\EE}{\color{OliveGreen}\biggl[
				\|G_{\delta,\epsilon}(u_0)\|_{L^1(\TT)}^p
				\,+\,  \biggl|\int_{\TT}u_0 \,dx\biggr|^{2p}
				\biggr]}\,+\, 1.
			\end{split}
		\end{align}
	\end{lemma}
\begin{proof}
	The proof follows along the lines of \cite[Lemma 4.3]{dareiotis2021nonnegative} using the analogs Lemma \ref{Lemma_Hde} and Lemma \ref{Lemma_logFde} of \cite[{\cO Lemmas 4.1, 4.2}]{dareiotis2021nonnegative}.
\end{proof}

The following version of the energy estimate is uniform in $R$ and $\epsilon${\cO, but not in $\delta$. A $\delta$-uniform version will be  provided in Lemma \ref{Lemma_EE} below.}
	\begin{lemma}[{\cO$(R,\epsilon)$-Uniform Energy Estimate}]\label{Lemma_EEeps}
		Let $p\ge 1$ and $u_0\in L^{n+2}(\Omega,\mathfrak{F}_0{\cO;} H^1(\TT))$. Then any martingale solution
	\[
	\bigl\{
	(\hat{\Omega}, \hat{\mathfrak{A}}, \hat{\mathfrak{F}}, \hat{\PP}), \,(\hat{\beta}^{(k)})_{k\in \ZZ}, \uderh 
	\bigr\}
	\]
	to \eqref{Eq108} in the sense of Definition \ref{defi_sol_non_degenerate} with initial value $u_0$ satisfies
		\begin{align}\begin{split}\label{Eq_En_E1}
			&
			\hat{\EE}\biggl[
			\sup_{t\in [0,T]} \|\partial_x\uderh(t))\|_{L^2(\TT)}^p\,+\, \biggl(
			\int_0^T\int_{\TT}
			F_{\delta,\epsilon}^2(\uderh)(\partial_x^3 \uderh)^2 
			\,dx\, dt\biggr)^\frac{p}{2}
			\biggr]\\&\quad \lesssim_{\delta,n,\nu,\sigma,p, T}\,{\EE}{\color{OliveGreen}\biggl[
			\|\partial_x u_0\|_{L^2(\TT)}^p
			\,+\, \biggl|\int_{\TT}u_0\, dx\biggr|^{\frac{pn}{2}}
			\biggr]}\,+\, \hat{\EE}\Bigl[ \|\partial_x^2\uderh\|_{L^2([0,T]\times \TT)}^{2p}
			\Bigr]\,+\, 1.
		\end{split}
	\end{align}
	\end{lemma}
	\noindent{\color{OliveGreen}The proof of Lemma \ref{Lemma_EEeps} is  quite analogous to the proof of \cite[Lemma 4.6]{dareiotis2021nonnegative} and therefore postponed to the following Section \ref{Sec_remaining_proofs}.}

	For the sake of completeness, we repeat Definition \ref{defi_sol} for
	{\cO the stochastic thin-film equation with inhomogeneous mobility function \eqref{Eq107}, i.e.,
		\begin{align}\begin{split}&
			\label{Eq107b}
			d u_\delta \,=\,  -\partial_x (F_\delta^2(u_\delta) \partial_x^3u_\delta )\, dt
			\\&\qquad+\, \frac{1}{2}\sum_{k\in \ZZ} \partial_x (\sigma_k F_\delta'(u_\delta)\partial_x (\sigma_k F_\delta(u_\delta)))\, dt\,+\, \sum_{k\in \ZZ} \, \partial_x (\sigma_k F_\delta(u_\delta))d\beta^{(k)}.
		\end{split}
	\end{align}
	} We remark that we restrict ourselves also here to the case of non-negative solutions.
		\begin{defi}\label{defi_sol_approx}
		A weak martingale solution to {\cO \eqref{Eq107b}} with initial value $u_0\in L^2(\Omega,\mathfrak{F}_0{\cO;} H^1(\TT))$ consists  of a probability space $(\check{\Omega}, \check{\mathfrak{A}}, \check{\PP})$ with a filtration $\check{\mathfrak{F}}$ satisfying the usual conditions, a family $(\check{\beta}^{(k)})_{k\in \ZZ}$ of independent $\check{\mathfrak{F}}$-Brownian motions and a weakly continuous, $\check{\mathfrak{F}}$-adapted, non-negative, $H^1(\TT)$-valued process $\udc$ such that
		\begin{enumerate}[label=(\roman*)]
			\item $\udc(0)$ has the same distribution as $u_0$,
			\item $\check{\EE}[\sup_{t\in [0,T]} \|\udc(t)\|_{H^1(\TT)}^2]<\infty$,
			\item for $\check{\PP}\otimes dt$-almost all $(\check{\omega},t)\in \check{\Omega}\times [0,T]$ the weak derivative of third order $\partial_x^3\udc$ exists on $\{\udc> 0\}$ and satisfies $\check{\EE}[\|\mathbbm{1}_{\{\udc> 0\}} F_\delta(\udc) \partial_x^3\udc\|_{L^2([0,T]\times \TT)}^2]<\infty$,
			\item for all $\vp\in C^\infty(\TT)$, $\tilde{\PP}$-almost surely, we have 
			\begin{align}&
				(\udc(t),\vp)_{L^2(\TT)}\,=\, 
				(\udc(0),\vp)_{L^2(\TT)}\,+\,\int_0^t \int_{\{\udc(s)>0\}} 
				F_\delta^2(\udc)\partial_x^3 \udc \,\partial_x \!\vp
				\, dx\, ds 
				\\&\qquad -\, \frac{1}{2}\sum_{k\in \ZZ} \int_0^t(\sigma_kF_\delta'(\udc) \partial_x (\sigma_k F_\delta(\udc)), \partial_x\! \vp )_{L^2(\TT)}\, ds\,-\, \sum_{k\in \ZZ} \int_0^t (\sigma_k F_\delta(\udc), \partial_x \!\vp)_{L^2(\TT)} \, d{\cO\check{\beta}^{(k)},}
			\end{align}{\cO
			for all $t\in [0,T]$.}
		\end{enumerate}
	\end{defi}
	Since, as in \cite{dareiotis2021nonnegative}, we have obtained an $(R,\epsilon)$-uniform estimate on the entropy and energy of the approximate solutions, we can take the limit $R\to \infty$ and then $\epsilon\searrow 0$ to extract a weak martingale solution to {\cO\eqref{Eq107b} as in \cite[Proposition 4.7, Section 5]{dareiotis2021nonnegative}}.
	\begin{lemma}[{\cO Existence for \eqref{Eq107b}}]
		\label{Lemma_Ex_ud}
		Let $p>n+2$ and $u_0 \in L^p(\Omega, \mathfrak{F}_0{\cO;} H^1(\TT))$ with $u_0\ge 0$, {\color{OliveGreen}$\EE[|\int_{\TT} u_0 \,dx|^{2p}]<\infty$} and $\EE[\|G_\delta(u_0)\|_{L^1(\TT)}^{p}]<\infty$. Then there exists a weak martingale solution
		\[
		\bigl\{
		(\check{\Omega}, \check{\mathfrak{A}},\check{\mathfrak{F}},\check{\PP} ), \,(\check{\beta}^{(k)})_{k\in \ZZ},\, \udc
		\bigr\}
		\]
		to {\cO\eqref{Eq107b}} in the sense of Definition \ref{defi_sol_approx} with initial value $u_0$. Moreover, this solution satisfies the estimate
		\begin{align}\begin{split}\label{Eq170}
				&
			\check{\EE}\biggl[
			\sup_{t\in [0,T]} \|\partial_x\udc(t)\|_{L^2(\TT)}^p \,+\, 
			\sup_{t\in [0,T]} \|G_\delta(\udc(t))\|_{L^1(\TT)}^{p}
			\biggr]
			\\&
			\qquad+\, \check{\EE}\Bigl[\|
			\mathbbm{1}_{\{\udc> 0\}}F_{\delta}(\udc)\partial_x^3(  \udc)
			\|_{L^2([0,T]\times \TT)}^p
			\,+\, \|\partial_x^2 \udc\|_{L^2([0,T]\times \TT)}^{2p}
			\Bigr] \\&\quad \lesssim_{\delta,n,\nu,\sigma,p, T}\, {\EE}{\color{OliveGreen}\biggl[
			\|\partial_x u_0\|_{L^2(\TT)}^p
			\,+\, \biggl|\int_{\TT} u_0\, dx\biggr|^{2p}
			\,+\,  \|G_\delta(u_0)\|_{L^1(\TT)}^{p}
			\biggr]}\,+\,1
			\end{split}
		\end{align}
		and in particular $\udc>0$, $\check{\PP}\otimes dt\otimes dx$-almost everywhere.
	\end{lemma}{\color{OliveGreen}
	\noindent
	As it essentially follows along the lines of \cite[Proposition 4.7, Section 5]{dareiotis2021nonnegative}, the proof of Lemma \ref{Lemma_Ex_ud} is deferred to the following Section \ref{Sec_remaining_proofs} as well.}
	{\color{OliveGreen}
		\subsection{\texorpdfstring{$\delta$}{delta}-uniform estimates} \label{SS_delta_Est}  As laid out at the beginning of this section, we derive in the current subsection $\log$-entropy and corresponding energy estimates on the solutions to \eqref{Eq107b}, which are uniform in $\delta$.}  
		From a technical viewpoint, the main obstacle is the application of It\^o's formula for the energy of $\udc$ for which one would require $\udc\in L^2(0,T; H^3(\TT))$. We solve this issue by applying It\^o's formula on the level of the approximations, namely {\cO to} $\uderh$, instead. The possible negativity of $\uderh$ leads to problems when closing the estimate, which is reflected by the right hand side of \eqref{Eq111} containing constants which depend on $\delta$. As we will show in Lemma \ref{Lemma_EE} and Lemma  \ref{Lemma_log_entropy_est}, the terms containing these constants disappear when $\epsilon\searrow 0$, i.e., as the solutions become non-negative. 
		Before proceeding, we recall the 
	notations $f^{\pm}$ and $f_{\pm}$ for the positive and negative part of a function $f$.
	\begin{lemma}[{\cO Approximate $\log$-Entropy Estimate}]\label{Lemma_log_Entr_Approx_level}
		Let $p> n+2$ and $u_0\in L^\infty(\Omega,\mathfrak{F}_0{\cO;} H^1(\TT))$ such that $u_0\ge \delta$. Then any martingale solution
		\[
		\bigl\{
		(\hat{\Omega}, \hat{\mathfrak{A}}, \hat{\mathfrak{F}}, \hat{\PP}), \,(\hat{\beta}^{(k)})_{k\in \ZZ}, \uderh 
			\bigr\}
		\]
		 to \eqref{Eq108} {\cO in the sense of Definition \eqref{defi_sol_non_degenerate}} with initial value $u_0$ constructed in Lemma \ref{Lemma_Ex_uder} satisfies
			\begin{align}\begin{split}\label{Eq111}
					&
			\hat{\EE}\biggl[
			\sup_{t\in [0,T]}	\biggr(
			\int_{\TT} L_{\delta,\epsilon}^+(\uderh(t))\, dx\biggl)^\frac{p}{2}
			\,+\, \biggl(\int_0^T \int_{\TT}
			J_{\delta,\epsilon}^+(\uderh)(\partial_x^2 \uderh)^2
			\, dx\,ds\biggr)^\frac{p}{2}\biggr]
			\\&
			\qquad +\, \hat{\EE}\biggl[	\biggl(\int_0^T
			\int_{\TT}   (F_{\delta,\epsilon}''(\uderh))^2 (\partial_x \uderh)^4\,dx\,ds\biggr)^\frac{p}{2} 
			\biggr]
			\\& \quad \lesssim_{n,\nu,p,\sigma, T} C_{\delta,n,\nu,p,T} \hat{\EE}\Bigl[
			\|\uderh^-\|_{C([0,T]\times \TT)}^p\Bigl]^\frac{1}{2}{\color{OliveGreen}\biggl(\EE\biggl[
			\|G_{\delta,\epsilon}(u_0)\|_{L^1(\TT)}^p
			\,+\,  \biggl|\int_{\TT}u_0 \, dx\biggr|^{2p}
			\biggr]\,+\, 1\biggr)^\frac{1}{2}}
			\\&\qquad+\, \hat{\EE}\biggl[
			\sup_{t\in [0,T]}\|\uderh^+(t)\|_{L^1(\TT)}^{(n-2)\frac{p}{2}} 
			\biggr]\,+\, C_{\delta,n,\nu,p}\hat{\EE}\Bigl[
			\|\uderh^-\|_{C([0,T]\times \TT)}^{(4-n)\frac{p}{2}}\,+\,
			\|\uderh^-\|_{C([0,T]\times \TT)}^{(4-\nu)\frac{p}{2}} \Bigl]
			\\&
			\qquad+\, \EE\biggl[\biggl(
			\int_{\TT} (u_0-1) -\log(u_0)\, dx\biggr)^\frac{p}{2}\biggr].
				\end{split}
		\end{align}
	\end{lemma}
	{\color{OliveGreen}\begin{rem}
			The appearance of the subscripts $(n,\nu,p,T)$ and $(n,\nu,p)$ in the $\delta$-dependent constants on the right-hand side of \eqref{Eq111} may seem surprising since they are already present in the subscript of the $\lesssim$-relation. However, they are indispensable since the constants $C_{\delta,n,\nu,p,T}$ and $C_{\delta,n,\nu,p}$ do not necessarily factorize into a $\delta$-dependent and $\delta$-independent part.
	\end{rem}}
	\begin{proof}[Proof of Lemma \ref{Lemma_log_Entr_Approx_level}]Since we assumed $u_0\in L^\infty(\Omega, \mathfrak{F}_0{\cO;} H^1(\TT))${\cO, Lemma \ref{Lemma_Ex_uder} is applicable.}
	To ease notation write $u$ for $\uder $ and drop the hat notation during this proof. {\cO Moreover, we make use of the technical estimates proved in Lemmas \ref{Lemma_prop_Jd}--\ref{Lemma_Lde}.} We apply It\^o's formula as in \cite[Proposition A.1]{DHV_16} to the composition of $L_{\delta,\epsilon}$ with $u$ and check the assumptions. 
	Firstly, by its definition \eqref{Eq_Lde}, $L_{\delta,\epsilon}$ admits $\frac{J_{\delta,\epsilon}}{F_{\delta,\epsilon}^2}$ as its second derivative, which needs to be bounded. 
	Since $F_{\delta,\epsilon}$ is bounded away from $0$, we have
	\[
	\frac{|J_{\delta,\epsilon}(r)|}{F_{\delta,\epsilon}^2(r)}\, \lesssim_{\delta,\epsilon}\,
	\frac{|J_{\delta,\epsilon}(r)|}{F_{\delta,\epsilon}(r)},
	\]
	which is indeed bounded by \eqref{Eq139} and \eqref{Eq189}.
	We observe that $u\in L^2(\Omega{\cO;} C([0,T]{\cO;} H^1(\TT)))$
and further  that 
		\begin{equation}\label{Eq113}
		F_{\delta,\epsilon}(r)\, \le \, K_\epsilon^\frac{n}{2}(r)\, \lesssim_n \, |r|^\frac{n }{2}\,+\, 1  
	\end{equation}
	by \eqref{Eq152} and \eqref{Eq_Fd}. Thus, also the quantity
	\begin{align}&
		\EE\Bigl[\| F_{\delta,\epsilon}^2(u)\partial_x^3 u\|_{L^2([0,T]\times \TT)}^2\Bigr]
		\\&\quad \le \, 
		\EE\Bigl[\| F_{\delta,\epsilon}(u)\|_{C([0,T]\times \TT)}^2 \| F_{\delta,\epsilon}(u)\partial_x^3 u\|_{L^2([0,T]\times \TT)}^2 \Bigr]
		\\&\quad \le \, 
		\EE\Bigl[\| F_{\delta,\epsilon}(u)\|_{C([0,T]\times \TT)}^\frac{2(n+2)}{n} \Bigr]^\frac{n}{n+2} 	\EE\Bigl[\|F_{\delta,\epsilon}(u)\partial_x^3 u\|_{L^2([0,T]\times \TT)}^{n+2} \Bigr]^\frac{2}{n+2}
		\\&\quad 
		\overset{\eqref{Eq113}}{\lesssim_n}\, 
			\EE\Bigl[\| u\|_{C([0,T]\times \TT)}^{n+2}+1 \Bigr ]^\frac{n}{n+2} 	\EE\Bigl[\|F_{\delta,\epsilon}(u)\partial_x^3 u\|_{L^2([0,T]\times \TT)}^{n+2} \Bigr]^\frac{2}{n+2}
			\\&\quad \le \,
			\Bigl(\EE\Bigl[\| u\|_{C([0,T]\times \TT)}^{n+2} \Bigr]^\frac{n}{n+2} +1\Bigr) \times 	\EE\Bigl[\|F_{\delta,\epsilon}(u)\partial_x^3 u\|_{L^2([0,T]\times \TT)}^{n+2} \Bigr]^\frac{2}{n+2}
	\end{align}
	 is finite due to \eqref{Eq114}. 
	 We observe that
	 \begin{align} \label{Eq116}\noeqref{Eq116}
		|F_{\delta,\epsilon}'(r)|\,=\, |F_{\delta}'(K_\epsilon(r)) K_\epsilon'(r)| \, \lesssim_{n,\nu}\, K_\epsilon^{\frac{n}{2}-1}(r)\, \le \, |r|^{\frac{n}{2}-1} \,+\, 1 
	 \end{align}
 	by the chain rule, \eqref{Eq152}, \eqref{Eq110} and \eqref{Eq44} and recall the consequence 
	\begin{equation}\label{Eq121}
		\sum_{k\in \ZZ} \|\sigma_k\|_{C^2(\TT)}^2\,<\, \infty,
	\end{equation}
	see
	\cite[Eq. (2.2f)]{dareiotis2021nonnegative}, of condition \eqref{Eq120}.
	 Hence, we can follow \cite[p.11]{dareiotis2023solutions} to obtain that 
	\begin{align}&
		\EE\biggl[
	\biggl|
	\int_0^T \int_{\TT} 
	\sum_{k\in \ZZ} \sigma_k F_{\delta,\epsilon}'(\uderh) \partial_x (\sigma_k F_{\delta,\epsilon}(\uderh) )
	\, dx\, dt
	\biggr|^2
		\biggr]\,
		\overset{{\cO\eqref{Eq113}\text{--} \eqref{Eq121}}}{\lesssim_{n,\nu,\sigma, T}}\,\EE\biggl[
		\sup_{t\in [0,T]}\| u\|_{H^1(\TT)}^{2n-2}
		\biggr]\,+\, 1
	\end{align}
and
	\begin{align}&
	\EE\biggl[
	{\cO\sum_{k\in \ZZ}}
	\int_0^T 
	\bigl\|\partial_x ( \sigma_k F_{\delta,\epsilon}(u))\bigr\|_{L^2(\TT)}^2
	\, dt
	\biggr]\,
	\overset{{\cO\eqref{Eq113}\text{--}\eqref{Eq121}}}{\lesssim_{n,\nu,\sigma, T}}\,\EE\biggl[
	\sup_{t\in [0,T]}\| u\|_{H^1(\TT)}^{n}
	\biggr]\,+\, 1
\end{align}
are finite by \eqref{Eq114}.
	Hence, the assumptions of \cite[Proposition A.1]{DHV_16} are verified and it follows
	\begin{align}&
		\int_{\TT}L_{\delta,\epsilon}(u(t))\, dx\,-\, \int_{\TT}
		L_{\delta,\epsilon}(u(0)) \,{\cO dx}
		\\&\quad
		=\, \int_0^t\int_{\TT}
		L_{\delta,\epsilon}''(u) \partial_x u \bigl(F_{\delta,\epsilon}^2(u) \partial_x^3 u\bigr)\, dx\, ds
		\\&\qquad
		-\, \tfrac{1}{2}\sum_{k\in \ZZ} \int_0^t \gamma_u^2 \int_{\TT}
		L_{\delta,\epsilon}''(u) \partial_x u\bigl( \sigma_k F_{\delta,\epsilon}'(u)) \partial_x (\sigma_k F_{\delta,\epsilon}(u) \bigr) \, dx\, ds
		\\&\qquad
		+\, \tfrac{1}{2}\sum_{k\in \ZZ} \int_0^t \gamma_u^2 \int_{\TT} 
		L_{\delta,\epsilon}''(u)(\partial_x (\sigma_k F_{\delta,\epsilon}(u)))^2
		\, dx\, ds
		\\&\qquad
		+\, \sum_{k\in \ZZ} \int_0^t \gamma_u\int_{\TT}
		L_{\delta,\epsilon}'(u)\partial_x (\sigma_k F_{\delta,\epsilon}(u))	
		\, dx\, d\beta^{(k)}_s,
		\end{align}
	where we use again the notation $\gamma_u = g_R(\|u\|_{C(\TT)})$. Integrating by parts several times leads to
	\begin{align}
			&
		\int_{\TT}L_{\delta,\epsilon}(u(t))\, dx\,-\, \int_{\TT}
		L_{\delta,\epsilon}(u(0))  \,{\cO dx}
		\\&\quad
		\overset{\eqref{Eq_Lde}}{=}\, \int_0^t\int_{\TT}
		J_{\delta,\epsilon}(u) \partial_x u \partial_x^3 u\, dx\, ds
		\\&\qquad
		+\, {\cO \tfrac{1}{2}}\sum_{k\in \ZZ} \int_0^t \gamma_u^2 \int_{\TT}
		L_{\delta,\epsilon}''(u) \partial_x \sigma_k F_{\delta,\epsilon}(u) \partial_x ( \sigma_k F_{\delta,\epsilon}(u))
		\,dx\, ds
		\\&\qquad
		+\, \sum_{k\in \ZZ}\int_0^t\gamma_u  \int_{\TT}   
		L_{\delta,\epsilon}'(u) \partial_x (\sigma_k F_{\delta,\epsilon}(u))
		\, dx\, d\beta_s^{(k)}
		\\& \quad=\, 
		 -\,\int_0^t \int_{\TT} J_{\delta,\epsilon}(u) (\partial_x^2 u)^2\, dx\, ds
		\\&\qquad
		-\,  {\cO \int_0^t }\int_{\TT} J_{\delta,\epsilon}'(u) (\partial_xu)^2 \partial_x^2 u\, dx\, ds
		\\&\qquad
		+\,  {\cO \tfrac{1}{2}}\sum_{k\in \ZZ} \int_0^t \gamma_u^2\int_{\TT}
		L_{\delta,\epsilon}''(u) (\partial_x\sigma_k)^2 (F_{\delta,\epsilon}(u))^2
		\, dx\, ds
		\\&\qquad
		+\,  {\cO \tfrac{1}{2}}
		\sum_{k\in \ZZ} \int_0^t \gamma_u^2 \int_{\TT}
		L_{\delta,\epsilon}''(u) (\partial_x \sigma_k)\sigma_k F_{\delta,\epsilon}(u) F_{\delta,\epsilon}'(u)\partial_x u\, dx\,ds
		\\& \qquad
		+\, \sum_{k\in \ZZ}\int_0^t \gamma_u \int_{\TT}   
		L_{\delta,\epsilon}'(u) \partial_x (\sigma_k F_{\delta,\epsilon}(u))
		\, dx\, d\beta_s^{(k)}
		\\&\quad
		\overset{\eqref{Eq_Lde}}{=}\,  -\,\int_0^t \int_{\TT} J_{\delta,\epsilon}(u) (\partial_x^2 u)^2\, dx\, ds
		\\&\qquad
		+\, \tfrac{1}{3}{\cO \int_0^t }\int_{\TT} J_{\delta,\epsilon}''(u) (\partial_xu)^4\, dx\, ds
		\\&\qquad
		+\, {\cO \tfrac{1}{2} } \sum_{k\in \ZZ}\int_0^t \gamma_u^2\int_{\TT}
		J_{\delta,\epsilon}(u) (\partial_x \sigma_k)^2
		\, dx\, ds
		\\&\qquad
		-\,{\cO \tfrac{1}{2} }  \sum_{k\in \ZZ}\int_0^t \gamma_u^2\int_{\TT}
		\int_0^u
		L_{\delta,\epsilon}''(r) F_{\delta,\epsilon}(r) F_{\delta,\epsilon}'(r)
		\, dr \, \partial_x (   (\partial_x \sigma_k ) \sigma_k )
		\, dx\, ds
		\\& \qquad
		+\, \sum_{k\in \ZZ}\int_0^t \gamma_u \int_{\TT}   
		L_{\delta,\epsilon}'(u) \partial_x (\sigma_k F_{\delta,\epsilon}(u))
		\, dx\, d\beta_s^{(k)}.
	\end{align}
	We introduce the function 
	\begin{equation}\label{Eq127}
	I_{\delta,\epsilon}(r)\,=\, 
	\int_0^r
	L_{\delta,\epsilon}''(r') F_{\delta,\epsilon}(r') F_{\delta,\epsilon}'(r')
	\, dr'
	\end{equation}
and use \eqref{Eq_Jde} to rearrange this to 
\begin{align}&
	\int_{\TT} L_{\delta,\epsilon}^+(u(t))\, dx\,+\, \int_0^t \int_{\TT}
	J_{\delta,\epsilon}^+(u)(\partial_x^2 u)^2
	\, dx\,ds\,
	+\, \tfrac{1}{3}\int_0^t \int_{\TT} 
	(F_{\delta,\epsilon}''(u))^2(\partial_x u)^4
	\, dx\, ds
	\\&\quad
	=\, 
	\int_0^t \int_{\TT} J_{\delta,\epsilon}^-(u) (\partial_x^2u)^2\, dx\, ds
	\,+\, \int_{\TT} L_{\delta,\epsilon}^-(u(t))\, dx
	\,+\,\int_{\TT} L_{\delta,\epsilon}(u(0))\, dx
	\\&\qquad
	+\, {\cO \tfrac{1}{2} } \sum_{k\in \ZZ} \int_0^t \gamma_u^2 \int_{\TT}J_{\delta,\epsilon}(u) (\partial_x \sigma_k)^2\, dx\,ds\,-\,{\cO \tfrac{1}{2} }  \sum_{k\in \ZZ}
	\int_0^t \gamma_u^2 \int_{\TT}I_{\delta,\epsilon}(u) \partial_x ( (\partial_x \sigma_k)\sigma_k)\, dx\,ds
	\\&\qquad
	+\, \sum_{k\in \ZZ}\int_0^t \int_{\TT}   
	L_{\delta,\epsilon}'(u) \partial_x (\sigma_k F_{\delta,\epsilon}(u))
	\, dx\, d\beta_s^{(k)}.
\end{align}

Taking absolute values, the $\frac{p}{2}$-th power and the supremum in time on both sides, this results in 
\begin{align}&
	\sup_{t\in [0,T]}	\biggr(
	\int_{\TT} L_{\delta,\epsilon}^+(u(t))\, dx\biggl)^\frac{p}{2}\,+\, \biggl(\int_0^T \int_{\TT}
	J_{\delta,\epsilon}^+(u)(\partial_x^2 u)^2
	\, dx\,ds\biggr)^\frac{p}{2}\,
	+\, \biggl(\int_0^T \int_{\TT} 
	(F_{\delta,\epsilon}''(u))^2(\partial_x u)^4
	\, dx\, ds\biggr)^\frac{p}{2}
	\\&\quad
	\lesssim_p\, 
	\biggl(	\int_0^T\int_{\TT} J_{\delta,\epsilon}^-(u) (\partial_x^2u)^2\, dx\, ds \biggr)^\frac{p}{2}\,+\, \sup_{t\in [0,T]}\biggl(\int_{\TT} L_{\delta,\epsilon}^-(u(t))\, dx\biggr)^\frac{p}{2}\,+\, \biggl(
	\int_{\TT} L_{\delta,\epsilon}^+(u(0))\, dx
	\biggr)^\frac{p}{2}
	\\&\qquad +\, 
	\biggl(\sum_{k\in \ZZ}\int_0^T\int_{\TT}|J_{\delta,\epsilon}(u)| (\partial_x \sigma_k)^2\, dx\,ds\biggr)^\frac{p}{2}\,+\,\biggl( \sum_{k\in \ZZ}
	\int_0^T \int_{\TT}|I_{\delta,\epsilon}(u) \partial_x ( (\partial_x \sigma_k)\sigma_k)|\, dx\,ds\biggr)^\frac{p}{2}
	\\&\qquad 
	+\, \sup_{t\in [0,T]} \biggl|\sum_{k\in \ZZ}\int_0^t \int_{\TT}   
	L_{\delta,\epsilon}'(u) \partial_x (\sigma_k F_{\delta,\epsilon}(u))
	\, dx\, d\beta_s^{(k)}\biggr|^\frac{p}{2}.
\end{align}
Taking the expectation and applying the Burkholder-Davis-Gundy inequality, we conclude
\begin{align}&
	\EE\biggl[
	\sup_{t\in [0,T]}	\biggr(
	\int_{\TT} L_{\delta,\epsilon}^+(u(t))\, dx\biggl)^\frac{p}{2}\,+\, \biggl(\int_0^T \int_{\TT}
	J_{\delta,\epsilon}^+(u)(\partial_x^2 u)^2
	\, dx\,ds\biggr)^\frac{p}{2}
	\biggr]
	\\&\qquad
	+\, 	\EE\biggl[\biggl(\int_0^T \int_{\TT} 
	(F_{\delta,\epsilon}''(u))^2(\partial_x u)^4
	\, dx\, ds\biggr)^\frac{p}{2}
	\biggr]
	\\&\quad
	\lesssim_p\, 
	\EE\biggl[\biggl(	\int_0^T\int_{\TT} J_{\delta,\epsilon}^-(u) (\partial_x^2u)^2\, dx\, ds \biggr)^\frac{p}{2}\,+\, \sup_{t\in [0,T]}\biggl(\int_{\TT} L_{\delta,\epsilon}^-(u(t))\, dx\biggr)^\frac{p}{2}\,+\, \biggl(
	\int_{\TT} L_{\delta,\epsilon}^+(u(0))\, dx
	\biggr)^\frac{p}{2}\biggr]
	\\&\qquad
	+\, \EE\biggl[
	\biggl(\sum_{k\in \ZZ}\int_0^T\int_{\TT}|J_{\delta,\epsilon}(u)| (\partial_x \sigma_k)^2\, dx\,ds\biggr)^\frac{p}{2}\,+\,\biggl( \sum_{k\in \ZZ}
	\int_0^T \int_{\TT}|I_{\delta,\epsilon}(u) \partial_x ( (\partial_x \sigma_k)\sigma_k)|\, dx\,ds\biggr)^\frac{p}{2}
	\biggr]
	\\& \qquad
	+\, \EE\biggl[ \biggl(\sum_{k\in \ZZ}\int_0^T \biggl(\int_{\TT}   
	L_{\delta,\epsilon}'(u) \partial_x (\sigma_k F_{\delta,\epsilon}(u))
	\, dx\biggr)^2\, ds\biggr)^\frac{p}{4}\biggr].
\end{align}
To simplify the last term, we obtain by integration by parts that
\begin{align}&
	\int_{\TT}   
	L_{\delta,\epsilon}'(u) \partial_x (\sigma_k F_{\delta,\epsilon}(u))
	\, dx
	\,=\, -\,	\int_{\TT}   
	L_{\delta,\epsilon}''(u) \partial_x u \sigma_k F_{\delta,\epsilon}(u)
	\, dx\,
	=\, 
	\int_{\TT}   \int_0^u
	L_{\delta,\epsilon}''(r)F_{\delta,\epsilon}(r)\, dr\,  \partial_x \sigma_k 
	\, dx.
\end{align}
Using again \eqref{Eq121}, it follows
\begin{align}\begin{split}\label{Eq63}
	&\EE\biggl[
	\sup_{t\in [0,T]}	\biggr(
	\int_{\TT} L_{\delta,\epsilon}^+(u(t))\, dx\biggl)^\frac{p}{2}\,+\, \biggl(\int_0^T \int_{\TT}
	J_{\delta,\epsilon}^+(u)(\partial_x^2 u)^2
	\, dx\,ds\biggr)^\frac{p}{2}
	\biggr]
	\\&\qquad
	+\, 	\EE\biggl[\biggl(\int_0^T \int_{\TT} 
	(F_{\delta,\epsilon}''(u))^2(\partial_x u)^4
	\, dx\, ds\biggr)^\frac{p}{2}
	\biggr]
		\\&\quad
		\lesssim_{p,\sigma}\, 
		\EE\biggl[\biggl|	\int_0^T\int_{\TT} J_{\delta,\epsilon}^-(u) (\partial_x^2u)^2\, dx\, ds \biggr|^\frac{p}{2}\,+\, \sup_{t\in [0,T]}\biggl|\int_{\TT} L_{\delta,\epsilon}^-(u(t))\, dx\biggr|^\frac{p}{2}\,+\, \biggl|
		\int_{\TT} L_{\delta,\epsilon}^+(u(0))\, dx
		\biggr|^\frac{p}{2}\biggr]
		\\&\qquad
		+\, \EE\biggl[
		\biggl|\int_0^T\int_{\TT}|J_{\delta,\epsilon}(u)|\, dx\,ds\biggr|^\frac{p}{2}\,+\,\biggl| 
		\int_0^T \int_{\TT}|I_{\delta,\epsilon}(u)|\, dx\,ds\biggr|^\frac{p}{2}
		\biggr]
		\\&\qquad
		+\, \EE\biggl[ \biggl|\int_0^T \biggl(\int_{\TT}   \biggl|\int_0^u
		L_{\delta,\epsilon}''(r)F_{\delta,\epsilon}(r)\, dr\biggr| 
		\, dx\biggr)^2\, ds\biggr|^\frac{p}{4}\biggr].
	\end{split}
\end{align}
We use the properties of the approximate functions appearing in the right hand side of \eqref{Eq63} proved {\cO in Lemmas \ref{Lemma_prop_Jd}--\ref{Lemma_Lde}} to estimate
\begin{align}
			&\EE\biggl[
		\sup_{t\in [0,T]}	\biggr(
		\int_{\TT} L_{\delta,\epsilon}^+(u(t))\, dx\biggl)^\frac{p}{2}\,+\, \biggl(\int_0^T \int_{\TT}
		J_{\delta,\epsilon}^+(u)(\partial_x^2 u)^2
		\, dx\,ds\biggr)^\frac{p}{2}
		\biggr]
		\\&\qquad
		+\, 	\EE\biggl[\biggl(\int_0^T \int_{\TT} 
		(F_{\delta,\epsilon}''(u))^2(\partial_x u)^4
		\, dx\, ds\biggr)^\frac{p}{2}
		\biggr]
		\\&\quad \lesssim_{n,\nu,p,\sigma}\,
		C_{\delta,n,\nu,p}\EE\biggl[
		\|u_-\|_{C([0,T]\times \TT)}^\frac{p}{2}\biggl(
		\int_0^T\int_{\TT}(\partial_x^2 u)^2\, dx\, ds
		\biggr)^\frac{p}{2}
		\biggr]
		\\&\qquad +\, C_{\delta,n,\nu,p}\EE\biggl[
		\|u_-\|_{C([0,T]\times \TT)}^\frac{p}{2}\biggl(1+\sup_{t\in [0,T]}\int_{\TT}G_{\delta,\epsilon}(u(t))\, dx\biggr)^\frac{p}{2}
		\biggr]
		\\&\qquad
		+\, \EE\biggl[\biggl(
		\int_{\TT} (u(0)-1)- \log(u(0))\, dx\biggr)^\frac{p}{2}\biggr]
		\\&\qquad
		+\, \EE\biggl[
		\biggl(
		\int_0^T \int_{\TT} u_+^{n-2}\,dx\, ds
		\biggr)^\frac{p}{2}
		\biggr]\,+\, C_{\delta,n,\nu,p}\EE\biggl[\biggl(
		\int_{0}^{T}\int_{\TT} u_-\, dx\, dt\biggr)^\frac{p}{2}
		\biggr]
		\\&\qquad
		+\, \EE\biggl[
		\biggl(
		\int_0^T \biggl(
		\int_{\TT} u_+^{\frac{n}{2}-1}\,+\, C_{\delta,n,\nu}(u_-^{2-\frac{n}{2}}\,+\, u_-^{2-\frac{\nu}{2}}) \, dx
		\biggr)^2\, ds
		\biggr)^\frac{p}{2}
		\biggr]
		\\&\quad
		\lesssim_{n,p,T}\,
		C_{\delta,n,\nu,p}\EE\Bigl[
		\|u_-\|_{C([0,T]\times \TT)}^p\Bigr]^\frac{1}{2} \EE\biggl[\biggl(
		\int_0^T\int_{\TT}(\partial_x^2 u)^2\, dx\, ds
		\biggr)^p
		\biggr]^\frac{1}{2}
		\\&\qquad+\,  C_{\delta,n,\nu,p} \EE\Bigl[
		\|u_-\|_{C([0,T]\times \TT)}^p\Bigr]^\frac{1}{2}
		\biggl(1\,+\,\EE\biggl[\biggl(\sup_{t\in [0,T]}\int_{\TT}G_{\delta,\epsilon}(u)\, dx \biggr)^p\biggr] \biggr)^\frac{1}{2}
		\\&\qquad+\, \EE\biggl[\biggl(
		\int_{\TT} (u(0)-1) -\log(u(0))\, dx\biggr)^\frac{p}{2}\biggr]
		\\&\qquad+\,
		\EE\biggl[
		\sup_{t\in [0,T]}\|u_+(t)\|_{L^1(\TT)}^{(n-2)\frac{p}{2}} 
		\biggr]
		\,+\, C_{\delta,n,\nu,p}\EE\Bigl[
		\|u_-\|_{C([0,T]\times \TT)}^{(4-n)\frac{p}{2}}\,+\,
		\|u_-\|_{C([0,T]\times \TT)}^{(4-\nu)\frac{p}{2}} \Bigl].
	\end{align}
We note that we used the assumption $u_0\ge \delta$, when applying \eqref{Eq142}. Employing the entropy estimate  \eqref{Eq_Entr_E} {\cO from Lemma \ref{Lemma_Entropy_Est}} we arrive at \eqref{Eq111}.
\end{proof}

\begin{lemma}[{\cO $\delta$-Uniform Energy Estimate}]\label{Lemma_EE}
	Let $p> n+2$, $u_0 \in L^\infty(\Omega, \mathfrak{F}_0, H^1(\TT))$ such that $u_0\ge \delta$. Then 
	any weak martingale solution
	\[
	\bigl\{
	(\check{\Omega}, \check{\mathfrak{A}},\check{\mathfrak{F}},\check{\PP} ), \,(\check{\beta}^{(k)})_{k\in \ZZ},\, \udc
	\bigr\}
	\]
	to {\cO\eqref{Eq107b} in the sense of Definition \ref{defi_sol_approx}} with initial value $u_0$ constructed in Lemma \ref{Lemma_Ex_ud} satisfies
	\begin{align}\begin{split}
			\label{Eq115}
		&
		\check{\EE}	\biggl[
		\sup_{t\in [0,T]} \|\partial_x \check{u}_\delta(t)\|_{L^2(\TT)}^p\,+\,\biggl(\int_0^{T} \int_{\TT} F_{\delta}^2(\check{u}_\delta)(\partial_x^3 \check{u}_\delta)^2\,dx\, ds\biggr)^\frac{p}{2}
		\biggr]
		\\&\quad  \lesssim_{n,\nu, \sigma, p,T}
		\, {\EE}\biggl[
		\|\partial_x u_0\|_{L^2(\TT)}^p
	\,
		+\, {\color{OliveGreen}\biggl|\int u_0\, dx\biggr|^\frac{p(n+2)}{8-2n}}
		\,+\, \EE\biggl(
		\int_{\TT} (u_0-1)-\log(u_0)\, dx\biggr)^\frac{p}{2}\biggr]
		\,+\, 1.
		\end{split}
	\end{align}
\end{lemma} 
\begin{proof}Since we assumed $u_0\in L^\infty(\Omega, \mathfrak{F}_0{\cO;} H^1(\TT))$ and $u_0\ge \delta$, we have in particular {\color{OliveGreen}$\EE[|\int u_0\, dx|^{2p}]<\infty$} and $\EE[\|G_\delta(u_0)\|_{L^1(\TT)}^{p}]<\infty$ so that Lemma \ref{Lemma_Ex_ud} is indeed applicable.
	We verify a version of \eqref{Eq115} on the level of the approximations $\uderh$ of $\udc$. To this end, we use the It\^o expansion	{\cO\cite[Eq. (4.10)]{dareiotis2021nonnegative}} of the energy functional, i.e.,
	\begin{align}\begin{split}\label{Eq25}
			&
			\tfrac{1}{2}\|\partial_x \uderh(t)\|_{L^2(\TT)}^2\,=\, 
			\tfrac{1}{2}\|\partial_x \uderh(0)\|_{L^2(\TT)}^2\,-\, \int_0^t \int_{\TT} F_{\delta,\epsilon}^2(\uderh)(\partial_x^3 \uderh)^2\,dx\, ds
			\\&\qquad +\, {\cO\tfrac{1}{2}} \sum_{k\in \ZZ}\int_0^t \gamma_{\uderh}^2\int_{\TT}
			\sigma_k^2 (F_{\delta,\epsilon}''(\uderh))^2 (\partial_x \uderh)^4\, dx\, ds\\&\qquad
			+\,\tfrac{1}{16}\sum_{k\in \ZZ} \int_0^t \gamma_{\uderh}^2 \int_{\TT} (\partial_x (\sigma_k^2)) \bigl(
			(F_{\delta,\epsilon}^2)'''(\uderh)\,+\, 4 ((F_{\delta,\epsilon}')^2)'(\uderh)
			\bigr) (\partial_x \uderh)^3\,dx\,ds \\&\qquad+\,
			\tfrac{3}{2}
			\sum_{k\in \ZZ} \int_0^t \gamma_{\uderh}^2 \int_{\TT} ((\partial_x\sigma_k)^2 - \sigma_k (\partial_x^2\sigma_k)) (F_{\delta,\epsilon}'(\uderh))^2  (\partial_x \uderh)^2\,dx\,ds
			\\&\qquad+\,  			
			\tfrac{3}{16}\sum_{k\in \ZZ} \int_0^t \gamma_{\uderh}^2 \int_{\TT}
			 (\partial_x^2(\sigma_k^2)) (F_{\delta,\epsilon}^2)''(\uderh)(\partial_x \uderh)^2
			\,dx\,ds
			\\&\qquad+\,
			\tfrac{1}{8}\sum_{k\in \ZZ}\int_0^t  \gamma_{\uderh}^2 \int_{\TT} (4\sigma_k\partial_x^4 \sigma_k -\partial_x^4(\sigma_k^2)) F_{\delta,\epsilon}^2(\uderh)\,dx\, ds \\&\qquad+\, \sum_{k\in \ZZ}\int_0^t  \gamma_{\uderh}
			\int_{\TT} 
			\sigma_k F_{\delta,\epsilon}(\uderh)\partial_x^3 \uderh
			\,dx
			\, d{\cO\hat{\beta}^{(k)}},
		\end{split}
	\end{align}
	{\cO see Appendix \ref{App_B} for details}, where we use again the notation $\gamma_{\uderh} = g_R(\|\uderh\|_{C(\TT)})$.
	We proceed as in the proof of \cite[Lemma 4.6]{dareiotis2021nonnegative} and estimate the deterministic terms on the right-hand side separately. Here we use in particular \eqref{Eq121} together with the relation
	\begin{equation}\label{Eq117}
		\partial_x \sigma_k\,=\, 2\pi k \sigma_{-k},\quad k\in \ZZ
	\end{equation}
	from \cite[Eq. (2.2b)]{dareiotis2021nonnegative} to estimate the terms involving $\sigma_k$.  {\cO As another key ingredient, we make use of the technical estimates from Lemma \ref{Lemma_F_de_prop}.}
	
	\textbf{$(\partial_x u)^4$-Term. } We readily see that
	\begin{align}&
		\sum_{k\in \ZZ}\int_0^t \gamma_{\uderh}^2\int_{\TT}
		\sigma_k^2 (F_{\delta,\epsilon}''(\uderh))^2 (\partial_x \uderh)^4\, dx\, ds
		\\&\quad
		\overset{\eqref{Eq121}}{\lesssim_{\sigma}}\,
		\int_0^t \int_{\TT}
		(F_{\delta,\epsilon}''(\uderh))^2 (\partial_x \uderh)^4\, dx\, ds.
	\end{align}
	{\color{OliveGreen}The resulting right-hand side appears in the left-hand side of the approximate $\log$-entropy estimate \eqref{Eq111}, which we use later on, to close the energy estimate.}
	
	\textbf{$(\partial_x u)^3$-Term. } We estimate this term by absorbing it into the highest and lowest order term. To this end, we use the properties of the approximate function $F_{\delta,\epsilon}$, which we prove in Lemma \ref{Lemma_F_de_prop}, to conclude that
	\begin{align}&
		\biggl|\sum_{k\in \ZZ} \int_0^t \gamma_{\uderh}^2 \int_{\TT} (\partial_x (\sigma_k^2))\bigl(
		(F_{\delta,\epsilon}^2)'''(\uderh)\,+\, 4((F_{\delta,\epsilon}')^2)'(\uderh)
		\bigr) (\partial_x\uderh)^3\,dx\, ds\biggr|
		\\&\quad
		\overset{\eqref{Eq121}}{\lesssim_\sigma}\,
		\int_0^t \int_{\TT} \bigl|\bigl(
		(F_{\delta,\epsilon}^2)'''(\uderh)\,+\, 4((F_{\delta,\epsilon}')^2)'(\uderh)
		\bigr) (\partial_x\uderh)^3 \bigr|\,dx\, ds
		\\&\quad
		\overset{\eqref{Eq53}, \eqref{Eq54}}{\lesssim_{n,\nu}}
		\, \int_0^t \int_{\TT} (F_{\delta,\epsilon}(\uderh))^\frac{1}{2} (F_{\delta,\epsilon}''(\uderh))^\frac{3}{2} |\partial_x\uderh|^3\,dx\, ds
		\\&\quad
		\le
		\, \biggl(\int_0^t \int_{\TT} F_{\delta,\epsilon}^2(\uderh)\,dx\, ds\biggr)^\frac{1}{4}
		\,+\, 
		\biggl(\int_0^t\int_{\TT}
		(F_{\delta,\epsilon}''(\uderh))^2 (\partial_x\uderh)^4\,dx\, ds\biggr)^\frac{3}{4}
		\\&\quad\le \, 2\,+\, \int_0^t \int_{\TT} F_{\delta,\epsilon}^2(\uderh)\,dx\, ds
		\,+\, 
		\int_0^t\int_{\TT}
		(F_{\delta,\epsilon}''(\uderh))^2 (\partial_x\uderh)^4\,dx\, ds.
	\end{align}
	{\color{OliveGreen}While the third term on the right-hand side is estimated in \eqref{Eq111},  the second term is essentially a power of $u$, which we can estimate using conservation of mass together with the energy of $\uderh$.}
	
	\textbf{$(\partial_x u)^2$-Term. }
	Similarly, we conclude that
	\begin{align}&
		\biggl|\sum_{k\in \ZZ} \int_0^t \gamma_{\uderh}^2 \int_{\TT} \bigl(8((\partial_x\sigma_k)^2 - \sigma_k (\partial_x^2\sigma_k)) (F_{\delta,\epsilon}'(\uderh))^2\,+\, (\partial_x^2(\sigma_k^2)) (F_{\delta,\epsilon}^2)''(\uderh) \bigr) (\partial_x \uderh)^2\,dx\,ds\biggr|
		\\&\quad
		\overset{\eqref{Eq121}}{\lesssim_\sigma}
		\int_0^t  \int_{\TT}  \bigl((F_{\delta,\epsilon}'(\uderh))^2\,+\,  \bigl|(F_{\delta,\epsilon}^2)''(\uderh) \bigr|\bigr) (\partial_x \uderh)^2\,dx\,ds
		\\&\overset{\eqref{Eq57},\eqref{Eq58}}{\lesssim_{n,\nu}}\,  \int_0^t
		\int_{\TT}  F_{\delta,\epsilon}(\uderh) F_{\delta,\epsilon}''(\uderh) (\partial_x\uderh)^2\,dx\,ds
		\\&\quad
		\le \, \biggl(
		\int_0^t
		\int_{\TT}  F_{\delta,\epsilon}^2(\uderh)\,dx\,ds
		\biggr)^\frac{1}{2} \,+\, \biggl(
		\int_0^t
		\int_{\TT}   (F_{\delta,\epsilon}''(\uderh))^2 (\partial_x \uderh)^4\,dx\,ds
		\biggr)^\frac{1}{2}
		\\&\quad
		\le \,2\,+\, 
		\int_0^t
		\int_{\TT}  F_{\delta,\epsilon}^2(\uderh)\,dx\,ds
		\,+\, 
		\int_0^t
		\int_{\TT}   (F_{\delta,\epsilon}''(\uderh))^2 (\partial_x \uderh)^4\,dx\,ds.
	\end{align}
	\textbf{$(\partial_x u)^0$-Term. } Lastly, we also have that
	\begin{align}&
		\biggl|\sum_{k\in \ZZ}\int_0^t \gamma_{\uderh}^2 \int_{\TT} (4\sigma_k\partial_x^4 \sigma_k -\partial_x^4(\sigma_k^2)) F_{\delta,\epsilon}^2(\uderh)\,dx\, ds\biggr|\,\overset{\eqref{Eq121}, \eqref{Eq117}}{\lesssim_\sigma}\, \int_0^t \int_{\TT} F_{\delta,\epsilon}^2(\uderh)\,dx\, ds.
	\end{align}
	Inserting all this in \eqref{Eq25}, we arrive at
	\begin{align}&
		\tfrac{1}{2}\|\partial_x \uderh(t)\|_{L^2(\TT)}^2\,+\, \int_0^t \int_{\TT} F_{\delta,\epsilon}^2(\uderh)(\partial_x^3 \uderh)^2\,dx\, ds \\&\quad
		\le \,\tfrac{1}{2}\|\partial_x \uderh(0)\|_{L^2(\TT)}^2\,+\, |M(t)|\\&\qquad
		+\, C_{n,\nu,\sigma } \biggl[ 1 \,+\, \int_0^t
		\int_{\TT}  F_{\delta,\epsilon}^2(\uderh)\,dx\,ds
		\,+\, 
		\int_0^t
		\int_{\TT}   (F_{\delta,\epsilon}''(\uderh))^2 (\partial_x \uderh)^4\,dx\,ds
		\biggr],
	\end{align}
	with 
	\begin{equation}\label{Eq59}
		M(t)\,=\, \sum_{k\in \ZZ}\int_0^t \gamma_{\uderh}
		\int_{\TT} 
		\sigma_k F_{\delta,\epsilon}(\uderh)\partial_x^3 \uderh
		\,dx
		\, d{\cO\hat{\beta}^{(k)}}.
	\end{equation} 
We apply the Gagliardo Nirenberg inequality with the choice of exponents
	\[
	\tfrac{-1}{n}\,=\, \theta (\tfrac{-1}{1})\,+\, (1-\theta)(1-\tfrac{1}{2})\qquad \iff \qquad \theta\,=\, \tfrac{n+2}{3n}\]
	and Young's inequality with $\frac{4-n}{3}+\frac{n-1}{3}=1$
	to estimate further 
	\begin{align}&
		\int_{\TT}  F_{\delta,\epsilon}^2(\uderh)\,dx\, \overset{\eqref{Eq113}}{\lesssim_n}\, 1\,+\, \|\uderh\|_{L^n(\TT)}^n
		\\&\quad \lesssim_n\,
		1\,+\, \|\uderh\|_{L^1(\TT)}^n\,+\, \|\uderh\|_{L^1(\TT)}^{\theta n} \|\partial_x \uderh\|_{L^2(\TT)}^{(1-\theta)n}
		\\& \quad\le \, 1\,+\, \|\uderh\|_{L^1(\TT)}^n\,+\,  \bigl(\tfrac{1}{\kappa}\bigr)^\frac{n-1}{4-n}\|\uderh\|_{L^1(\TT)}^\frac{3\theta n}{4-n} \,+\, \kappa\|\partial_x \uderh\|_{L^2(\TT)}^\frac{3(1-\theta)n}{n-1}
		\\&\quad=\, 1\,+\, \|\uderh\|_{L^1(\TT)}^n\,+\,  \bigl(\tfrac{1}{\kappa}\bigr)^\frac{n-1}{4-n}\|\uderh\|_{L^1(\TT)}^\frac{n+2}{4-n} \,+\, \kappa\|\partial_x \uderh\|_{L^2(\TT)}^2
		\\&\quad\le \, 2\,+\,  \bigl(\bigl(\tfrac{1}{\kappa}\bigr)^\frac{n-1}{4-n}+1\bigr)\|\uderh\|_{L^1(\TT)}^\frac{n+2}{4-n} \,+\, \kappa\|\partial_x \uderh\|_{L^2(\TT)}^2.
	\end{align}
	In the last step we used that $n> 2$. Choosing $\kappa$ sufficiently small, we obtain that
	\begin{align}&
		\tfrac{1}{2}\|\partial_x \uderh(t)\|_{L^2(\TT)}^2\,+\, \int_0^t \int_{\TT} F_{\delta,\epsilon}^2(\uderh)(\partial_x^3 \uderh)^2\,dx\, ds \\&\quad
		\le \,\tfrac{1}{2}\|\partial_x \uderh(0)\|_{L^2(\TT)}^2\,+\, |M(t)|\,+\,
		\tfrac{1}{4}\sup_{s\in [0,t]}\|\partial_x \uderh(s)\|_{L^2(\TT)}^2
		\\&\qquad
		+\, C_{n,\nu,\sigma,T } \biggl[ 1 \,+\,  \sup_{t\in [0,T]}\|\uderh(t)\|_{ L^1(\TT)}^\frac{n+2}{4-n}
		\,+\, 
		\int_0^T
		\int_{\TT}   (F_{\delta,\epsilon}''(\uderh))^2 (\partial_x \uderh)^4\,dx\,ds
		\biggr].
	\end{align}
	We take the supremum in time
and take the $\frac{p}{2}$-th power on both sides
	to obtain 
	\begin{align}&
	\tfrac{1}{4}\sup_{t\in[0,T]}\|\partial_x \uderh(t)\|_{L^2(\TT)}^p\,+\,\biggl( \int_0^{T} \int_{\TT} F_{\delta,\epsilon}^2(\uderh)(\partial_x^3 \uderh)^2\,dx\, dt\biggr)^\frac{p}{2} \\&\quad
\lesssim_p \,\|\partial_x \uderh(0)\|_{L^2(\TT)}^p\,+\, \sup_{t\in[0,T]}|M(t)|^\frac{p}{2}
\\&\qquad
+\, C_{n,\nu,p,\sigma,T } \biggl[ 1 \,+\,  \sup_{t\in [0,T]}\|\uderh(t)\|_{ L^1(\TT)}^\frac{p(n+2)}{8-2n}
\,+\, 
\int_0^{T}
\biggl(\int_{\TT}   (F_{\delta,\epsilon}''(\uderh))^2 (\partial_x \uderh)^4\,dx\,dt\biggr)^\frac{p}{2}
\biggr].
	\end{align}
Taking the expectation and employing the {\cO Burkholder--Davis--Gundy} inequality leads to
	\begin{align}\begin{split}\label{Eq60}
			&
			\hat{\EE}	\biggl[
			\sup_{t\in [0,T]} \|\partial_x \uderh(t)\|_{L^2(\TT)}^p\,+\, \biggl(\int_0^{T} \int_{\TT} F_{\delta,\epsilon}^2(\uderh)(\partial_x^3 \uderh)^2\,dx\, dt\biggr)^\frac{p}{2}
			\biggr]
			\\&\quad \lesssim_{n,\nu, \sigma, p,T}
			\, \hat{\EE}\Bigl[
			\|\partial_x \uderh(0)\|_{L^2(\TT)}^p
		\,+\, 
			\langle
			M
			\rangle_{T}^\frac{p}{4}
			\Bigr]\,+\, 1
			\\&\qquad
			+\, \hat{\EE}\biggl[
			\sup_{t\in [0,T]}\|\uderh(t)\|_{ L^1(\TT)}^\frac{p(n+2)}{8-2n}
			 \,+\,
			\biggl(\int_0^T
			\int_{\TT}   (F_{\delta,\epsilon}''(\uderh))^2 (\partial_x \uderh)^4\,dx\,dt\biggr)^\frac{p}{2} 
			\biggr].
		\end{split}
	\end{align}
Since 
\begin{align}\label{Eq166}&
	\langle M \rangle_t\,=\, \sum_{k\in \ZZ} \int_0^t \biggl(
	\gamma_{\uderh} \int_{\TT} \sigma_k F_{\delta,\epsilon}(\uderh) \partial_x^3 \uderh \,dx
	\biggr)^2 \, ds\\&\quad \overset{\eqref{Eq121}}{\lesssim_\sigma} \, \int_0^t  \int_{\TT} (F_{\delta,\epsilon}(\uderh))^2 (\partial_x^3 \uderh)^2  \,dx \, ds,
\end{align}
we can estimate
\begin{align} \begin{split}&
		\label{Eq61}
	\hat{\EE} \bigl[
	\langle M \rangle_{T}^\frac{p}{4}
	\bigr] \,\lesssim_{\sigma,p} \, \hat{\EE}\biggl[\biggl(
	\int_0^{T} \int_{\TT} (F_{\delta,\epsilon}(\uderh))^2 (\partial_x^3 \uderh)^2  \,dx \, dt\biggr)^\frac{p}{4}
	\biggr] \\&\quad \le\, \kappa \hat{\EE}\biggl[\biggl(
	\int_0^{T} \int_{\TT} (F_{\delta,\epsilon}(\uderh))^2 (\partial_x^3 \uderh)^2  \,dx \, dt\biggr)^\frac{p}{2}
	\biggr]\,+\, \frac{1}{\kappa}.
	\end{split}
\end{align}
Inserting this as well as {\cO the approximate $\log$-entropy estimate} \eqref{Eq111} in \eqref{Eq60}, and choosing $\kappa$ sufficiently small, we conclude
	\begin{align}\begin{split}
			\label{Eq118}
		&
		\hat{\EE}	\biggl[
		\sup_{t\in [0,T]} \|\partial_x \uderh(t)\|_{L^2(\TT)}^p\biggr]\,+\, \hat{\EE}	\biggl[\biggl(\int_0^{T} \int_{\TT} F_{\delta,\epsilon}^2(\uderh)(\partial_x^3 \uderh)^2\,dx\, dt\biggr)^\frac{p}{2}
		\biggr]
		\\&\quad \lesssim_{n,\nu, \sigma, p,T}
		\, \EE\biggl[
		\|\partial_x u_0\|_{L^2(\TT)}^p
		\,+\, \biggl(
		\int_{\TT} (u_0-1) -\log(u_0)\, dx\biggr)^\frac{p}{2}\biggr]
		\\&\qquad+\, \hat{\EE}\biggl[
		\sup_{t\in [0,T]}\|\uderh(t)\|_{ L^1(\TT)}^\frac{p(n+2)}{8-2n}
		\biggr]\,+\, 1 
		\\&\qquad+\,C_{\delta,n,\nu,p,T} \hat{\EE}\Bigl[
		\|\uderh^-\|_{C([0,T]\times \TT)}^p\Bigr]^\frac{1}{2}{\color{OliveGreen}\biggl(\EE\biggl[
		\|G_{\delta,\epsilon}(u_0)\|_{L^1(\TT)}^p
		\,+\, \biggl|\int_{\TT} u_0\, dx \biggr|^{2p}
		\biggr]\,+\, 1\biggr)^\frac{1}{2}}
		\\&\qquad+\, C_{\delta,n,\nu,p}\hat{\EE}\Bigl[
		\|\uderh^-\|_{C([0,T]\times \TT)}^{(4-n)\frac{p}{2}}\,+\,
		\|\uderh^-\|_{C([0,T]\times \TT)}^{(4-\nu)\frac{p}{2}} \Bigr],
		\end{split}
\end{align}
where we also used 
\[\frac{p(n+2)}{8-2n}\,>\, \frac{p(n-2)}{2}.\]
When the limit $R\to \infty$ is taken in the proof of Lemma \ref{Lemma_Ex_ud}, which follows along the lines of \cite[Proposition 4.7]{dareiotis2021nonnegative}, the estimate \eqref{Eq118} is preserved. By taking the limit $\epsilon\searrow 0$, we derive \eqref{Eq115} as follows. We denote the sequence from \cite[Eq. (5.14a)]{dareiotis2021nonnegative} which converges $\check{\PP}$-almost surely to $\udc$ in $C([0,T]\times \TT)$ by $\check{u}_{\delta,\epsilon}$, such that in particular
\begin{align}&
	\sup_{t\in [0,T]}\|\check{u}_{\delta,\epsilon}(t)\|_{ L^1(\TT)}\,\to \, 
		\sup_{t\in [0,T]}\|\udc(t)\|_{ L^1(\TT)},
	\end{align}
$\check{\PP}$-almost surely.
{\color{OliveGreen} Since $r \mapsto r^-$ is Lipschitz continuous and $\check{u}_\delta$ is non-negative, we deduce  additionally
\begin{align}
			\|\check{u}_{\delta,\epsilon}^-\|_{C([0,T]\times \TT)} \,=\, 
			\|\check{u}_{\delta,\epsilon}^- - \check{u}_{\delta}^-\|_{C([0,T]\times \TT)} \,\le\, 
			 \|\check{u}_{\delta,\epsilon}- \check{u}_{\delta}\|_{C([0,T]\times \TT)}
			\,\to \, 0,
\end{align} 
$\check{\PP}$-almost surely.} By \cite[Eq. (4.24)]{dareiotis2021nonnegative} the sequence $\check{u}_{\delta,\epsilon}$ is uniformly $\epsilon$ bounded in $L^q(\check{\Omega}, C([0,T]\times \TT))$ for all $q\in (1,\infty)$, and we obtain by Vitali's convergence theorem that
\[
\check{\EE}\Bigl[
\|\check{u}_{\delta,\epsilon}^-\|_{C([0,T]\times \TT)}^p\Bigr]^\frac{1}{2}
\,+\,\check{\EE}\Bigl[
\|\check{u}_{\delta,\epsilon}^-\|_{C([0,T]\times \TT)}^{(4-n)\frac{p}{2}}\,+\,
\|\check{u}_{\delta,\epsilon}^-\|_{C([0,T]\times \TT)}^{(4-\nu)\frac{p}{2}} \Bigr]\,\to\, 0
\]
as $\epsilon\searrow 0$. Since $\check{u}_\delta$ is non-negative and preserves mass, \eqref{Eq115} follows.
\end{proof}

\begin{lemma}[{\cO $\delta$-Uniform $\log$-Entropy Estimate}]\label{Lemma_log_entropy_est}
Let $p> n+2$, $u_0 \in L^\infty(\Omega, \mathfrak{F}_0{\cO;} H^1(\TT))$ such that $u_0\ge \delta$. Then 
any weak martingale solution
\[
\bigl\{
(\check{\Omega}, \check{\mathfrak{A}},\check{\mathfrak{F}},\check{\PP} ), \,(\check{\beta}^{(k)})_{k\in \ZZ},\, \udc
\bigr\}
\]
to {\cO\eqref{Eq107b} in the sense of Definition \ref{defi_sol_approx}} with initial value $u_0$ constructed in Lemma \ref{Lemma_Ex_ud} satisfies
	\begin{align}\begin{split}
		&
		\check{\EE}	\biggl[
		\sup_{t\in [0,T]} 
		\biggl(\int_{\TT}L_\delta(\check{u}_\delta(t)) \, dx\biggr)^\frac{p}{2}
		\biggr]\, \lesssim_{n,\nu, \sigma, p,T}
		\, \EE\biggl[{\color{OliveGreen}\biggl|\int_{\TT} u_0\, dx\biggr|^{(n-2)\frac{p}{2}}}
		\,+\, \biggl(
		\int_{\TT} (u_0-1)-\log(u_0)\, dx\biggr)^\frac{p}{2}\biggr].
	\end{split}
\end{align}
\end{lemma}
\begin{proof}
	As in the proof of Lemma \ref{Lemma_EE} we first take $R\to\infty$ in \eqref{Eq111} and obtain that
	\begin{align}
		&
		\check{\EE}\biggl[
		\sup_{t\in [0,T]}	\biggr(
		\int_{\TT} L_{\delta,\epsilon}^+(\check{u}_{\delta,\epsilon}(t))\, dx\biggl)^\frac{p}{2}\biggr]
		\\&\quad  \lesssim_{n,\nu,p,\sigma, T} C_{\delta,n,\nu,p,T} \check{\EE}\Bigl[
		\|\check{u}_{\delta,\epsilon}^-\|_{C([0,T]\times \TT)}^p\Bigr]^\frac{1}{2}{\color{OliveGreen}\biggl(\EE\biggl[
		\|G_{\delta,\epsilon}(u_0)\|_{L^1(\TT)}^p
		\,+\,  \biggl|\int u_0\, dx\biggr|^{2p}
		\biggr]\,+\, 1\biggr)^\frac{1}{2}}
		\\&\qquad+\, \check{\EE}\biggl[
		\sup_{t\in [0,T]}\|\check{u}_{\delta,\epsilon}^+(t)\|_{L^1(\TT)}^{(n-2)\frac{p}{2}} 
		\biggr]\,+\, C_{\delta,n,\nu,p}\check{\EE}\Bigl[
		\|\check{u}_{\delta,\epsilon}^-\|_{C([0,T]\times \TT)}^{(4-n)\frac{p}{2}}\,+\,
		\|\check{u}_{\delta,\epsilon}^-\|_{C([0,T]\times \TT)}^{(4-\nu)\frac{p}{2}} \Bigr]
		\\&\qquad+\, \EE\biggl[\biggl(
		\int_{\TT} (u_0-1) -\log(u_0)\, dx\biggr)^\frac{p}{2}\biggr]
	\end{align}
	holds for the sequence $\check{u}_{\delta,\epsilon}$ converging $\check{\PP}$-almost surely to $\check{u}_\delta$ from \cite[Eq. (5.14)]{dareiotis2021nonnegative}. Continuing as in the proof of Lemma \ref{Lemma_EE}, we obtain
	\begin{align}
		&
		\check{\EE}\biggl[
		\sup_{t\in [0,T]}	\biggr(
		\int_{\TT}\liminf_{\epsilon\searrow 0} L_{\delta,\epsilon}^+(\check{u}_{\delta,\epsilon}(t))\, dx\biggl)^\frac{p}{2}\biggr]
		\\&\quad  \lesssim_{n,\nu,p,\sigma, T}  \EE\biggl[{\color{OliveGreen}\biggl|\int u_0\, dx\biggr|^{(n-2)\frac{p}{2}}}
		\,+\,\biggl(
		\int_{\TT} (u_0-1) -\log(u_0)\, dx\biggr)^\frac{p}{2}\biggr]
	\end{align}
	by letting $\epsilon\searrow 0$ and additionally employing Fatou's lemma. 
	It is left to argue that
	\begin{equation}\label{Eq129}
	\liminf_{\epsilon\searrow 0} L_{\delta,\epsilon}^+(\check{u}_{\delta,\epsilon})\, \ge \, L_\delta(\check{u}_\delta).
	\end{equation}
	To this end, we use that $\check{u}_{\delta,\epsilon}$ becomes eventually positive $\check{\PP}\otimes dt \otimes dx$-almost everywhere by Lemma \ref{Lemma_Ex_ud} and the notation
	\begin{equation}
		\mathcal{S}_r\,=\, 
		\begin{cases}
			(1,r),& r> 1,
			\\
			\emptyset, & r=1,
			\\
			(r,1),
			&r<1
		\end{cases}
	\end{equation} 
 to deduce
	\begin{align}&
			\liminf_{\epsilon\searrow 0} L_{\delta,\epsilon}^+(\check{u}_{\delta,\epsilon}(t))\,=\, 
				\liminf_{\epsilon\searrow 0} \int_0^\infty \mathbbm{1}_{\mathcal{S}_{\check{u}_{\delta,\epsilon}}} (r')  \int_{0}^\infty \mathbbm{1}_{\mathcal{S}_{r'}}(r'') 
					\frac{J_{\delta,\epsilon}(r'')}{F_{\delta,\epsilon}^2(r'')}\, dr''\, dr'
					\\&\quad
					\ge \,
					\int_0^\infty  \Bigl(	\liminf_{\epsilon\searrow 0} \mathbbm{1}_{\mathcal{S}_{\check{u}_{\delta,\epsilon}}} (r')\Bigr) \times \biggl(	\liminf_{\epsilon\searrow 0}\int_{0}^\infty \mathbbm{1}_{\mathcal{S}_{r'}}(r'') 
					\frac{J_{\delta,\epsilon}(r'')}{F_{\delta,\epsilon}^2(r'')}\, dr''\biggr)\, dr'
					\\&\quad
					\ge \, 
					\int_0^\infty   \mathbbm{1}_{\mathcal{S}_{\check{u}_{\delta}}} (r') 	\int_{0}^\infty  \mathbbm{1}_{\mathcal{S}_{r'}}(r'')\times \biggl( \liminf_{\epsilon\searrow 0}
					\frac{J_{\delta,\epsilon}(r'')}{F_{\delta,\epsilon}^2(r'')}\biggr)\, dr''\, dr'
					\\&\quad
					\overset{{\cO\eqref{Eq39}\text{--} \eqref{Eq45}}}{\ge}\, 
					\int_0^\infty   \mathbbm{1}_{\mathcal{S}_{\check{u}_{\delta}}} (r') 	\int_{0}^\infty  \mathbbm{1}_{\mathcal{S}_{r'}}(r'')
					\frac{J_{\delta}(r'')}{F_{\delta}^2(r'')}\, dr''\, dr'\,\overset{\eqref{Eq_Ld}}{=}\,
					L_\delta(\check{u}_{\delta}).
	\end{align}
	Here, we repeatedly applied  Fatou's lemma together with the the properties of the limes inferior. This shows \eqref{Eq129} and finishes the proof.
\end{proof}

{\color{OliveGreen}\subsection{The limit \texorpdfstring{$\delta\searrow 0$}{delta to zero}}\label{SS_delta_lim}
In this subsection we prove our main results Theorem \ref{Thm_Existence} and Proposition \ref{Prop_Entr_Est}, i.e., existence for the original equation \eqref{Eq102} and corresponding estimates on the solution. Since we have the}  $\delta$-uniform energy estimate {\cO from Lemma \ref{Lemma_EE}} at hand, the proof of Theorem \ref{Thm_Existence} follows along the lines of the proof of \cite[Theorem 2.2]{dareiotis2021nonnegative}. 
\begin{proof}[Proof of Theorem \ref{Thm_Existence}]
	 Let $p>n+2$ and $u_0\in L^p(\Omega, \mathfrak{F}_0,H^1(\TT))$ in accordance with \eqref{Eq124}. Moreover, for the given $n$, we fix one feasible $\nu$, {\cO so} that we can drop the $\nu$-dependence in the following estimates. We define 
	\begin{equation}\label{Eq132}
		u_{0,\delta} \,=\, \mathbbm{1}_{\bigl\{\|u_0\|_{H^1(\TT)} <e^\frac{1}{\delta}\bigr\} }u_0 \,+\,  \delta.
	\end{equation}
	In particular, Lemma \ref{Lemma_Ex_ud} is applicable and yields the existence of a weak martingale solutions $(\udc)_{\delta>0}$ 
	to {\cO\eqref{Eq107b}} in the sense of Definition \ref{defi_sol_approx} with initial value $u_{0,\delta}$. {\cO For ease} of notation, we assume {\cO these solutions} to be defined with respect to the same stochastic basis	\begin{equation}
		\bigl\{
		(\check{\Omega}, \check{\mathfrak{A}},\check{\mathfrak{F}},\check{\PP} ), \,(\check{\beta}^{(k)})_{k\in \ZZ} 
		\bigr\}.
	\end{equation} Also Lemma \ref{Lemma_EE} and Lemma \ref{Lemma_log_entropy_est} are applicable by \eqref{Eq132} and yield that the solutions satisfy the uniform estimate
	\begin{align}\begin{split}\label{Eq141}
		&
		\check{\EE}	\biggl[
		\sup_{t\in [0,T]} \|\partial_x \check{u}_\delta(t)\|_{L^2(\TT)}^p\,+\, 
		\sup_{t\in [0,T]} 
		\biggl(\int_{\TT}L_\delta(\check{u}_\delta(t)) \, dx\biggr)^\frac{p}{2}\,+\,\biggl(\int_0^{T} \int_{\TT} F_{\delta}^2(\check{u}_\delta)(\partial_x^3 \check{u}_\delta)^2\,dx\, ds\biggr)^\frac{p}{2}
		\biggr]
		\\&\quad \lesssim_{n, p, \sigma,T}
		\, {\EE}\biggl[
		\|\partial_x u_{0,\delta}\|_{L^2(\TT)}^p\,
		+\,{\color{OliveGreen}\biggl|\int_{\TT} u_{0,\delta} \, dx\biggr|^\frac{p(n+2)}{8-2n}}
		\,+\,\biggl(
		\int_{\TT} (u_{0,\delta}-1)-\log(u_{0,\delta})\, dx\biggr)^\frac{p}{2}\biggr]
		\,+\, 1.
	\end{split}
\end{align}
Proceeding as in \cite[Lemma 5.1]{dareiotis2021nonnegative}, we obtain the estimate
\begin{align}&
	\check{\EE}\Bigl[
	\|\check{u}_\delta\|_{ C^\frac{1}{4}([0,T]; L^2(\TT))}^q
	\Bigr]\,
	\lesssim_{n,p,q,\sigma, T}\,
	\check{\EE}\biggl[\sup_{t\in[0,T]} \| \partial_x \udc\|_{L^2(\TT)}^p  \,+\,{\color{OliveGreen}\biggl|\int_{\TT} \udc(0)\, dx\biggr|^p} \biggr]^\frac{(n+2)q}{2p}
	\\&\qquad
	+\, \check{\EE}\biggl[
	\biggl(
	\int_0^T \int_{\TT}
	F_\delta^2(\check{u}_\delta)(\partial_x^3 \udc)^2
	\,dx\,dt
	\biggr)^\frac{p}{n+2}
	\biggr]^\frac{(n+2)q}{2p}\,+\, 1
\end{align}
for any $q\in [1,\frac{2p}{n+2})$ from the last equation of the proof of \cite[Lemma 5.1]{dareiotis2021nonnegative}. Indeed, to achieve this one has to estimate the nonlinearities from {\cO \eqref{Eq107b}}, which satisfy the same bounds
\begin{equation}
	F_\delta(r)\,\le \, r^\frac{n}{2}, \quad r>0{\cO,}
\end{equation} 
and \eqref{Eq110} as in the case of a homogeneous mobility function.
 Inserting \eqref{Eq141}, we conclude
\begin{align}\begin{split}\label{Eq167}
		&
	\check{\EE}\Bigl[
	\|\check{u}_\delta\|_{ C^\frac{1}{4}([0,T]; L^2(\TT))}^q
	\Bigr]\,
	 \lesssim_{n,  p,q,\sigma,T}
	\, {\EE}{\color{OliveGreen}\biggl[
	\|\partial_x u_{0,\delta}\|_{L^2(\TT)}^p
\,
	+\, \biggl|\int_{\TT} u_{0,\delta}\, dx\biggr|^\frac{p(n+2)}{8-2n}
	\biggr]^\frac{(n+2)q}{2p} }
	\\&\qquad+\, \EE\biggl[\biggl(
	\int_{\TT} (u_{0,\delta}-1)-\log(u_{0,\delta})\, dx\biggr)^\frac{p}{2}\biggr]^\frac{(n+2)q}{2p}
	\,+\, 1.
	\end{split}
\end{align}
Next, we verify that
\begin{equation}\label{Eq165}
	\limsup_{\delta\searrow 0} \EE\biggl[\biggl(
	\int_{\TT} (u_{0,\delta}-1)-\log(u_{0,\delta})\, dx\biggr)^\frac{p}{2}\biggr]\,\le \, 
	\EE\biggl[\biggl(
	\int_{\TT} (u_{0}-1)-\log(u_{0})\, dx\biggr)^\frac{p}{2}\biggr]
\end{equation}
to ensure that \eqref{Eq141} and \eqref{Eq167} have a uniformly bounded right-hand side in $\delta$.
To this end, we calculate
\begin{align}\begin{split}
		\label{Eq164}&
		\limsup_{\delta\searrow 0} \EE\biggl[\biggl(
		\int_{\TT} (u_{0,\delta}-1)-\log(u_{0,\delta})\, dx\biggr)^\frac{p}{2}\biggr]
		\\&\quad  \overset{\eqref{Eq132}}{\le }\,
		\limsup_{\delta\searrow 0} \EE\biggl[  \mathbbm{1}_{\bigl\{\|u_0\|_{H^1(\TT)} <e^\frac{1}{\delta}\bigr\} }\biggl(
		\int_{\TT} (u_{0}+\delta-1)-\log(u_0+\delta)\, dx\biggr)^\frac{p}{2}\biggr]
		\\&\qquad
		+\,  
		\limsup_{\delta\searrow 0} \EE\biggl[  \mathbbm{1}_{\bigl\{\|u_0\|_{H^1(\TT)} \ge e^\frac{1}{\delta}\bigr\} }\biggl(
		\int_{\TT} (\delta-1)-\log(\delta)\, dx\biggr)^\frac{p}{2}\biggr]
		\\&\quad\le \, 
		\limsup_{\delta\searrow 0} \EE\biggl[ \biggl(
		\int_{\TT} (u_{0}+\delta-1)-\log(u_0)\, dx\biggr)^\frac{p}{2}\biggr]
		\\&\qquad+\, 
		\limsup_{\delta\searrow 0} \EE\biggl[  \mathbbm{1}_{\bigl\{\|u_0\|_{H^1(\TT)} \ge e^\frac{1}{\delta}\bigr\} }\biggl(
		\int_{\TT} \log\bigl(\tfrac{1}{\delta}\bigr)\, dx\biggr)^\frac{p}{2}\biggr]
		\\&\quad\le\,  \EE\biggl[ \biggl(
		\int_{\TT} (u_{0}-1)-\log(u_0)\, dx\biggr)^\frac{p}{2}\biggr]\,+\, 
		\limsup_{\delta\searrow 0} \EE\biggl[  \mathbbm{1}_{\bigl\{\|u_0\|_{H^1(\TT)} \ge e^\frac{1}{\delta}\bigr\} }\bigl(
		\log\bigl(\log\bigl(
		\|u_0\|_{H^1(\TT)}
		\bigr)\bigr)\bigr)^\frac{p}{2}\biggr]
	\\&\quad
	=\,  \EE\biggl[ \biggl(
	\int_{\TT} (u_{0}-1)-\log(u_0)\, dx\biggr)^\frac{p}{2}\biggr],
\end{split}
\end{align}
where in the last step we used dominated convergence together with  $u_0\in L^p(\Omega, \mathfrak{F}_0{\cO;} H^1(\TT))$. This shows \eqref{Eq165} and following \cite[Corollary 5.2, Proposition 5.4, Lemma 5.5]{dareiotis2021nonnegative}, we obtain a new filtered probability space $(\tilde{\Omega}, \tilde{\mathfrak{A}}, \tilde{\mathfrak{F}}, \tilde{\PP})$ with a family of independent $\tilde{\mathfrak{F}}$-Brownian motions $(\tilde{\beta}_k)_{k\in \ZZ}$ and an equidistributed subsequence $\tilde{u}_\delta\sim \check{u}_\delta$ converging to an $\tilde{\mathfrak{F}}$-adapted process $\tilde{u}$ in $C^{\frac{1}{8}-, \frac{1}{2}-}([0,T]\times \TT)$, $\tilde{\PP}$-almost everywhere. As in \cite[Proposition 5.6]{GessGann2020}, we conclude that
\begin{equation}
	\mathbbm{1}_{\{\tilde{u}_\delta >0 \}}F_\delta(\tilde{u}_\delta) \partial_x^3 \tilde{u}_\delta \, \rightharpoonup\, 	\mathbbm{1}_{\{\tilde{u} >0 \}}F_0(\tilde{u}) \partial_x^3 \tilde{u}
\end{equation}
in $L^2([0,T]\times \TT)$, $\tilde{\PP}$-almost surely. 
Estimate \eqref{Eq137} follows  as in \cite[Proposition 5.6]{dareiotis2021nonnegative} from \eqref{Eq141} and \eqref{Eq165} and consequently also that $\tilde{u}>0$ $\tilde{\PP}\otimes dt\otimes dx$-almost everywhere. As in \cite[Proof of Theorem 2.2]{dareiotis2021nonnegative}, we show that  equation \eqref{Eq136} holds. Lastly, the temporal regularity statement \eqref{Eq144} can be deduced in the same way as \cite[Eq. (5.15)]{dareiotis2021nonnegative}. 
\end{proof}

\begin{proof}[Proof of Proposition \ref{Prop_Entr_Est}]We remark again, that, since we choose one particular $\nu$ in the proof of Theorem \ref{Thm_Existence}, we can drop the $\nu$-dependence in the following estimates.
	Let $\uderh$ be the sequence of solutions to \eqref{Eq108} used in Lemma \ref{Lemma_Ex_ud} to construct the  $\udc$ from the proof of Theorem \ref{Thm_Existence}. By Lemma \ref{Lemma_Entropy_Est}, $\uderh$ {\cO satisfies} the estimate 
	\begin{align}\begin{split}
			\label{Eq163}
		&
		\hat{\EE}\biggl[
		\sup_{t\in [0,T]} \|G_{\delta,\epsilon}(\uderh(t))\|_{L^1(\TT)}^q
		\,+\, 
		\|\partial_x^2 \uderh\|_{L^2([0,T]\times \TT)}^{2q}
		\biggr]\\&\quad \lesssim_{n,q,\sigma, T}\, {\EE}{\color{OliveGreen}\biggl[
		\|G_{\delta,\epsilon}(u_{0,\delta})\|_{L^1(\TT)}^q
		\,+\, \biggl|\int_{\TT} u_{0,\delta}\, dx\biggr|^{2q}
		\biggr]}\,+\, 1.
		\end{split}
	\end{align}
	We use \eqref{Eq46} to estimate
	\begin{align}&
		G_{\delta,\epsilon}(r)\,\lesssim\, \int_r^\infty\int_{r'}^\infty \frac{1}{(r'')^n}\,dr''\, dr' \,\eqsim_n\, r^{2-n},\quad r\ge \delta
	\end{align}
	and since $u_{0,\delta}\ge \delta$ by \eqref{Eq132}, we can use that in \eqref{Eq163} to obtain
	\begin{align}
			&
			\hat{\EE}\biggl[
			\sup_{t\in [0,T]} \|G_{\delta,\epsilon}(\uderh(t))\|_{L^1(\TT)}^q\,+\, 
			\|\partial_x^2 \uderh\|_{L^2([0,T]\times \TT)}^{2q}
			\biggr]\\&\quad \lesssim_{n,q,\sigma, T}\, {\EE}{\color{OliveGreen}\biggl[
			\|u_{0,\delta}^{2-n}\|_{L^1(\TT)}^q
			\,+\,  \biggl|\int_{\TT} u_{0,\delta}\, dx\biggr|^{2q}
			\biggr]}\,+\, 1.
	\end{align}
	The estimate \eqref{Eq162} follows by Fatou's lemma and
	\begin{align}\limsup_{\delta\searrow 0}
		\EE \Bigl[
		\|u_{0,\delta}^{2-n}\|_{L^1(\TT)}^q
		\Bigr] \, \le \, \EE \Bigl[
		\|u_{0}^{2-n}\|_{L^1(\TT)}^q
		\Bigr],
	\end{align}
	which can be derived analogously to \eqref{Eq164}.
\end{proof}
{\color{OliveGreen}\section{Proofs omitted in the previous section}\label{Sec_remaining_proofs}
In this section, we provide the proofs of Lemma \ref{Lemma_EEeps} and Lemma \ref{Lemma_Ex_ud}, i.e., the main steps to deduce existence of solutions to the auxiliary equation \eqref{Eq107}. They were omitted in the previous section, since they are mostly analogous to the proofs of \cite[Theorem 2.2 and Lemma 4.6]{dareiotis2021nonnegative}. We start with the proof of the $(R,\epsilon)$-uniform energy estimate.
}
\begin{proof}[Proof of Lemma \ref{Lemma_EEeps}]To ease notation, we drop the hat notation and write $u$ for $\uderh $ and $\gamma_u = g_R(\|u\|_{C(\TT)})$ during this proof. Analogously to the proof of \cite[Lemma 4.6]{dareiotis2021nonnegative}, we start from {\cO the It\^o expansion of the energy functional} \eqref{Eq25} and estimate the terms on the right-hand side. {\cO In particular, we make use of the technical estimates from Lemmas \ref{LE_1}--\ref{LE_3}.}
		
		\textbf{$(\partial_x u)^4$-Term. } Using Lemma \ref{LE_1} and the interpolation inequality
		\begin{equation}\label{Eq168}
			\|\partial_x u \|_{L^\infty(\TT)}\,\lesssim\, \|\partial_x u\|_{L^2(\TT)}^\frac{1}{2}\|\partial_x^2u\|_{L^2(\TT)}^\frac{1}{2}
		\end{equation} 
		from \cite[Eq. (4.12)]{dareiotis2021nonnegative}, we see that
		\begin{align}&
			\sum_{k\in \ZZ} \int_0^t \gamma_u^2\int_{\TT}
			\sigma_k^2(F_{\delta,\epsilon}''(u))^2(\partial_xu)^4 \,dx\, ds\\&\quad
			\overset{\eqref{Eq121}}{\lesssim_{\sigma}}\, \int_0^t \int_{\TT} (F_{\delta,\epsilon}''(u))^2(\partial_xu)^4\,dx\, ds \\&\quad
			=\, -3\int_0^t \int_{\TT} \biggl(
			\int_1^u(F_{\delta,\epsilon}''(r) )^2\, dr
			\biggr) (\partial_xu)^2 (\partial_x^2 u)\, dx\, ds 
			\\&\quad\overset{\eqref{Eq161}}{
				\lesssim_{\delta,n,\nu}}\,
			\int_0^t 
			\|\partial_x u\|_{L^\infty(\TT)}^2 \|\partial_x^2 u\|_{L^2(\TT)}
			\, ds \\&\quad
			\overset{\eqref{Eq168}}{\lesssim} \,
			\int_0^t 
			\|\partial_x u\|_{L^2(\TT)} \|\partial_x^2 u\|_{L^2(\TT)}^2
			\, ds \\&\quad
			\le \sup_{s\in [0,t]} \|\partial_x u(s)\|_{L^2(\TT)} \times \|\partial_x^2 u\|_{L^2([0,t]\times \TT)}^2 \\&\quad
			\le\, \kappa \sup_{s\in [0,t]} \|\partial_x u(s)\|_{L^2(\TT)}^2\,+\, \tfrac{1}{\kappa}\|\partial_x^2 u\|_{L^2([0,T]\times \TT)}^4
		\end{align}
		for each $\kappa>0$. {\color{OliveGreen} We recall that the latter term is estimated in the entropy estimate \eqref{Eq_Entr_E} from Lemma \ref{Lemma_Entropy_Est}, which allows us to choose $\kappa$ small later on.}

		\textbf{$(\partial_x u)^3$-Term. }
		Using Lemma \ref{LE_2} we retrieve moreover that
		\begin{align}&
			\biggl|\sum_{k\in \ZZ} \int_0^t \gamma_u^2 \int_{\TT} \partial_x (\sigma_k^2)\bigl(
			(F_{\delta,\epsilon}^2)'''(u)\,+\, 4((F_{\delta,\epsilon}')^2)'(u)
			\bigr) (\partial_xu)^3\,dx\, ds\biggr|
			\\&\quad
			\overset{\eqref{Eq121},\eqref{Eq19}, \eqref{Eq20}}{\lesssim_{\delta,n,\nu,\sigma}} \,\int_0^t \int_{\TT} |\partial_xu|^3 \,dx\, ds
			\\&\quad
			\le \,\int_0^t \|\partial_x u\|_{L^\infty(\TT)}^3\, ds
			\\&\quad
			\le \,\int_0^t \|\partial_x u\|_{L^2(\TT)}^\frac{3}{2}
			\|\partial_x^2 u\|_{L^2(\TT)}^\frac{3}{2}\, ds
			\\&\quad
			\overset{\eqref{Eq168}}{\lesssim} \,\int_0^t \|\partial_x u\|_{L^2(\TT)}
			\|\partial_x^2 u\|_{L^2(\TT)}^2, ds
			\\&\quad
			\le \,\sup_{s\in [0,t]} \|\partial_x u(s)\|_{L^2(\TT)}
			\|\partial_x^2 u\|_{L^2([0,t]\times\TT)}^2
			\\&\quad
			\le \,\kappa \sup_{s\in [0,t]} \|\partial_x u(s)\|_{L^2(\TT)}^2\,+\,
			\tfrac{1}{\kappa}\|\partial_x^2 u\|_{L^2([0,T]\times\TT)}^4,
		\end{align}
		 for each $\kappa>0$.

		\textbf{$(\partial_x u)^2$-Term. } We use Lemma \ref{LE_3}, the embedding $H^1(\TT)\hookrightarrow L^\infty(\TT)$, Poincar\'e's inequality and the conservation of mass of $u$ to conclude furthermore
		\begin{align}&
			\biggl|	\sum_{k\in \ZZ} \int_0^t \gamma_u^2 \int_{\TT} \bigl(8((\partial_x\sigma_k)^2 - \sigma_k (\partial_x^2\sigma_k)) (F_{\delta,\epsilon}'(u))^2\,+\, (\partial_x^2(\sigma_k^2)) (F_{\delta,\epsilon}^2)''(u) \bigr) (\partial_x u)^2\,dx\,ds\biggr|\\&\quad
			\overset{\eqref{Eq121}}{\lesssim_\sigma}\, \int_0^t\int_{\TT}
			( (F_{\delta,\epsilon}'(u))^2 \,+\,|(F_{\delta,\epsilon}^2)''(u)|)(\partial_x u)^2
			\,dx\, ds \\&\quad
			\overset{\eqref{Eq21}, \eqref{Eq22}}{\lesssim_{n,\nu}}\,
			\int_0^t \int_{\TT} (|u|^{n-2}+1) (\partial_x u)^2\,dx\, ds \\&\quad
			\le \, \int_0^t (\|u\|_{L^\infty(\TT)}^{n-2}+1) \|\partial_x u\|_{L^2(\TT)}^2\, ds \\&\quad
			\lesssim \int_0^t {\cO\biggl(\|\partial_x u\|_{L^2(\TT)}^{n-2} \,+\, \biggl|\int_{\TT} u(0) \, dx\biggr|^{n-2}\,+\, 1 \biggr)}\|\partial_x^2 u\|_{L^2(\TT)}^2\, ds \\&\quad
			\le \,
			{\cO\biggl(\sup_{s\in [0,t]} \|\partial_x u(s)\|_{L^2(\TT)}^{n-2} \,+\, \biggl|\int_{\TT}u(0)\, dx \biggr|^{n-2}\,+\, 1 \biggr)} \|\partial_x^2 u\|_{L^2([0,t]\times \TT)}^2 \\&\quad
			\overset{\frac{n-2}{2}+\frac{4-n}{2}=1}{\lesssim_n }\, \kappa \sup_{s\in [0,t]} \|\partial_x u(s)\|_{L^2(\TT)}^2 \,+\,  \kappa {\cO\biggl|\int_{\TT}u(0)\, dx\biggr|^2} \,+\,\kappa \,+\, (\tfrac{1}{\kappa})^\frac{n-2}{4-n}\|\partial_x^2 u\|_{L^2([0,T]\times\TT)}^\frac{4}{4-n},
		\end{align}
		where again $\kappa>0$.
		
		\textbf{$(\partial_x u)^0$-Term. } 
		We use once more Lemma \ref{LE_3}, the embedding $H^1(\TT)\hookrightarrow L^\infty(\TT)$, Poincar\'e's inequality and the conservation of mass of $u$ to obtain		
		\begin{align}&\biggl|\sum_{k\in \ZZ}\int_0^t \gamma_u^2 \int_{\TT} (4\sigma_k\partial_x^4 \sigma_k -\partial_x^4(\sigma_k^2)) F_{\delta,\epsilon}^2(u)\,dx\, ds\biggr|\\&\quad\overset{\eqref{Eq121},\eqref{Eq117}}{\lesssim_\sigma}\, \int_0^t \int_{\TT} F_{\delta,\epsilon}^2(u)\,dx\, ds\\&\quad
			\overset{\eqref{Eq23}}{\lesssim_{n}} \,\int_0^t \int_{\TT}
			|u|^n\,+\, 1\,dx\, ds\\&\quad
			\le \, \int_0^t  \|u\|_{L^\infty(\TT)}^n\,+\,1 \, ds 
			\\&\quad \lesssim\int_0^t {\cO\biggl(  \|\partial_x u\|_{L^2(\TT)}\,+\, \biggl|\int_{\TT} u(0) \, dx\biggr|\biggr)^n} \,+\,1 \, ds 
			\\&\quad\lesssim_{n,T}\,   \int_0^t \|\partial_x u\|_{L^2(\TT)}^{n-2} \|\partial_x^2 u \|_{L^2(\TT)}^2\, ds\,+\,{\cO \biggl|\int_{\TT}u(0) \, dx \biggr|^n}\,+\,1\\&\quad
			\le \,  \sup_{s\in [0,t]} \|\partial_x u(s)\|_{L^2(\TT)}^{n-2}\|\partial_x^2 u \|_{L^2([0,t]\times \TT)}^2\,+\, {\cO \biggl|\int_{\TT}u(0) \, dx \biggr|^n} \,+\,1
			\\&\quad
			\overset{\frac{n-2}{2}+\frac{4-n}{2}=1}{\le }\,  \kappa\sup_{t'\in [0,t]} \|\partial_x u\|_{L^2(\TT)}^{2}\,+\, {\cO \biggl|\int_{\TT}u(0) \, dx \biggr|^n} \,+\,1\,+\, (\tfrac{1}{\kappa})^\frac{{\cO n-2}}{4-n}\|\partial_x^2 u \|_{L^2([0,T]\times \TT)}^\frac{{\cO 4}}{4-n}{\cO,}
		\end{align}
		for $\kappa>0$.
		
		\textbf{Closing the estimate.} We insert the previous estimates in the It\^o expansion \eqref{Eq25} and conclude that 
		\begin{align}&
			\tfrac{1}{2}\|\partial_x u(t)\|_{L^2(\TT)}^2\,+\, \int_0^t \int_{\TT} F_{\delta,\epsilon}^2(u)(\partial_x^3 u)^2\,dx\, ds \\&\quad
			\le \,\tfrac{1}{2}\|\partial_x u(0)\|_{L^2(\TT)}^2\,+\, |M(t)|\\&\qquad
			+\, C_{\delta, n,\nu,\sigma , T} \biggl[
			\kappa \sup_{s\in [0,t]} \|\partial_x u(s)\|_{L^2(\TT)}^2 \,+\, \tfrac{1}{\kappa} \|\partial_x^2 u \|_{L^2([0,T]\times \TT)}^4
			\,+\, (\tfrac{1}{\kappa})^\frac{n-2}{4-n} \|\partial_x^2 u \|_{L^2([0,T]\times \TT)}^\frac{4}{4-n}
			\biggr] \\&\qquad
			+\, C_{\delta, n,\nu,\sigma , T} {\cO\biggl[  \biggl|\int_{\TT} u(0)\, dx\biggr|^n \,+\, \kappa \biggl|\int_{\TT} u(0)\, dx\biggr|^2\,+\,\kappa\,+\, 1
			\biggr]},
		\end{align}
		where $M(t)$ is given by \eqref{Eq59}.
		Choosing $\kappa$ small enough, taking the supremum in $t$ until the stopping time 
		\begin{equation}\label{Eq64}
			\tau_m \,=\, \inf \biggl\{ t\in [0,T]:\, \sup_{s\in [0,t]} \| \partial_x u(s)\|_{L^2(\TT)}^2\,+\,\int_0^{t}\int_{\TT} (F_{\delta,\epsilon} (u))^2 (\partial_x^3 u)^2\, dx\, ds \,\ge \, m \biggr\}
		\end{equation} and absorbing the intermediate powers, we proceed to
		\begin{align}&
			\tfrac{1}{4}\sup_{t\in [0,\tau_m]}\|\partial_x u(t)\|_{L^2(\TT)}^2\,+\, \int_0^{\tau_m} \int_{\TT} F_{\delta,\epsilon}^2(u)(\partial_x^3 u)^2\,dx\, ds \\&\quad 
			\le \|\partial_xu(0)\|_{L^2(\TT)}^2\,+\, 2\sup_{t\in [0,\tau_m]}|M(t)| \\&\qquad +\, 
			C_{\delta, n,\nu,\sigma , T} {\cO\biggl[
			\biggl|\int_{\TT} u(0)\, dx\biggr|^n \,+\, \|\partial_x^2 u \|_{L^2([0,T]\times \TT)}^4\,+\,	1 
			\biggr]}.
		\end{align}
		We raise both sides to the power $\frac{p}{2}$ and use the {\cO Burkholder--Davis--Gundy} inequality to conclude
		\begin{align}&
			\EE\biggl[\sup_{t\in [0,\tau_m]}\|\partial_x u(t)\|_{L^2(\TT)}^p \,+\,\biggl(\int_0^{\tau_m} \int_{\TT} F_{\delta,\epsilon}^2(u)(\partial_x^3 u)^2\,dx\, ds\biggr)^\frac{p}{2}\biggr]
			\\&\quad
			\lesssim_{\delta,n,\nu,\sigma,p, T} 	\EE\Bigl[\|\partial_x u(0)\|_{L^2(\TT)}^p \,+\, 
			\langle M \rangle_{\tau_m}^\frac{p}{4}
			\Bigr]\\&\qquad
			+\, \EE{\cO\biggl[
			\biggl|\int_{\TT} u(0)\, dx\biggr|^\frac{pn}{2}\,+\, \|\partial_x^2 u \|_{L^2([0,T]\times \TT)}^{2p}
			\biggr]}	+\, 1.
		\end{align}
		From \eqref{Eq166}, we deduce that
		\begin{align} &
			\EE \Bigl[
			\langle M \rangle_{\tau_m}^\frac{p}{4}
			\Bigr] \,\lesssim_{\sigma,p} \, \EE\biggl[\biggl(
			\int_0^{\tau_m} \int_{\TT} (F_{\delta,\epsilon}(u))^2 (\partial_x^3 u)^2  \,dx \, ds\biggr)^\frac{p}{4}
			\biggr] \\&\quad \le\, \kappa \EE\biggl[\biggl(
			\int_0^{\tau_m} \int_{\TT} (F_{\delta,\epsilon}(u))^2 (\partial_x^3 u)^2  \,dx \, ds\biggr)^\frac{p}{2}
			\biggr]\,+\, \frac{1}{\kappa}
		\end{align}
		and therefore
		\begin{align}&
			\EE\biggl[\sup_{t\in [0,\tau_m]}\|\partial_x u(t)\|_{L^2(\TT)}^p \,+\, \biggl(\int_0^{\tau_m} \int_{\TT} F_{\delta,\epsilon}^2(u)(\partial_x^3 u)^2\,dx\, ds\biggr)^\frac{p}{2}\biggr]
			\\&\quad
			\lesssim_{\delta,n,\nu,\sigma,p, T} 	\EE{\cO\biggl[\|\partial_x u(0)\|_{L^2(\TT)}^p
			+\, 
			\biggl|\int_{\TT} u(0)\, dx\biggr|^\frac{pn}{2}\,+\, \|\partial_x^2 u \|_{L^2([0,T]\times \TT)}^{2p}
			\biggr]}	+\, 1
		\end{align}
		by choosing $\kappa>0$ again small.
		Letting $m\to \infty$, applying Fatou's lemma and using that $u(0)\sim u_0$, we obtain \eqref{Eq_En_E1}.
\end{proof}{\cO
As a second step to prove existence for \eqref{Eq107}, we take the limits $R\to \infty$ and $\epsilon\searrow 0$ in \eqref{Eq108} by following \cite[Proposition 4.7, Section 5]{dareiotis2021nonnegative}.}
\begin{proof}[Proof of Lemma \ref{Lemma_Ex_ud}]
	By Lemma \ref{Lemma_Ex_uder}, there exists for each $R$ and $\epsilon$ a  martingale solution 
	\[
	\bigl\{
	(\hat{\Omega}, \hat{\mathfrak{A}},\hat{\mathfrak{F}},\hat{\PP} ), \,(\hat{\beta}^{(k)})_{k\in \ZZ},\, \uderh
	\bigr\}
	\]
	to \eqref{Eq108} in the sense of Definition \ref{defi_sol_non_degenerate} with initial value $u_0$. 
	Due {\cO to} Lemma \ref{Lemma_Entropy_Est} and Lemma \ref{Lemma_EEeps}, $\uderh$ satisfies the estimate
	\begin{align}\begin{split}\label{Eq169}
			&
			\hat{\EE}\biggl[
			\sup_{t\in [0,T]} \|\partial_x\uderh(t))\|_{L^2(\TT)}^p\,+\, 
			\sup_{t\in [0,T]} \|G_{\delta,\epsilon}(\uderh(t))\|_{L^1(\TT)}^p
			\biggr]
			\\&\qquad +\, \hat{\EE} \biggl[\biggl(
			\int_0^T\int_{\TT}
			F_{\delta,\epsilon}^2(\uderh)(\partial_x^3 \uderh)^2 
			\,dx\, dt\biggr)^\frac{p}{2}\,+\, 
			\|\partial_x^2 \uderh\|_{L^2([0,T]\times \TT)}^{2p}
			\biggr]\\&\quad \lesssim_{\delta,n,\nu,\sigma,p, T}\,{\EE}{\cO\biggl[
			\|\partial_x u_0\|_{L^2(\TT)}^p\,+\, \biggl|\int_{\TT} u_0\, dx\biggr|^{2p}\,+\,
			\|G_{\delta,\epsilon}(u_0)\|_{L^1(\TT)}^p
			\biggr]}\,+\, 1,
		\end{split}
	\end{align}
	which is uniform in $R$ and $\epsilon$ by \cite[Remark 4.8]{dareiotis2021nonnegative}.
	The limiting procedures $R\to \infty$ and $\epsilon\searrow 0$ follow along the lines of \cite[Proposition 4.7, Section 5]{dareiotis2021nonnegative}. Indeed, while our mobility function $F_\delta$ is not homogeneous, the approximations $F_{\delta,\epsilon}$ still satisfy the growth bounds \eqref{Eq113} and \eqref{Eq116}, which is sufficient to estimate the nonlinear terms appearing in \eqref{Eq108} and to identify their limits. 
	
	{\color{OliveGreen}Finally,}  estimate \eqref{Eq170} follows by using lower semi-continuity of the norm with respect to weak-* convergence and Fatou's lemma in \eqref{Eq169}.
	{\color{OliveGreen} Indeed, we use that
		$F_{\delta,\epsilon}(\check{u}_{\delta,\epsilon}) \partial_x^3\check{u}_{\delta,\epsilon} \rightharpoonup \mathbbm{1}_{\{\udc> 0\}} F_\delta(\check{u}_\delta) \partial_x^3 \check{u}_{\delta} $, as $\epsilon\searrow 0$  for an appropriate subsequence  in $L^p (\check{\Omega} ; L^2([0,T]\times \TT))$ by \cite[Proposition 5.6]{dareiotis2021nonnegative}. Thus, it follows
		\begin{align}
			& \check{\EE}\Bigl[\|
			\mathbbm{1}_{\{\udc> 0\}}F_{\delta}(\udc)\partial_x^3(  \udc)
			\|_{L^2([0,T]\times \TT)}^p \Bigr] \,\le \, \liminf_{\epsilon\searrow 0} 
			\check{\EE}\Bigl[\|
			F_{\delta}(\check{u}_{\delta,\epsilon})\partial_x^3(  \check{u}_{\delta,\epsilon})
			\|_{L^2([0,T]\times \TT)}^p \Bigr] ,
		\end{align}
		and the right-hand side can be estimated using the same argument as $R\to \infty$ in \eqref{Eq169}. Analogously, we bound
		\begin{align}
			\check{\EE}\biggl[
			\sup_{t\in [0,T]} \|\partial_x\udc(t)\|_{L^2(\TT)}^p \biggr] \quad \text{ and }
			\quad 
			\check{\EE}\Bigl[
			\|\partial_x^2 \udc\|_{L^2([0,T]\times \TT)}^{2p}
			\Bigr],
		\end{align}
		based on the convergence $\check{u}_{\delta,\epsilon}\rightharpoonup^* \check{u}_\delta$ in $L^p(\check{\Omega}; L^\infty(0,T;H^1(\TT)))$ and $L^{2p}(\check{\Omega}; L^2([0,T]\times \TT))$, see again \cite[Proposition 5.6]{dareiotis2021nonnegative}. For the remaining term involving $G_{\delta}(\udc)$ one has to argue slightly differently. Namely, we use $\check{u}_{\delta,\epsilon}\to \udc$ in $C([0,T]\times \TT)$, almost surely, by \cite[Proposition 5.4]{dareiotis2021nonnegative} together with Fatou's lemma and definitions \eqref{Eq_Gd} and \eqref{Eq_Gde}, to deduce that
		\begin{align}&
			\liminf_{\epsilon\searrow 0} G_{\delta,\epsilon}(\check{u}_{\delta,\epsilon})\,=\, 
			\liminf_{\epsilon\searrow 0} \int_{\RR} \mathbbm{1}_{(\check{u}_{\delta,\epsilon},\infty)}(r') \int_{\RR} 
			\frac{\mathbbm{1}_{(r',\infty)}(r'')}{F_{\delta,\epsilon}^2(r'')}
			\,dr''\, dr'
			\\&\quad \ge \, 
			\int_{\RR} \biggl(\liminf_{\epsilon\searrow 0} \mathbbm{1}_{(\check{u}_{\delta,\epsilon},\infty)}(r')\biggr) \biggl(\liminf_{\epsilon\searrow 0} \int_{\RR} 
			\frac{\mathbbm{1}_{(r',\infty)}(r'')}{F_{\delta,\epsilon}^2(r'')}
			\,dr''\biggr)\, dr'
			\\&\quad \ge \, 
			\int_{\RR}  \mathbbm{1}_{(\check{u}_{\delta},\infty)}(r') \int_0^\infty
			\frac{\mathbbm{1}_{(r',\infty)}(r'')}{F_{\delta}^2(r'')}
			\,dr''\, dr' \,=\, G_\delta(\check{u}_\delta).
		\end{align}
		It follows that
		\begin{align}&
			\check{\EE}\biggl[\sup_{t\in[0,T]}\int_{\TT} G_{\delta}(\check{u}_{\delta}(t))\,dx\biggr] \, \le \,
			\check{\EE}\biggl[ \sup_{t\in[0,T]} \int_{\TT}
			\liminf_{\epsilon\searrow 0} G_{\delta,\epsilon}(\check{u}_{\delta,\epsilon}(t))\,dx\biggr]
			\\&
			\le \, 
			\check{\EE}\biggl[ \sup_{t\in[0,T]}
			\liminf_{\epsilon\searrow 0} \int_{\TT} G_{\delta,\epsilon}(\check{u}_{\delta,\epsilon}(t))\,dx\biggr] \,\le \,
			\liminf_{\epsilon\searrow 0}	\check{\EE}\biggl[	 \sup_{t\in[0,T]}
			\int_{\TT} G_{\delta,\epsilon}(\check{u}_{\delta,\epsilon}(t))\,dx\biggr] ,
		\end{align}
		again by Fatou's lemma. An estimate on the right-hand side can be obtained by letting $R\to \infty$ in \eqref{Eq169} using the same line of arguments.
	}
\end{proof}
\appendix
\section{Technical estimates on the approximate mobilities and functionals}\label{App_A}
{\color{OliveGreen}In this appendix we collect the vital estimates on the approximating mobilities and the correspondingly defined functionals, as defined in Section \ref{Sec_approx_functions}. They are at the heart of the various applications of the stochastic compactness method, since they result in bounds on the approximate solutions to the stochastic thin-film equation, which are uniform in $(R,\epsilon)$ and $\delta$, respectively. We focus first on the estimates, which are employed in the proof of the approximate $\log$-entropy estimate Lemma \ref{Lemma_log_Entr_Approx_level}.}

\begin{lemma}\label{Lemma_prop_Jd}It holds that
	\begin{align}&\label{Eq32}
		J_{\delta,\epsilon}^+(r)\,\lesssim_{n,\nu} \, r_+^{n-2},
		\\& \label{Eq37}
		J_{\delta,\epsilon}^-(r)\,\lesssim_{\delta,n,\nu} \, r_-,
		\\& \label{Eq36}
		|I_{\delta,\epsilon}(r)| \,\lesssim_{n,\nu} \, r_+^{n-2}\,+\, C_{\delta,n,\nu} r_-.
	\end{align}
\end{lemma}
\begin{proof}
	Ad \eqref{Eq32}. From \eqref{Eq_Jde} and \eqref{Eq123}, we  conclude that
	\begin{align}\label{Eq1}
		{\color{OliveGreen}0\,\le \,} J_{\delta,\epsilon}(r)\,\lesssim_{n,\nu}\, \int_0^r \int_{r'}^\infty (K_\epsilon(r''))^{n-4}\, dr'' \, dr'
		\,\le \,
		\int_0^r \int_{r'}^\infty (r'')^{n-4}\, dr'' \, dr' \, \eqsim_{n}\, r^{n-2}, \quad r\ge 0.
	\end{align}
	{\color{OliveGreen}At the same time we see from \eqref{Eq_Jde} that $J_{\delta,\epsilon}(r)\le 0$ for $r<0$.}

	Ad \eqref{Eq37}. We observe that 
	\begin{align}
		0\, \ge \, J_{\delta,\epsilon}(r)\, \ge  \,  - \int_r^0 \int_{-\infty}^\infty (F_{\delta,\epsilon}''(r''))^2\, dr''\, dr'\,\gtrsim_{\delta,n,\nu}\, r,\quad r\le  0
	\end{align}
	{\cO follows},
	as soon as we can verify 
	\begin{equation}\label{Eq40}
		\int_{-\infty}^\infty (F_{\delta,\epsilon}''(r))^2\, dr\,\lesssim_{\delta,n,\nu}\, 1.
	\end{equation}
	To this end, we use that  $F_{\delta,\epsilon}$ is an even function {\cO to} obtain
	\begin{equation}\label{Eq41}
		\int_{-\infty}^\infty (F_{\delta,\epsilon}''(r))^2\, dr\,\eqsim\, \int_0^\infty (F_{\delta,\epsilon}''(r))^2\, dr\,=\, 
		\int_0^\frac{\delta}{2} (F_{\delta,\epsilon}''(r))^2\, dr\,+\, \int_\frac{\delta}{2}^\infty (F_{\delta,\epsilon}''(r))^2\, dr.
	\end{equation}
	Using \eqref{Eq123}, we can estimate the latter integral by
	\begin{align}\label{Eq42}
		\int_\frac{\delta}{2}^\infty (F_{\delta,\epsilon}''(r))^2\, dr\,\lesssim_{n,\nu} \, \int_\frac{\delta}{2}^\infty (K_\epsilon(r))^{n-4}\, dr\,\le\, 
		\int_\frac{\delta}{2}^\infty r^{n-4}\, dr\, \eqsim \, \bigl(\tfrac{\delta}{2}\bigr)^{n-3}.
	\end{align}
	From \eqref{Eq34} and \eqref{Eq33}, we conclude furthermore that
	\begin{align}&\label{Eq130}
		{\cO|}F_{\delta}'(r){\cO|}\, \lesssim_{n,\nu} \frac{r^{\frac{\nu}{2}-1}}{\delta^l },\quad r>0,
		\\&\label{Eq131}
		|F_{\delta}''(r)|\, \lesssim_{n,\nu} \frac{r^{\frac{\nu}{2}-2}}{\delta^l }, \quad r>0.
	\end{align}
	Using {\cO\eqref{Eq39}--\eqref{Eq45}}, we proceed to
	\begin{align}\label{Eq133}
		F_{\delta,\epsilon}''(r)\,\lesssim_{n,\nu}\,
		\delta^{-l}\biggl(
		(K_\epsilon(r))^{\frac{\nu}{2}-2}\,+\, \frac{(K_\epsilon(r))^{\frac{\nu}{2}-1}}{K_\epsilon(r)}
		\biggr)\, \lesssim_{\delta,n,\nu}\,
		(K_\epsilon(r))^{\frac{\nu}{2}-2}.
	\end{align}
	Using this estimate, we derive that
	\begin{equation}\label{Eq125}
		\int_0^{\frac{\delta}{2}} (F_{\delta,\epsilon}''(r))^2\, dr\,\lesssim_{\delta,n,\nu}\, \int_0^{\frac{\delta}{2}} (K_\epsilon(r))^{\nu-4}\, dr\, \le \,
		\int_0^{\frac{\delta}{2}} r^{\nu-4}\, dr
		\,\eqsim_\nu\, \bigl(\tfrac{\delta}{2}\bigr)^{\nu-3}.
	\end{equation}
	Combining this with \eqref{Eq41} and \eqref{Eq42}, we conclude \eqref{Eq40}. {\color{OliveGreen}The claimed estimate \eqref{Eq37} follows by using that $J_{\delta,\epsilon}(r)\ge 0$ for $r>0$, see \eqref{Eq1}.}

	Ad \eqref{Eq36}. By \eqref{Eq_Lde} and {\cO the definition \eqref{Eq127} of $I_{\delta,\epsilon}$} it holds
	\[
	I_{\delta,\epsilon}(r)\,=\, 
	\int_0^r
	J_{\delta,\epsilon}(r') \frac{F_{\delta,\epsilon}'(r')}{ F_{\delta,\epsilon}(r')}
	\, dr',
	\]
	where we can express the fraction as
	\begin{align}
		\frac{F_{\delta,\epsilon}'(r)}{ F_{\delta,\epsilon}(r)}\,=\,  \frac{F_\delta'(K_\epsilon(r))}{F_\delta(K_\epsilon(r))}  K_\epsilon'(r){\cO,}
	\end{align}
	using the chain rule.
	From \eqref{Eq_Fd} and \eqref{Eq34} we deduce 
	\[
	\frac{F_\delta'(r)}{F_\delta(r)}\,\lesssim_{n,\nu}\, r^{-1}, \quad r>0,
	\]
	{\cO so} that
	\begin{align}
		\frac{F_{\delta,\epsilon}'(r)}{ F_{\delta,\epsilon}(r)}\,\lesssim_{n,\nu}\,  \frac{1}{K_\epsilon(r)}
	\end{align}
	by  \eqref{Eq44}.
	Using \eqref{Eq32}, we observe that
	\begin{align}&
		|I_{\delta,\epsilon}(r)|\,\lesssim_{n,\nu}
		\int_0^r 
		\frac{J^+_{\delta,\epsilon}(r')}{K_\epsilon(r')}
		\, dr'\,\lesssim_{n,\nu}\, 
		\int_0^r 
		\frac{(r')^{n-2}}{K_\epsilon(r')}
		\, dr'
		\\&\quad 
		\le \, 
		\int_0^r 
		{(r')^{n-3}}
		\, dr'\, \eqsim_{n} \, r^{n-2}, \quad r\ge 0.
	\end{align}
	For $r<0$, we use instead \eqref{Eq37} to conclude
	\begin{align}
		|I_{\delta,\epsilon}(r)|\, \lesssim_{n,\nu}\,
		\int_r^0 \frac{J_{\delta,\epsilon}^-(r')}{K_\epsilon(r')}\, dr'
		\, \lesssim_{\delta, n,\nu} \, 
		\int_r^0 \frac{r_-'}{K_\epsilon(r')}\, dr'\, \le\, r_-, \quad r\,<\, 0.
	\end{align}
\end{proof}

\begin{lemma}\label{Lemma_prop_Int}
	It holds that
	\begin{align}
		\label{Eq43}
		\biggl|\int_0^r
		L_{\delta,\epsilon}''(r')F_{\delta,\epsilon}(r')\, dr'\biggr|\, \lesssim_{n,\nu} \, r_+^{\frac{n}{2}-1}\,+\, C_{\delta,n,\nu} \Bigl(r_-^{2-\frac{n}{2}}\,+\, r_-^{2-\frac{\nu}{2}}\Bigr).
	\end{align}
\end{lemma}
\begin{proof}
	We notice that
	\begin{align}\label{Eq140}
		L_{\delta,\epsilon}''(r)F_{\delta,\epsilon}(r)\,=\, \frac{J_{\delta,\epsilon}(r)}{F_{\delta,\epsilon}(r)}
	\end{align}
	by \eqref{Eq_Lde} and estimate this term by distinguishing different cases of $r$. 
	
	For $r\ge \delta$ we have that
	\begin{equation}\label{Eq46}
		F_{\delta,\epsilon}(r)\, \ge\,
		F_{\delta}(r)\,\ge \,\frac{r^\frac{n+\nu}{2}}{2 r^\frac{\nu}{2}}\,\eqsim\, r^\frac{n}{2}, \quad r\ge \delta,
	\end{equation}
	by \eqref{Eq_Fd}.
	Hence, using \eqref{Eq32}, we conclude that
	\begin{equation}\label{Eq139}
		\frac{J_{\delta,\epsilon}(r)}{F_{\delta,\epsilon}(r)}\, \lesssim_{n,\nu} \, \frac{r^{n-2}}{r^\frac{n}{2}}
		\, \le \, r^{\frac{n}{2}-2}, \quad r\ge \delta.
	\end{equation}
	
	For $r\in (0,\delta)$, we provide another estimate on $J_{\delta,\epsilon}(r)${\color{OliveGreen}. Namely, we recall that in \eqref{Eq133} we proved that
	\[
	F_{\delta,\epsilon}''(r)\,\lesssim_{n,\nu} \, \frac{(K_\epsilon(r))^{\frac{\nu}{2}-2}}{\delta^l} 
	\]
	for any $r\in \RR$. }
	Hence, using \eqref{Eq135}, \eqref{Eq_Jde} and \eqref{Eq134}, we compute that
	\begin{align}\begin{split}\label{Eq10}
			&
		J_{\delta,\epsilon}(r)\,\lesssim_{n,\nu} \,\delta^{-l}\int_0^r \int_{r'}^\infty
		(K_\epsilon(r''))^{\frac{\nu+n}{2}-4}
		\, dr'' \, dr' \\&\quad \le \, 
		\delta^{-l}\int_0^r \int_{r'}^\infty
		(r'')^{\frac{\nu+n}{2}-4}
		\, dr'' \, dr'
		\,
		\eqsim_{n,\nu}
		\,   \delta^{-l}
		r^{\frac{\nu+n}{2}-2},\quad r\ge 0.
		\end{split}
	\end{align}
	We moreover have that
	\begin{equation}\label{Eq47}
		F_{\delta,\epsilon}(r) \ge F_\delta(r)\, \ge \, \frac{r^\frac{n+\nu}{2}}{2 \delta^l r^\frac{n}{2}}\,\eqsim\, \delta^{-l} r^\frac{\nu}{2},\quad  r \in (0,\delta).
	\end{equation}
	Consequently, we arrive also in this case at
	\begin{equation}\label{Eq138}
		\frac{J_{\delta,\epsilon}(r)}{F_{\delta,\epsilon}(r)}\, \lesssim_{n,\nu} \, \frac{ \delta^{-l}r^{\frac{\nu+n}{2}-2}}{\delta^{-l}r^\frac{\nu}{2}}
		\, \le \, r^{\frac{n}{2}-2},\quad r\in (0,\delta).
	\end{equation}
	Hence, using \eqref{Eq140}, \eqref{Eq139} and \eqref{Eq138}, we deduce that 
	\begin{align}
		\biggl|\int_0^r
		L_{\delta,\epsilon}''(r')F_{\delta,\epsilon}(r')\, dr'\biggr|\, \lesssim_{n,\nu} \, \int_0^r
		(r')^{\frac{n}{2}-2}\, dr'\, \eqsim_{n}\, r^{\frac{n}{2}-1}, \quad r\ge 0.
	\end{align}
	
	For $r<0$, we use that $F_{\delta,\epsilon}$ is an even function to conclude that
	\begin{align}&\label{Eq150}
		F_{\delta,\epsilon}(r)\, \gtrsim \, r_-^\frac{n}{2},\quad r\le-\delta,
		\\&\label{Eq151}
		F_{\delta,\epsilon}(r)\, \gtrsim \, \delta^{-l} r_-^\frac{\nu}{2},\quad r\in (-\delta,0) 
	\end{align}
	from \eqref{Eq46} and \eqref{Eq47}.
	Invoking \eqref{Eq37}, we arrive at 
	\begin{equation}\label{Eq189}
		\frac{|J_{\delta,\epsilon}(r)|}{F_{\delta,\epsilon}(r)}\,\lesssim_{\delta,n,\nu} \, r_-^{1-\frac{n}{2}}\,+\, r_-^{1-\frac{\nu}{2}} ,\quad r<0.
	\end{equation}
	Hence, integration together with \eqref{Eq140} yields
	\[
	\biggl|\int_r^0
	L_{\delta,\epsilon}''(r')F_{\delta,\epsilon}(r')\, dr'\biggr|\, \lesssim_{\delta, n,\nu} \, \int_r^0
	(r_-')^{1-\frac{n}{2}}\,+\, (r_-')^{1-\frac{\nu}{2}}\, dr'\,\eqsim_{n,\nu}\,
	r_-^{2-\frac{n}{2}}\,+\, r_-^{2-\frac{\nu}{2}},\quad r\,<\, 0.
	\]
\end{proof}
\begin{lemma}\label{Lemma_Lde}
	It holds that
	\begin{align}&
		\label{Eq142}
		L_{\delta,\epsilon}^+(r)\,\le\, (r-1)\,-\, \log(r), \quad r\ge \delta,
		\\& \label{Eq143}
		L_{\delta,\epsilon}^{-}(r)\, \lesssim_{\delta,n,\nu} \bigl(G_{\delta,\epsilon}(r)\,+\, 1\bigr)r_-.
	\end{align}
\end{lemma}
\begin{proof}
	Ad \eqref{Eq142}.
	We let $r\ge \delta$ and use \eqref{Eq32}  and \eqref{Eq46} to conclude
	\[L_{\delta,\epsilon}(r)\,=\,
	\int_1^r\int_1^{r'} \frac{J_{\delta,\epsilon}(r'')}{F_{\delta,\epsilon}^2(r'')}\, dr''\, dr'\, \lesssim_{n,\nu} 
	\int_1^r\int_1^{r'} (r'')^{-2}\, dr''\, dr'\,=\, (r-1)\,-\,\log(r),\quad r\ge \delta.
	\]
	
	Ad \eqref{Eq143}. We point out that $L_{\delta,\epsilon}(r)\ge 0$ for $r\ge 0$, since it is convex on $(0,\infty)$ by {\cO\eqref{Eq_Jde} and} \eqref{Eq_Lde} with a minimum at $r=1$. For $r<0$, we use that
	\begin{align}&
		L_{\delta,\epsilon}(r)\,=\, 
		\int_{r}^{1}\int_{r'}^1
		\frac{J_{\delta,\epsilon}(r'')}{F_{\delta,\epsilon}^2(r'')}
		\, dr''\, dr'
		\, \ge\, -
		\int_r^1 \int_{r'}^1 \frac{J_{\delta,\epsilon}^-(r'')}{F_{\delta,\epsilon}^2(r'')}\, dr''\, dr'\\&\quad 
		=\, -\, 	\int_r^0 \int_{r'}^0 \frac{J_{\delta,\epsilon}^-(r'')}{F_{\delta,\epsilon}^2(r'')}\, dr''\, dr'\,
		\ge \,r \int_{-\infty}^0 \frac{J_{\delta,\epsilon}^-(r'')}{F_{\delta,\epsilon}^2(r'')}\,  dr'',\quad  r<0{\cO,}
	\end{align}
	and consequently
	\begin{align}\label{Eq158}
		L_{\delta,\epsilon}^-(r)\, \le \, 
		r_- \int_{-\infty}^0 \frac{J_{\delta,\epsilon}^-(r')}{F_{\delta,\epsilon}^2(r')}\,  dr'
		\,=\, r_- \biggl(
		\int_{-\infty}^{-\delta} \frac{J_{\delta,\epsilon}^-(r')}{F_{\delta,\epsilon}^2(r')}\,  dr'\,+\, 
		\int_{-\delta}^0 \frac{J_{\delta,\epsilon}^-(r')}{F_{\delta,\epsilon}^2(r')}\,  dr'\biggr),\quad r<0
		.
	\end{align}
	We use \eqref{Eq37} and \eqref{Eq150} to estimate 
	\begin{align}\label{Eq155}
		\int_{-\infty}^{-\delta} \frac{J_{\delta,\epsilon}^-(r')}{F_{\delta,\epsilon}^2(r')}\,  dr'
		\lesssim_{\delta,n,\nu} \, \int_{-\infty}^{-\delta} \frac{r_-'}{(r_-')^n}\,  dr'
		\,\eqsim_{\delta,n}\,1 
	\end{align}
	For the remaining part, we deduce form \eqref{Eq47} and the fact that $F_{\delta,\epsilon}$ is even
	\[
	F_{\delta,\epsilon}(r)\,\gtrsim\, \delta^{-l} (K_\epsilon(r))^\frac{\nu}{2} ,\quad r\in (-\delta,0).
	\]
	Thus, using {\color{OliveGreen}
	\eqref{Eq145}}
	and again \eqref{Eq37}, we obtain
	\begin{align}
		\int_{-\delta}^0 \frac{J_{\delta,\epsilon}^-(r')}{F_{\delta,\epsilon}^2(r')}\,  dr'\,
		\lesssim_{\delta,n,\nu} \, \int_{-\delta}^0 \frac{r_-'}{(K_\epsilon(r'))^{\nu}}\,  dr'\,
		\le \, \int_0^{\delta} {(K_\epsilon(r'))^{1-\nu}}\,  dr'\,\eqsim_\nu\,\int_{\epsilon}^{\delta+\epsilon} (r')^{1-\nu}
		\,dr'\,\lesssim_\nu\,\epsilon^{2-\nu}.
	\end{align}
	Combining this with \eqref{Eq155} we obtain that
	\begin{equation}\label{Eq159}
		\int_{-\infty}^0 \frac{J_{\delta,\epsilon}^-(r')}{F_{\delta,\epsilon}^2(r')}\,  dr'\, \lesssim_{\delta,n,\nu}
		\, 1+\epsilon^{2-\nu}.
	\end{equation}
	Next, from \eqref{Eq_Fd}, we deduce that
	\begin{equation}\label{Eq157}
		F_{\delta,\epsilon}(r)\,\le \, \delta^{-l} (K_\epsilon(r))^\frac{\nu}{2},
	\end{equation}
	which together with
	\eqref{Eq_Gde} and \eqref{Eq145} leads to 
	\begin{align}
		G_{\delta,\epsilon}(0)\,\ge\, 
		\delta^{2l}\int_0^\infty \int_{r'}^\infty  
		(K_\epsilon(r))^{-\nu}
		\,dr''\,dr'\,\eqsim_\nu\, 
		\int_0^\infty \int_{r'}^\infty  
		\delta^{2l} (r''+\epsilon)^{-\nu}
		\,dr''\,dr'\,\eqsim_{\delta,n,\nu}\, \epsilon^{2-\nu}.
	\end{align} 
	Since $G_{\delta,\epsilon}(r)$ is decreasing, we conclude
	\[
	G_{\delta,\epsilon}(r)
	\,\ge \, G_{\delta,\epsilon}(0)\,\gtrsim_{\delta,n,\nu} \, \epsilon^{2-\nu}, \quad r\le 0.
	\]
	In combination with \eqref{Eq158} and \eqref{Eq159}, we arrive at \eqref{Eq143}.
\end{proof}
{\cO Lemmas \ref{Lemma_prop_Jd}--\ref{Lemma_Lde} are used to prove the approximate $\log$-entropy estimate Lemma \ref{Lemma_log_Entr_Approx_level}, which, as laid out in Subsection \ref{Sec_strategy},  results in an energy estimate. However, while it is straightforward to bound the suppressed terms in \eqref{Eq9} using the displayed leading order term, conservation of mass and the surface energy itself, the lower order terms on the approximate level are more delicate to estimate. More precisely, in the proof of the $\delta$-uniform energy estimate Lemma \ref{Lemma_EE} we also use the following versions of the trivial inequalities}
\begin{align}&
	u^{n-3}\,\le \, (u^\frac{n}{2})^\frac{1}{2}(u^{\frac{n}{2}-2})^\frac{3}{2},
	\\&
	u^{n-2}\,\le \, u^\frac{n}{2} u^{\frac{n}{2}-2},
\end{align}
{\cO on the regularized level.}
\begin{lemma}\label{Lemma_F_de_prop} 
	It holds that
	\begin{align}&\label{Eq53}
		|((F_{\delta,\epsilon}')^2)'(r)|\,
		\lesssim_{n,\nu} \, (F_{\delta,\epsilon}(r))^\frac{1}{2} (F_{\delta,\epsilon}''(r))^\frac{3}{2},\\&
		\label{Eq54}
		|(F_{\delta,\epsilon}^2)'''(r)|
		\lesssim_{n,\nu} \, (F_{\delta,\epsilon}(r))^\frac{1}{2}
		(F_{\delta,\epsilon}''(r))^\frac{3}{2},
		\\& \label{Eq57}
		(F_{\delta,\epsilon}'(r))^2\, \lesssim_{n,\nu} \, F_{\delta,\epsilon}(r) F_{\delta,\epsilon}''(r),
		\\& \label{Eq58}
		\bigl|(F_{\delta,\epsilon}^2(r))''\bigr|\, \lesssim_{n,\nu} \, F_{\delta,\epsilon}(r) F_{\delta,\epsilon}''(r).
	\end{align}
\end{lemma}
\begin{proof}
	Ad \eqref{Eq53}.
	We have that
	\begin{equation}\label{Eq126}(	F_{\delta,\epsilon}'(r))^2\,=\, (F_\delta'(K_\epsilon(r)))^2(K_{\epsilon}'(r))^2
	\end{equation}
	by the chain rule
	and hence
	\begin{align}\begin{split}\label{Eq52}
			&
			((	F_{\delta,\epsilon}')^2)'(r)\,=\, 2 (F_\delta'(K_\epsilon(r)))^2 K_\epsilon'(r) K_\epsilon''(r)\,+\, 2 F_\delta'(K_\epsilon(r)) F_\delta''(K_\epsilon(r)) (K_\epsilon'(r))^3
			\\&\quad=\, 2F_\delta'(K_\epsilon(r)) K_\epsilon'(r)
			F_{\delta,\epsilon}''(r)
		\end{split}
	\end{align}
	by \eqref{Eq39}.
	Next, we convince ourselves that 
	\begin{equation}\label{Eq51}
		|F_\delta'(r)|^2\, \lesssim_{n,\nu}\,
		F_{\delta}(r) F_\delta''(r), \quad r>0.
	\end{equation}
	Using \eqref{Eq_Fd}, \eqref{Eq34} and \eqref{Eq33} we observe that
	\begin{align}\begin{split}\label{Eq55}
			&
			(F_\delta'(r))^2 \,=\,
			r^{{n+\nu}-2}\tfrac{(nr^\frac{\nu}{2}\,+\, \delta^{l}  \nu r^\frac{n}{2})^2}{4(r^\frac{\nu}{2}+\delta^l  r^\frac{n}{2})^4}\,\lesssim\, 
			r^{{n+\nu}-2}\tfrac{n^2r^\nu\,+\, \delta^{2l}  \nu^2 r^n}{(r^\frac{\nu}{2}+\delta^l  r^\frac{n}{2})^4}, \quad r>0,
			\\& 
			F_\delta(r)F_\delta''(r)\,=\, r^{{n+\nu}-2}\tfrac{\delta^{2l}(\nu-2)\nu r^n \,-\, \delta^l  (\nu^2+\nu(2-4n) +n(n+2)) r^\frac{n+\nu}{2}\,+\, n(n-2)r^\nu}{4(r^\frac{\nu}{2}+\delta^l  r^\frac{n}{2})^4}, \quad r>0,
		\end{split}
	\end{align}
	{\cO and hence} it suffices to show that
	\begin{equation}\label{Eq56}n^2r^\nu\,+\, \delta^{2l}  \nu^2 r^n
		\,\lesssim_{n,\nu}\, \delta^{2l}(\nu-2)\nu r^n \,-\, \delta^l  (\nu^2+\nu(2-4n) +n(n+2)) r^\frac{n+\nu}{2}\,+\, n(n-2)r^\nu, \quad r>0.
	\end{equation}
	However, this follows from the assumption \eqref{Eq106}, {\cO so} that \eqref{Eq51} indeed holds. Together with  \eqref{Eq52} we proceed to estimate
	\begin{align}
		&
		|((	F_{\delta,\epsilon}')^2)'(r)|\, \lesssim_{n,\nu}
		\,(F_{\delta}(K_\epsilon(r)) F_\delta''(K_\epsilon(r)))^\frac{1}{2}  K_\epsilon'(r)
		F_{\delta,\epsilon}''(r)
		\\&\quad =\, (F_{\delta,\epsilon}(r) )^\frac{1}{2}  
		F_{\delta,\epsilon}''(r) \bigl(F_{\delta}''(K_\epsilon(r))(K_\epsilon'(r))^2\bigr)^\frac{1}{2}.
	\end{align}
	By combining \eqref{Eq34}, \eqref{Eq39} and \eqref{Eq45}, we see that
	\begin{equation}\label{Eq128}F_{\delta}''(K_\epsilon(r))(K_\epsilon'(r))^2\,
		\le \,
		F_{\delta,\epsilon}''(r)
	\end{equation}
	and \eqref{Eq53} follows.
	
	Ad \eqref{Eq54}. An algebraic computation shows that
	\begin{align}&
		(F_{\delta,\epsilon}^2)'''(r)\,=\,
		(F_{\delta}^2)'''(K_\epsilon(r))(K_\epsilon'(r))^3\,+\, 3 (F_{\delta}^2)''(K_\epsilon(r))K_\epsilon'(r)K_\epsilon''(r)\,+\, (F_{\delta}^2)'(K_\epsilon(r))K_\epsilon'''(r)
		\\&\quad
		=\, 6 (F_\delta'F_\delta'')(K_\epsilon(r))(K_\epsilon'(r))^3
		\,+\, 
		2(F_\delta F_\delta''')(K_\epsilon(r)) (K_\epsilon'(r))^3
		\,+\, 6 (F_\delta F_\delta'') (K_\epsilon(r)) K_\epsilon'(r)K_\epsilon''(r)
		\\&\qquad+\, 6 ( F_\delta')^2 (K_\epsilon(r)) K_\epsilon'(r)K_\epsilon''(r)
		\,+\, 2(F_\delta'F_\delta)(K_\epsilon(r)) K_\epsilon'''(r)
		\\&\quad
		=\, T_1\,+\,\dots\,+\,  T_5.
	\end{align}
	We notice that
	\begin{equation}\label{Eq119}
		|T_1\,+\, T_4|\,=\, 3| ((	F_{\delta,\epsilon}')^2)'(r)|
		\, 	\lesssim_{n,\nu} \, (F_{\delta,\epsilon}(r))^\frac{1}{2} (F_{\delta,\epsilon}''(r))^\frac{3}{2}
	\end{equation}
	by \eqref{Eq53} and \eqref{Eq52}. For $T_3$, we use
	\eqref{Eq55} to conclude that
	\[
	|F_\delta(r)F_\delta''(r)|\,\lesssim_{n,\nu}\,
	(F_\delta'(r))^2, \quad r>0
	\]
	and thus $|T_3|\lesssim_{n,\nu} |T_4|$. By \eqref{Eq34}, \eqref{Eq33}, \eqref{Eq44}, \eqref{Eq45} and \eqref{Eq56}, we have that $\sgn(T_1) =\sgn(T_4) = \sgn(r)$ {\cO and therefore}
	\begin{equation}\label{Eq122}
		|T_1| \,+\, |T_4|\,=\, |T_1+T_4|,
	\end{equation}
	which was estimated appropriately in \eqref{Eq119}. For $T_2$, we use that
	\begin{align}&
		F_{\delta}'''(r)\,=\, \tfrac{r^{\frac{n+\nu}{2}-3}}{8(r^\frac{\nu}{2}+\delta^l r^\frac{n}{2})^4} \times \Bigl( 
		\delta^{3l} \nu(\nu^2-6\nu+8)r^\frac{3n}{2}\\&\qquad-\, \delta^{2l}(4\nu^3+n(-12\nu^2-n^2 -6n -8) +2\nu(3n^2 +12n-8) )r^\frac{2n+\nu}{2}
		\\&\qquad+\,\delta^l (\nu^3-6\nu^2(n-1)+4\nu(3n^2-6n+2) -4n(n^2-4))r^\frac{n+2\nu}{2} \\&\qquad+\, n(n^2-6n+8)r^\frac{3\nu}{2} \Bigr)
	\end{align}
	{\cO so} that
	\[
	|F_\delta(r)F_\delta'''(r)|
	\,\lesssim_{n,\nu}
	\,r^{{n+\nu}-3} \tfrac{\delta^{3l}r^\frac{3n}{2} \,+\, r^\frac{3\nu}{2}}{(r^\frac{\nu}{2}+\delta^l r^\frac{n}{2})^5},\quad r>0,
	\]
	by Young's inequality. Using \eqref{Eq34},\eqref{Eq33} and \eqref{Eq56} we observe that 
	\begin{align}&F_\delta'(r)F_\delta''(r)\, \gtrsim_{n,\nu}\,
		r^{n+\nu-3}\tfrac{(nr^\frac{\nu}{2}\,+\, \delta^l  \nu r^\frac{n}{2})(n^2r^\nu\,+\, \delta^{2l}  \nu^2 r^n)}{8(r^\frac{\nu}{2} +\delta^l r^\frac{n}{2})^5}\\&\quad \gtrsim_{n,\nu}\,		
		\,r^{{n+\nu}-3} \tfrac{\delta^{3l}r^\frac{3n}{2} \,+\, r^\frac{3\nu}{2}}{(r^\frac{\nu}{2}+\delta^l r^\frac{n}{2})^5},\quad r>0.
	\end{align}
	The bound $|T_2|\lesssim_{n,\nu} |T_1|$ follows, which again is bounded by \eqref{Eq119} and \eqref{Eq122}. It is left to bound $T_5$ and to this end we compute 
	\begin{align}
		K_\epsilon'''(r)\,=\, \frac{-3\epsilon^2 r }{(\epsilon^2+r^2)^\frac{5}{2}}\,=\, -3 \frac{K_\epsilon''(r)K_\epsilon'(r)}{K_\epsilon(r)}
	\end{align}
	using \eqref{Eq44} and \eqref{Eq45}. Hence,
	\[
	|T_5|\,=\, |2(F_\delta'F_\delta)(K_\epsilon(r)) K_\epsilon'''(r)|\,
	\lesssim\,
	\Bigl(\frac{F_\delta'F_\delta}{(\cdot)}\Bigr)(K_\epsilon(r)) K_\epsilon''(r)|K_\epsilon'(r)|
	\]
	and we obtain the estimate $|T_5|\lesssim |T_4|$ as soon as we can show that
	\[
	\frac{F_{\delta}(r)}{r}\,\lesssim_{n,\nu} \, F_{\delta}'(r),\quad r>0.
	\]
	But this follows by
	\begin{align}
		\frac{F_{\delta}(r)}{rF_{\delta}'(r)}\,\overset{\eqref{Eq_Fd}, \eqref{Eq34}}{=} \, 
		\tfrac{2(r^\frac{\nu}{2}+ \delta^l r^\frac{n}{2})^2}{(r^\frac{\nu}{2}+ \delta^l r^\frac{n}{2})(nr^\frac{\nu}{2}+\delta^l \nu r^\frac{n}{2})}\,\lesssim_{n,\nu}\, 1,\quad r>0,
	\end{align}
	{\cO so} that the \eqref{Eq54} is a consequence of \eqref{Eq119} and \eqref{Eq122}.
	
	Ad \eqref{Eq57}. From \eqref{Eq126} and \eqref{Eq51}, we conclude that
	\[
	(F_{\delta,\epsilon}'(r))^2\,\lesssim_{n,\nu}\,
	F_\delta(K_\epsilon(r))F_\delta''(K_\epsilon(r))(K_\epsilon'(r))^2
	\]
	Inserting \eqref{Eq128}, we obtain
	\[
	(F_{\delta,\epsilon}'(r))^2\,\lesssim_{n,\nu}\, F_{\delta,\epsilon}(r)F_{\delta,\epsilon}''(r)
	\]
	as desired.
	
	Ad \eqref{Eq58}. We have that
	\[
	(F_{\delta,\epsilon}^2)''(r)\,=\, 2(F_{\delta,\epsilon}'(r))^2\,+\, 2F_{\delta,\epsilon}(r)F_{\delta,\epsilon}''(r),
	\]
	{\cO and thus} \eqref{Eq58} is a consequence of \eqref{Eq57}.
\end{proof}
{\cO After having provided the technical parts of the $\delta$-uniform $\log$-entropy and energy estimate, we next prove the ingredients of the $(R,\epsilon,\delta)$-uniform entropy estimate Lemma \ref{Lemma_Entropy_Est}. They can be seen as generalizations of \cite[Lemmas 4.1, 4.2]{dareiotis2021nonnegative} to the case of an inhomogeneous mobility function.}

\begin{lemma}\label{Lemma_Hde}
	It holds that
	\begin{equation}\label{Eq149}
		H_{\delta,\epsilon}^2(r)\,\lesssim_{n,\nu}\, G_{\delta,\epsilon}(r).
	\end{equation}
\end{lemma}
\begin{proof}
	We use that $F_\delta$ is increasing on $(0,\infty)$ together with \eqref{Eq145}
	to conclude that
	\begin{align}\begin{split}\label{Eq148}
			&
			H_{\delta,\epsilon}(r)\,\overset{\eqref{Eq_Hde},\eqref{Eq145}}{\le} \, \int_r^\infty \frac{1}{F_\delta\bigl(\tfrac{1}{\sqrt{2}}(r'+\epsilon)\bigr)}\, dr'
			\,\overset{\eqref{Eq_Fd}}{=}\, 
			\int_r^\infty 2^\frac{n}{4}(r'+\epsilon)^{-\frac{n}{2}}\,+\,  2^\frac{\nu}{4}\delta^l (r'+\epsilon)^{-\frac{\nu}{2}}\, dr'
			\\&\quad
			\eqsim_{n,\nu} (r+\epsilon)^{1-\frac{n}{2}}
			\,+\, \delta^l (r+\epsilon)^{1-\frac{\nu}{2}}, \quad r\ge 0.
		\end{split}
	\end{align}
	Next, we use that
	\begin{align}&\label{Eq146}
		F_\delta(r)\,\le \, r^\frac{n}{2},\quad r> 0,\\&
		\label{Eq147}
		F_\delta(r)\, \le \, \delta^{-l}r^\frac{\nu}{2}, \quad r> 0{\cO,}
	\end{align}
	by \eqref{Eq_Fd} to derive
	\begin{align}\begin{split}
			\label{Eq154}&
			G_{\delta,\epsilon}(r)\,\overset{\eqref{Eq_Gde}, \eqref{Eq145}}{\ge} \, \int_r^\infty \int_{r'}^\infty \frac{1}{F_\delta^2(r''+\epsilon)}\, dr''\, dr'
			\\&\quad \overset{\eqref{Eq146}}{\ge} \, \int_r^\infty \int_{r'}^\infty \frac{1}{(r''+\epsilon)^n}\, dr''\, dr'\, \eqsim_{n}\, (r+\epsilon)^{2-n}, \quad r\ge 0{\cO,}
		\end{split}
	\end{align}
	and
	\begin{equation}
		G_{\delta,\epsilon}(r)\, \gtrsim_{\nu}\, \delta^{2l}(r+\epsilon)^{2-\nu}, \quad r\ge 0
	\end{equation}
	using \eqref{Eq147} instead. Inserting these estimates in \eqref{Eq148}, we arrive at \eqref{Eq149} for $r\ge 0$. For $r<0$, we use that $F_{\delta,\epsilon}$ is an even function together with $G_{\delta,\epsilon}$ being decreasing to conclude that
	\[
	H_{\delta,\epsilon}^2(r)\, \overset{\eqref{Eq_Hde}}{\le} \, (2H_{\delta,\epsilon}(0))^2\,
	\lesssim_{n,\nu} \, G_{\delta,\epsilon}(0)\, \le \, G_{\delta,\epsilon}(r) 
	,\quad r<0,
	\]
	finishing the proof.
\end{proof}
\begin{lemma}\label{Lemma_logFde}
	It holds that
	\begin{equation}\label{Eq156}
		|\log(F_{\delta,\epsilon}(r))|\, \lesssim_{{\cO n,\nu}} \, G_{\delta,\epsilon}(r)\,+\, |r|\,+\, 1.
	\end{equation}
\end{lemma}
\begin{proof}
	We distinguish several cases of $r$. Firstly, we assume that $r\ge 0$ and $K_\epsilon(r)\ge \delta${\cO, in which case } we deduce from \eqref{Eq145}, \eqref{Eq46}  and {\cO\eqref{Eq146}} that 
	\begin{align}
		\tfrac{1}{{\cO 2^\frac{n+4}{4}}}(r+\epsilon)^\frac{n}{2}\,\le \, \tfrac{1}{2}K_\epsilon^\frac{n}{2}(r)\, \le \, 
		F_{\delta,\epsilon}(r)\,\le \, K_\epsilon^{\frac{n}{2}}(r) \, \le \, (r+\epsilon)^\frac{n}{2}
		, \quad  r\ge 0\,\wedge\,  K_\epsilon(r)\ge \delta.
	\end{align}
	Thus, we have
	\begin{align}
		\tfrac{n}{2}\log(r+\epsilon) \,-\, {\cO\tfrac{n+4}{4}}\log(2) \, \le \,\log(F_{\delta,\epsilon}(r))\, \le \,  	\tfrac{n}{2}\log(r+\epsilon) 
		, \quad  r\ge 0\,\wedge\,  K_\epsilon(r)\ge \delta
	\end{align}
	and therefore
	\begin{align}\label{Eq153}
		|\log(F_{\delta,\epsilon}(r))|\, \lesssim_n \,|\log(r+\epsilon)| \,+\,1, \quad  r\ge 0\,\wedge\,  K_\epsilon(r)\ge \delta.
	\end{align}
	
	Next, we consider the case that $r\ge 0$ and $K_\epsilon(r)< \delta$ and use 
	\eqref{Eq152}, \eqref{Eq145}, \eqref{Eq47} and {\cO\eqref{Eq146}} to estimate
	\begin{align}\tfrac{1}{{\cO 2^\frac{\nu+4}{4}}}(r+\epsilon)^\frac{\nu}{2}\, \le \, \tfrac{1}{2\delta^l}K_\epsilon^\frac{\nu}{2}(r)\, \le \, F_{\delta,\epsilon}(r)\, \le \,  K_\epsilon^\frac{n}{2}(r)\, \le \,1
		,\quad  r\ge 0\,\wedge\,  K_\epsilon(r)< \delta
	\end{align}
	Hence,
	\begin{align}
		\tfrac{\nu}{2} \log(r+\epsilon)\,-\,{\cO\tfrac{\nu+4}{4}}\log(2)
		\,\le\, 
		\log(F_{\delta,\epsilon}(r))\,\le \, 0 ,\quad  r\ge 0\,\wedge\,  K_\epsilon(r)< \delta.
	\end{align}
	Combining this with \eqref{Eq152}, \eqref{Eq154} and \eqref{Eq153}, we obtain
	\begin{align}&
		|\log(F_{\delta,\epsilon}(r))|\, \lesssim_{n,\nu}\,|\log(r+\epsilon)|\,+\, 1
		\\&\lesssim_n \, (r+\epsilon)^{2-n}\,+\, (r+\epsilon)\,+\, 1\,
		\lesssim_n\, G_{\delta,\epsilon}(r) \,+\, r\,+\, 1
		,\quad r\ge 0.
	\end{align}
	For $r<0$, \eqref{Eq156} follows as well, because $F_{\delta,\epsilon}$ is even and $G_{\delta,\epsilon}$ is decreasing.
\end{proof}
{\cO Finally, we provide the technical ingredients  of the $(R,\epsilon)$-uniform energy estimate Lemma \ref{Lemma_EEeps}, which is proved in Section \ref{Sec_remaining_proofs}. They can be seen as generalizations of \cite[Lemmas 4.4, 4.5]{dareiotis2021nonnegative} to the case of an inhomogeneous mobility function.}
	\begin{lemma}\label{LE_1} It holds that
	\begin{align}&\label{Eq161}
		\int_{-\infty}^{\infty} (F_{\delta,\epsilon}''(r))^2\, dr \,\lesssim_{\delta,n,\nu}\, 1.	
	\end{align}
\end{lemma}
\begin{proof}
	We use that $F_{\delta,\epsilon}$ is an even function together with \eqref{Eq134} and \eqref{Eq133}, to estimate
	\begin{align}&
		\int_{-\infty}^{\infty} (F_{\delta,\epsilon}''(r))^2\, dr\,=\, 2\int_0^\infty 
		(F_{\delta,\epsilon}''(r))^2
		\, dr
		\\&\quad\lesssim_{\delta,n,\nu}\,
		\int_0^1 K_\epsilon^{\nu-4}(r)\, dr\,+\, 
		\int_1^\infty K_\epsilon^{n-4}(r)\, dr
		\\&\quad
		\le \, \int_0^1 r^{\nu-4}\, dr\,+\, 
		\int_1^\infty r^{n-4}\, dr\,\lesssim_{n,\nu}\, 1.
	\end{align}
\end{proof}
\begin{lemma}\label{LE_2}It holds that
	\begin{align}&\label{Eq19}
		|(F_{\delta,\epsilon}^2)'''(r)|\,\lesssim_{\delta,n,\nu}\, 1,
		\\&\label{Eq20}
		|((F_{\delta,\epsilon}')^2)'(r)|\,\lesssim_{\delta,n,\nu}\, 1.
	\end{align}
\end{lemma}
\begin{proof}
	We observe that by \eqref{Eq53} and \eqref{Eq54}, it suffices to show 
	\begin{align}\label{Eq160}
		(F_{\delta,\epsilon}(r))^\frac{1}{2} (F_{\delta,\epsilon}''(r))^\frac{3}{2}\,\lesssim_{\delta,n,\nu}\, 1
	\end{align}
	to conclude \eqref{Eq19} and \eqref{Eq20}. To this end, we combine \eqref{Eq134} and \eqref{Eq113} to conclude that
	\begin{align}
		(F_{\delta,\epsilon}(r))^\frac{1}{2} (F_{\delta,\epsilon}''(r))^\frac{3}{2}\,\lesssim_{n,\nu} \, K_\epsilon^{n-3}(r)\,\le\, |r|^{n-3}\,\le \, 1,\quad |r|\ge 1.
	\end{align}
	On the other hand, \eqref{Eq133} and \eqref{Eq157}   yield that 
	\begin{align}
		(F_{\delta,\epsilon}(r))^\frac{1}{2} (F_{\delta,\epsilon}''(r))^\frac{3}{2}\,\lesssim_{\delta,n,\nu}\, K_\epsilon^{\nu-3}(r)\,\overset{\eqref{Eq152}}{\le } \,(1^2\,+\, 1^2)^\frac{\nu-3}{2}\,\lesssim_{\nu}\,1,\quad |r|< 1. 
	\end{align}
	Consequently, \eqref{Eq160} holds and the proof is complete.
\end{proof}
\begin{lemma}\label{LE_3}It holds that
	\begin{align}&
		|F_{\delta,\epsilon}'(r)|^2\,\lesssim_{n,\nu} \, |r|^{n-2}\,+\,1, \label{Eq21}\\&
		|(F_{\delta,\epsilon}^2)'' (r)|\,\lesssim_{n,\nu} \, |r|^{n-2}\,+\,1, \label{Eq22} \\&
		|(F_{\delta,\epsilon}^2(r))|\,\lesssim_n\, |r|^n\,+\,1. \label{Eq23}
	\end{align}
\end{lemma}
\begin{proof}As a consequence of \eqref{Eq57} and \eqref{Eq58}, it suffices to verify
	\begin{equation}
		F_{\delta,\epsilon}(r) F_{\delta,\epsilon}''(r)\,\lesssim_{n,\nu}\, |r|^{n-2}\,+\, 1
	\end{equation}
	to conclude \eqref{Eq21} and \eqref{Eq22}.
	Combining \eqref{Eq134} and \eqref{Eq113}, we estimate 
	\begin{align}
		F_{\delta,\epsilon}(r) F_{\delta,\epsilon}''(r)\,\lesssim_n\, K_\epsilon^{n-2}(r)\,\le\, 
		|r|^{n-2}\,+\, \epsilon^{n-2}
		\,\overset{\eqref{Eq152}}{\le}\,|r|^{n-2}\,+\, 1,
	\end{align}
	as desired.
	The remaining \eqref{Eq23} follows from \eqref{Eq113}.
\end{proof}
{\color{OliveGreen}\section{It\^o expansion of the energy functional}\label{App_B}
In this second appendix we recall the  It\^o expansion of the energy of a martingale solution to \eqref{Eq108} in the sense of Definition \ref{defi_sol_non_degenerate} and its proof, as provided in \cite[Lemma 4.6]{dareiotis2021nonnegative}, with some additional details.
\begin{lemma}\label{Lemma_App_B}
	Any martingale solution
	\[
	\bigl\{
	(\hat{\Omega}, \hat{\mathfrak{A}}, \hat{\mathfrak{F}}, \hat{\PP}), \,(\hat{\beta}^{(k)})_{k\in \ZZ}, \uderh 
	\bigr\}
	\]
	to \eqref{Eq108} in the sense of Definition \ref{defi_sol_non_degenerate}  satisfies
	\begin{align}\begin{split}\label{Eq25b}
			&
			\tfrac{1}{2}\|\partial_x \uderh(t)\|_{L^2(\TT)}^2\,=\, 
			\tfrac{1}{2}\|\partial_x \uderh(0)\|_{L^2(\TT)}^2\,-\, \int_0^t \int_{\TT} F_{\delta,\epsilon}^2(\uderh)(\partial_x^3 \uderh)^2\,dx\, ds
			\\&\qquad +\, \tfrac{1}{2} \sum_{k\in \ZZ}\int_0^t \gamma_{\uderh}^2\int_{\TT}
			\sigma_k^2 (F_{\delta,\epsilon}''(\uderh))^2 (\partial_x \uderh)^4\, dx\, ds\\&\qquad
			+\,\tfrac{1}{16}\sum_{k\in \ZZ} \int_0^t \gamma_{\uderh}^2 \int_{\TT} (\partial_x (\sigma_k^2)) \bigl(
			(F_{\delta,\epsilon}^2)'''(\uderh)\,+\, 4 ((F_{\delta,\epsilon}')^2)'(\uderh)
			\bigr) (\partial_x \uderh)^3\,dx\,ds \\&\qquad+\,
			\tfrac{3}{2}
			\sum_{k\in \ZZ} \int_0^t \gamma_{\uderh}^2 \int_{\TT} ((\partial_x\sigma_k)^2 - \sigma_k (\partial_x^2\sigma_k)) (F_{\delta,\epsilon}'(\uderh))^2  (\partial_x \uderh)^2\,dx\,ds
			\\&\qquad+\,  			
			\tfrac{3}{16}\sum_{k\in \ZZ} \int_0^t \gamma_{\uderh}^2 \int_{\TT}
			(\partial_x^2(\sigma_k^2)) (F_{\delta,\epsilon}^2)''(\uderh)(\partial_x \uderh)^2
			\,dx\,ds
			\\&\qquad+\,
			\tfrac{1}{8}\sum_{k\in \ZZ}\int_0^t  \gamma_{\uderh}^2 \int_{\TT} (4\sigma_k\partial_x^4 \sigma_k -\partial_x^4(\sigma_k^2)) F_{\delta,\epsilon}^2(\uderh)\,dx\, ds \\&\qquad+\, \sum_{k\in \ZZ}\int_0^t  \gamma_{\uderh}
			\int_{\TT} 
			\sigma_k F_{\delta,\epsilon}(\uderh)\partial_x^3 \uderh
			\,dx
			\, d\hat{\beta}^{(k)},
		\end{split}
	\end{align}
	for all $t\in [0,T]$ with $\gamma_{\uderh} = g_R(\|\uderh\|_{C(\TT)})$.
\end{lemma}
\begin{proof}
	As pointed out at the beginning of this appendix, we argue as in \cite[Lemma 4.6]{dareiotis2021nonnegative}.
	It\^o's formula for the functional $\|\partial_x \cdot\|_{L^2(\TT)}^2$ can be justified, e.g., by applying \cite[Theorem 4.2.5]{liurock} on the Gelfand triple of homogeneous Sobolev spaces $\dot{H}^3(\TT)\subset \dot{H}^1(\TT)\subset \dot{H}^{-1}(\TT)$ as elaborated in \cite[Appendix C]{Sauerbrey_2021}. The required regularity requirements follow from  
	\begin{align}
		&
		\uderh \in L^2(\hat{\Omega}; C([0,T]; H^1(\TT))) \cap 
		 L^2(\hat{\Omega}\times [0,T]; H^3(\TT)),
	\end{align}
	as stated in Definition \ref{defi_sol_non_degenerate} together with a stopping time argument. Thus, employing  the conventions $u=\uderh$ and $\gamma_{u} = g_R(\|\uderh\|_{C(\TT)})$, we obtain that
	\begin{align}\begin{split}
			\label{Eq4}
		\frac{1}{2} \|\partial_x u(t)\|_{L^2(\TT)}^2 \,=&\, 
	\frac{1}{2} \|\partial_x u(0)\|_{L^2(\TT)}^2 \,-\, 
	\int_0^t   
	\int_{\TT} 
	F_{\delta,\epsilon}^2(u) (\partial_x^3 u)^2
	\, dx
	\, ds
	\\& +\, \frac{1}{2} \sum_{k\in \ZZ} \int_0^t \gamma_{u}^2 \int_{\TT} 
	(\partial_x u) \partial_x^2 (\sigma_k F_{\delta,\epsilon}'(u) \partial_x( \sigma_k F_{\delta,\epsilon}(u)))
	\,
	dx\, ds
	\\&+\,
	\frac{1}{2} \sum_{k\in \ZZ} \int_0^t \gamma_{u}^2 \int_{\TT} 
\bigl( \partial_x^2( \sigma_k F_{\delta,\epsilon}(u))\bigr)^2
	\,
	dx\, ds
	\\&
	+\sum_{k\in \ZZ} \int_0^t \gamma_{u}
	\int_{\TT} (\partial_x u ) \partial_x^2(\sigma_k F_{\delta,\epsilon}(u)) \, dx\, d\hat{\beta}^{(k)} \,=\, \mathcal{I}_1\,+\, \dots\,+\, \mathcal{I}_5.
		\end{split}
	\end{align}
	We already observe by integrating by parts two times in  $\mathcal{I}_5$, that $\mathcal{I}_1+\mathcal{I}_2+\mathcal{I}_5$ appears on the right-hand side of \eqref{Eq25b}. Hence, it suffices to show that $\mathcal{I}_3 +\mathcal{I}_4$ equals the remaining terms from \eqref{Eq25b} for which we follow \cite[Lemma 3.1]{dareiotis2021nonnegative}. 
	
	Firstly, we observe that
	\begin{align}&
		\int_{\TT} 
		(\partial_x u) \partial_x^2 (\sigma_k F_{\delta,\epsilon}'(u) \partial_x( \sigma_k F_{\delta,\epsilon}(u)))
		\,
		dx
		\\&\quad =\, -
		\int_{\TT} 
		(\partial_x^2 u) \partial_x \Bigl(\sigma_k^2 (\partial_x u)\bigl(F_{\delta,\epsilon}'(u)\bigr)^2   \,+\,
		\sigma_k (\partial_x \sigma_k) \bigl(F_{\delta,\epsilon}'(u) F_{\delta,\epsilon}(u)\bigr)
		\Bigr)
		\,
		dx
		\\& \quad =\,  -\int_{\TT} 
		\sigma_k^2 (\partial_x^2 u)^2
		(F_{\delta,\epsilon}')^2(u)\,+\, \sigma_k^2 (\partial_x^2 u) (\partial_x u)^2 \bigl((F_{\delta,\epsilon}')^2\bigr)'(u)
		\,+\, (\partial_x (\sigma_k^2)) (\partial_x^2 u) (\partial_x u )
		(F_{\delta,\epsilon}')^2(u)
		\, dx
		\\&
		\qquad -
		\frac{1}{4}\int_{\TT}
		( \partial_x (\sigma_k)^2) (\partial_x^2 u )(\partial_x u)
		(\Fde^2)'' (u) \,+\, (\partial_x^2 (\sigma_k^2) ) (\partial_x^2 u )
		(\Fde^2)' (u)
		\, dx
		\\&\quad =\,
		- \,\int_{\TT} \sigma_k^2 (\partial_x^2 u)^2
		(F_{\delta,\epsilon}')^2(u)\, dx \,+\, \frac{1}{3}\int_{\TT}  \sigma_k^2 (\partial_x u)^4 \bigl((F_{\delta,\epsilon}')^2\bigr)''(u)
		\,+\,  (\partial_x (\sigma_k)^2)  (\partial_x u)^3 \bigl((F_{\delta,\epsilon}')^2\bigr)'(u)
		\, dx
		\\&\qquad +\frac{1}{2}\int_{\TT} 
		(\partial_x (\sigma_k^2))(\partial_x u)^3 \bigl((F_{\delta,\epsilon}')^2\bigr)'(u) \,+\, (\partial_x^2 (\sigma_k^2))(\partial_x u)^2 (F_{\delta,\epsilon}')^2(u)
		\,
		dx
			\\&\qquad +\frac{1}{8} \int_{\TT} 
		(\partial_x (\sigma_k^2))(\partial_x u)^3 (F_{\delta,\epsilon}^2)'''(u) \,+\, (\partial_x^2 (\sigma_k^2))(\partial_x u)^2 (F_{\delta,\epsilon}^2)''(u)
		\,
		dx
		\\& \qquad +  \frac{1}{4}\int_{\TT} (\partial_x ^2(\sigma_k^2)) (\partial_x u)^2 (\Fde^2)''(u) \, dx
		\,-\, \frac{1}{4} \int_{\TT} 
		(\partial_x^4 (\sigma_k^2)) F_{\delta,\epsilon}^2(u) 
		\, dx,
	\end{align}
	by applying the product rule and integrating by parts.	The latter is justified $\check{\PP}\otimes dt$-almost everywhere, where $u\in H^3(\TT)$, since $F_{\delta,\epsilon}\colon \RR\to \RR$ is smooth. After rearranging all these terms, we arrive at
	\begin{align}\begin{split}\label{Eq2}
			&
			\int_{\TT} 
		(\partial_x u) \partial_x^2 (\sigma_k F_{\delta,\epsilon}'(u) \partial_x( \sigma_k F_{\delta,\epsilon}(u)))
		\,
		dx
		\\&\qquad =
		\int_{\TT} 
		\sigma_k^2 \Bigl(
		 - 
		(F_{\delta,\epsilon}')^2(u)(\partial_x^2 u)^2\,+\, \tfrac{1}{3}  \bigl((F_{\delta,\epsilon}')^2\bigr)''(u)(\partial_x u)^4
		\Bigr)
		\, dx
		\\&\qquad+\int_{\TT}
		(\partial_x (\sigma_k^2)) \Bigl(
		\tfrac{1}{8} (F_{\delta,\epsilon}^2)'''(u)\,+\, \tfrac{5}{6}\bigl((F_{\delta,\epsilon}')^2\bigr)'(u)
		\Bigr) (\partial_x u)^3
		\,dx
		\\&\qquad+\int_{\TT}(
		\partial_x^2(\sigma_k^2)) \Bigl(
		\tfrac{1}{2} (F_{\delta,\epsilon}')^2(u)\,+\, \tfrac{3}{8}  (\Fde^2)''(u)
		\Bigr)(\partial_x u)^2
		dx
		\\&\qquad
		-\frac{1}{4}\int_{\TT}
		(\partial_x^4 (\sigma_k^2)) F_{\delta,\epsilon}^2(u)\,
		dx.
		\end{split}
	\end{align}

	Secondly, we use that 
	\begin{align}\begin{split}\label{Eq5}&
		\bigl(\partial_x^2 (\sigma_k \Fde (u))\bigr)^2 
		\\&\quad  =\, 
		\bigl(
		(\partial_x^2 \sigma_k) F_{\delta,\epsilon}(u) \,+\, 2(\partial_x \sigma_k)(\partial_x u)\Fde'(u)
		\,+\, \sigma_k  (\partial_x u)^2 \Fde''(u)
		\,+\, \sigma_k (\partial_x^2 u)\Fde'(u)
		\bigr)^2
		\\&\quad =\, 
		(\partial_x^2 \sigma_k)^2\Fde^2(u) \,+\, 4(\partial_x\sigma_k)^2(\partial_x u)^2 (\Fde')^2(u) 
		\,+\, \sigma_k^2 (\partial_xu)^4(\Fde''(u))^2 
		\,+\, \sigma_k^2 (\partial_x^2 u)^2(\Fde'(u))^2 
		\\&\qquad +\, 2\Bigl(
		\tfrac{1}{2}(\partial_x (\partial_x \sigma_k)^2) (\partial_x u)(\Fde^2)'(u)\,+\, (\partial_x^2 \sigma_k)\sigma_k  (\partial_x u)^2 (\Fde\Fde'')(u)
		\,+\, \tfrac{1}{2} (\partial_x^2 \sigma_k)\sigma_k (\partial_x^2 u)(\Fde^2)'(u)
		\Bigr)
		\\&\qquad +\, 2\Bigl(
		\tfrac{1}{2}(\partial_x (\sigma_k^2)) (\partial_x u)^3\bigl((\Fde')^2\bigr)'(u)\,+\, \tfrac{1}{2} (\partial_x (\sigma_k^2)) (\partial_x (\partial_x u)^2 )(\Fde')^2(u)
		\Bigr)
		\\&\qquad +\, 2 \Bigl(\tfrac{1}{6}\sigma_k^2
		 \partial_x (\partial_x u)^3 \bigl((\Fde')^2\bigr)'(u)
		 \Bigr) . 
		 \end{split}
	\end{align}
	Integrating the first line of the right-hand side yields
	\begin{align}
		\int_{\TT} (\partial_x^2\sigma_k)^2 \Fde^2(u) \,+\, 4(\partial_x\sigma_k)^2(\partial_x u)^2 (\Fde')^2(u) 
		\,+\, \sigma_k^2 (\partial_xu)^4(\Fde''(u))^2 
		\,+\, \sigma_k^2 (\partial_x^2 u)^2(\Fde'(u))^2 \, dx.
	\end{align}
	For the second line of the right-hand side of \eqref{Eq5}, we deduce 
		\begin{align}&
		\int_{\TT}
		2\Bigl(
		\tfrac{1}{2}(\partial_x (\partial_x \sigma_k)^2) (\partial_x u)(\Fde^2)'(u)\,+\, (\partial_x^2 \sigma_k)\sigma_k  (\partial_x u)^2 (\Fde\Fde'')(u)
		\,+\, \tfrac{1}{2} (\partial_x^2 \sigma_k)\sigma_k (\partial_x^2 u)(\Fde^2)'(u)
		\Bigr)
		\, dx
		\\&\quad =\, \int_{\TT} 
		\bigl(\tfrac{1}{2}\partial_x (\partial_x \sigma_k)^2\, -\,( \partial_x^3 \sigma_k )\sigma_k\bigr) (\partial_x u)(\Fde^2)'(u)
	 \,-\, 
		(\partial_x^2 \sigma_k)\sigma_k (\partial_x u)^2 (\Fde^2)''(u)
		\, dx
		\\&\qquad  +\, \int_{\TT} 2 (\partial_x^2 \sigma_k)\sigma_k  (\partial_x u)^2 (\Fde\Fde'')(u)  \, dx
		\\&\quad =\, -\int_{\TT}\bigl((\partial_x^2 \sigma_k )^2\,-\, (\partial_x^4 \sigma_k) \sigma_k\bigr)   \Fde^2(u)\, dx - 2\int_{\TT} 
		(\partial_x^2 \sigma_k)\sigma_k (\partial_x u)^2 (\Fde')^2(u)
		\, dx,
	\end{align}
	by integration by parts.
	For the third line of the right-hand side of \eqref{Eq5}, we obtain instead that
	\begin{align}&
		\int_{\TT}  2\Bigl(
		\tfrac{1}{2}(\partial_x (\sigma_k^2)) (\partial_x u)^3\bigl((\Fde')^2\bigr)'(u)\,+\, \tfrac{1}{2} (\partial_x (\sigma_k^2)) (\partial_x (\partial_x u)^2 )(\Fde')^2(u)
		\Bigr)\, dx
		\\& \quad =\, -\int_{\TT}
		(\partial_x^2 (\sigma_k^2)) (\partial_x u)^2 (\Fde')^2(u)
		\, dx
		\\&
		\quad =\,-\,2\int_{\TT}
		\bigl((\partial_x \sigma_k)^2 \,+\,  (\partial_x^2 \sigma_k)\sigma_k \bigr) (\partial_x u)^2 (\Fde')^2(u)
		\, dx.
	\end{align}
	Integration by parts in the last line of \eqref{Eq5} yields
	\begin{align}&
		\int_{\TT}
		2 \Bigl(\tfrac{1}{6}\sigma_k^2
		\partial_x (\partial_x u)^3 \bigl((\Fde')^2\bigr)'(u)
		\Bigr) \, dx \\&\quad =\, - \frac{1}{3}\int_{\TT} \sigma_k^2 (\partial_x u)^4 \bigl((\Fde')^2\bigr)''(u)\,+\, 
		 (\partial_x (\sigma_k^2)) (\partial_x u)^3 \bigl((\Fde')^2\bigr)'(u)
		\, dx.
	\end{align}
	Summing everything up and rearranging, results in 
		\begin{align}\begin{split}\label{Eq3}
				&
		\int_{\TT}
		\bigl(\partial_x^2 (\sigma_k \Fde (u))\bigr)^2 
		\, dx
		\\&\quad  =\, \int_{\TT}
		\sigma_k^2 \Bigl(  (\Fde'(u))^2(\partial_x^2 u)^2  \,+\, 
		(\Fde''(u))^2(\partial_xu)^4 
		\,-\, \tfrac{1}{3} \bigl((\Fde')^2\bigr)''(u)(\partial_x u)^4
		\Bigr)
		\, dx
		\\&\qquad 
		-\, \frac{1}{3}\int_{\TT} (\partial_x (\sigma_k^2))  \bigl((\Fde')^2\bigr)'(u)(\partial_x u)^3\, dx
		\\&\qquad +2 \int_{\TT} \bigl((\partial_x \sigma_k )^2 -2(\partial_x^2 \sigma_k)\sigma_k \bigr)
		 (\Fde')^2(u)(\partial_x u)^2
		\, dx
		\\&\qquad 
		+\, \int_{\TT} (\partial_x^4 \sigma_k)  \sigma_k \Fde^2(u)\, dx.
			\end{split}
	\end{align}
	Finally, we obtain the desired identity \eqref{Eq25b} by inserting \eqref{Eq2} and \eqref{Eq3} in \eqref{Eq4}.
\end{proof}

}

	\section*{Acknowledgements}The author thanks the anonymous referees for their detailed suggestions to improve this manuscript.
	He thanks his doctoral advisor Manuel V. Gnann for regular discussions on the
	subject as well as a careful reading of this manuscript. He also thanks Konstantinos Dareiotis and Benjamin Gess for discussing the main idea of this article.

\end{document}